\documentclass[a4paper,12pt,twoside]{amsart}
\usepackage{amsmath}
\usepackage{amsthm}
\usepackage{amssymb}
\theoremstyle{plain}
\newtheorem{thm}{\indent\bf Theorem}[section]
\newtheorem{lem}[thm]{\indent\bf Lemma}
\newtheorem{prop}[thm]{\indent\bf Proposition}

\theoremstyle{definition}
\newtheorem{rem}{\indent\it Remark}[section]

\numberwithin{equation}{section}

\def \d {\mathrm{d}}
\def \A {\mathbb{A}}
\def \C {\mathbb{C}}
\def \R {\mathbb{R}}
\def \Q {\mathbb{Q}}
\def \Z {\mathbb{Z}}
\def \N {\mathbb{N}}
\def \re {\mathrm{Re\,}}
\def \im {\mathrm{Im\,}}
\def \diag {\mathrm{diag}}
\def \dist {\mathrm{dist}}
\begin{document}
\title[The fifth Painlev\'e equation]{Critical behaviours of the 
fifth Painlev\'e transcendents and the monodromy data} 
\author[Shun Shimomura]{Shun Shimomura} 
\address{Department of Mathematics, 
Keio University, 
3-14-1, Hiyoshi, Kohoku-ku,
Yokohama 223-8522 
Japan}
\date{}
\maketitle
\begin{abstract}
For the fifth Painlev\'e equation, we present families of convergent series 
solutions near the origin and the corresponding
monodromy data for the associated isomonodromy linear system. 
These solutions are of complex power type, of inverse logarithmic type
and of Taylor series type. It is also possible to compute the monodromy data 
in non-generic cases. 
Solutions of logarithmic type are
derived from those of inverse logarithmic type through
a B\"{a}cklund transformation found by Gromak.
In a special case the complex power type of solutions
have relatively simple oscillatory expressions. 
For the complex power type of solutions in the generic case, we clarify the
structure of the analytic continuation on the universal covering around the
origin, and examine the distribution of zeros, poles and $1$-points. It is
shown that two kinds of spiral domains including a sector as a special
case are alternately arrayed; the domains
of one kind contain sequences both of zeros and of poles, and those of
the other kind sequences of $1$-points.
\end{abstract}
\allowdisplaybreaks
\section{Introduction}\label{sc1}
For the solutions of the sixth Painlev\'e equation
Guzzetti \cite{G-Table} provided the tables of their critical behaviours 
and parametric connection formulas. The solutions near each critical point
are classified as follows: complex power type, logarithmic type, 
inverse oscillatory type, inverse logarithmic type, and Taylor series
type. All of them are derived through birational transformations from four
basic solutions, two of which are of complex power type and two are 
of logarithmic type.
These tables consist of important formulas as nonlinear special functions
and are expected to be of great use in applications to a variety of problems
in mathematics and mathematical physics.
\par
The fifth Painlev\'e equation normalised in the form 
\begin{align*}
\tag*{(V)}
&\frac{d^2y}{d x^2}=  \Bigl(\frac 1{2y} + \frac 1{y-1}
 \Bigr) \Bigl(\frac{d y}{d x} \Bigr)^{\!\! 2}
- \frac 1x \frac{d y}{d x}
\\[0.2cm]
&  +\frac{(y-1)^2}{8x^2} \Bigl((\theta_0-\theta_x+\theta_{\infty}
)^2 y - \frac{(\theta_0 -\theta_x - \theta_{\infty})^2 }
{y}\Bigr) + (1-\theta_0-\theta_x) \frac{y}{x} -\frac{y(y+1)}{2(y-1)}
\end{align*}
with $\theta_0, \theta_x, \theta_{\infty} \in \C$ follows from the isomonodromy
deformation of a $2$-dimensional linear system of the form
\begin{equation}
\label{1.1}
\frac{dY}{d\lambda} =\Bigl( \frac{A_0(x)}{\lambda} +\frac{A_x(x)}{\lambda -
x} + \frac{J} 2 \Bigr) Y,
\end{equation}
$J=\diag [1, -1],$ where $A_0(x)$ and $A_x(x)$ satisfy the following:
\par
(a) the eigenvalues of $A_0(x)$ and $A_x(x)$ are $\pm \theta_0/2$ and
$\pm \theta_x/2,$ respectively;
\par
(b) $(A_0(x) +A_x(x))_{11} = -(A_0(x) + A_x(x))_{22} \equiv -\theta_{\infty}/2$
\par\noindent
(cf. \cite[\S 3]{Jimbo}, \cite[Appendix C]{JM}). Jimbo \cite{Jimbo} studied 
the asymptotic behaviour of the $\tau$-function for (V)
near $x=0$ parametrised by its corresponding monodromy data for \eqref{1.1}. 
Applying WKB analysis to \eqref{1.1}, Andreev and Kitaev
\cite{Andreev-Kitaev} obtained asymptotic solutions of (V) along the real axis
near $x=0$ and
$x=\infty$ together with their monodromy data that yield the connection formulas
between them. For more general integration constants, a family of solutions
near $x=0$ expanded into convergent series in spiral domains or sectors 
was given by the present author
\cite{S}. Kaneko and Ohyama \cite{K-O} presented certain Taylor series
solutions around $x=0$ such that each corresponding linear system \eqref{1.1} is 
solvable in terms of
hypergeometric functions and that the monodromy may be explicitly calculated. 
\par
For the fifth Painlev\'e transcendents as well, it
is preferable to give tables of critical behaviours
like those of Guzzetti \cite{G-Table}.
Toward this goal, in this paper, near the critical point $x=0$ of (V), we 
present families of solutions 
expanded into convergent series of three types and the respective
monodromy data parametrised by integration constants yielding analogues of
the parametric connection formulas in \cite{G-Table}.
The monodromy data are given under more general conditions on $\theta_0,$
$\theta_x,$ $\theta_{\infty}$ and the integration constants than those of 
\cite{Jimbo} or \cite{Andreev-Kitaev}.
These solutions, including degenerate cases, are of complex power type, 
of inverse logarithmic type and of Taylor series type (cf. Theorems
\ref{thm2.1}, and \ref{thm2.2} through \ref{thm2.4}). 
All of them are derived by a unified
method. The key is finding suitable matrix solutions of the Schlesinger
equation equivalent to (V) controlling the isomonodromy deformation of
\eqref{1.1}. 
Solutions of logarithmic type \cite{S-Log} are also obtained from those of
inverse logarithmic type through a B\"{a}cklund transformation 
found by Gromak \cite{Gromak} (cf. Remark \ref{rem2.6} and Section \ref{sc6}).  
In such a sense these inverse logarithmic solutions play the same role
as that of the basic 
logarithmic solutions of the sixth Painlev\'e equation. 
If $\theta_0-\theta_x=\theta_{\infty}=0,$ the complex power type of solutions 
have relatively simple oscillatory expressions although, 
in general, such expressions are complicated 
(cf. Theorem \ref{thm2.1.a} and Remark \ref{rem2.3.a}).
For the complex power type of solutions in the generic case, 
we clarify the structure of the analytic continuation on the 
universal covering around $x=0,$ and examine the distribution of zeros, poles
and $1$-points (note that $y=0,$ $1$ and $\infty$ are singular values of
equation (V)). It is shown that the domains where each solution behaves
like a complex power are separated by two kinds of spiral domains
including a sector as a special case, which are
alternately arrayed; the separating domains of one kind contain 
sequences both of zeros and of poles, and those of the other kind 
sequences of $1$-points (cf. Remark \ref{rem2.8}). 
This situation is different from that of the sixth
Painlev\'e transcendents, in which the separating domain with zeros 
and that with poles 
alternately appear (cf. \cite{G-Critical}, \cite{G-Elliptic}, \cite{G-Pole},
\cite{G-Rev}, \cite{S2}). 
\par
All the results above are described in Section \ref{sc2}. In Section \ref{sc5}
the families of solutions of three types are derived from the matrix solutions
of the Schlesinger equation given in Section \ref{sc3} by using the lemmas in
Section \ref{sc4}. In Sections \ref{sc6} and \ref{sc7} we prove the result on
the analytic continuation, those on the distribution of zeros, poles and 
$1$-points, and Theorem \ref{thm2.1.a} on the special oscillatory expressions. 
In proving them, the B\"{a}cklund transformation referred to
above is crucial. Section \ref{sc8} is devoted to
the summary of the argument in \cite[\S 2]{Jimbo} concerning limiting 
procedure applied to \eqref{1.1} and its monodromy data, and Section \ref{sc9} 
to the computation of the
connection formulas for the related Whittaker and hypergeometric systems. 
In non-generic cases of these linear systems it is also possible to compute 
the monodromy data 
(cf. Remark \ref{rem2.11} and Section \ref{ssc9.3}).
In the final section, 
using the material above, we derive the results on the monodromy data for 
our solutions.  
\par
Throughout this paper we use the following symbols:
\par
(1) for a ring $\A$, $M_2(\A)$ is the ring of $2\times 2$ matrices whose
entries are in $\A$, and $GL_2(\A):=\{C\in M_2(\A);\, C^{-1}\in M_2(\A) \};$
\par
(2) $I,$ $J,$ $\Delta,$ $\Delta_-$ denote the matrices
$$
I= \begin{pmatrix}
1 & 0 \\  0 & 1 \\
\end{pmatrix},
\quad
J= \begin{pmatrix}
1 & 0 \\  0 & -1 \\
\end{pmatrix},
\quad
\Delta = \begin{pmatrix}
0 & 1 \\  0 & 0 \\
\end{pmatrix},
\quad
\Delta_- = \begin{pmatrix}
0 & 0 \\  1 & 0 \\
\end{pmatrix};
$$
\par
(3) $\mathcal{R}(\C\setminus \{0\} )$ denotes the universal covering of
$\C \setminus \{0 \}$;
\par
(4) $\Q_{\theta}:= \Q[\theta_0, \theta_x, \theta_{\infty}];$
\par
(5) for $A, B \subset \C$, $\mathrm{cl}(A)$ denotes the closure of $A$, 
$\dist(A,B)$ the distance between $A$ and $B$;
\par
(6) $\psi(x):=\Gamma'(x)/\Gamma(x)$ is the di-Gamma function.
\section{Main results}\label{sc2}
\subsection{Solutions near $x=0$}\label{ssc2.1}
Set $\Q_{\theta}:= \Q[\theta_0, \theta_x, \theta_{\infty}].$ For each
$(\theta_0, \theta_x,\theta_{\infty})\in \C^3$, we have solutions of complex
power type near $x=0$.
\begin{thm}\label{thm2.1}
Let $\Sigma_0$ be a bounded domain satisfying
$$
\Sigma_0 \subset \C \setminus \Sigma_* \quad \text{with}
\quad \Sigma_*= \{\sigma \le -1 \} \cup \{0 \} \cup \{\sigma \ge 1 \}
\subset \R
$$ 
and $\dist (\Sigma_0, \Sigma_*)>0.$ Suppose that $(\sigma^2-(\theta_0 \pm
\theta_x)^2 )(\sigma^2-\theta_{\infty}^2) \not= 0$ for every
$\sigma\in\mathrm{cl}(\Sigma_0).$ Then $(\mathrm{V})$ admits a two-parameter 
family of solutions
$
\{ y(\sigma, \rho, x); \, (\sigma, \rho)\in \Sigma_0 \times (\C\setminus
\{0\} ) \}
$
with the following properties.
\par
$(\mathrm{i})$ $y(\sigma, \rho, x)$ is holomorphic in $(\sigma,\rho,x) \in
\Omega^+(\Sigma_0,\varepsilon_0) \cup \Omega^-(\Sigma_0, \varepsilon_0)$
$\subset \Sigma_0 \times (\C \setminus \{0\}) \times \mathcal{R}(\C\setminus
\{0\}),$ where
\begin{align*}
& \Omega^{\pm}(\Sigma_0, \varepsilon_0):= \bigcup_{(\sigma,\rho) \in 
\Sigma_0 \times (\C\setminus \{0\})} \{(\sigma,\rho) \}\times \Omega^{\pm}
_{\sigma,\rho}(\varepsilon_0),
\\
& \Omega^{+}_{\sigma,\rho}(\varepsilon_0):=  
\{ x\in \mathcal{R}(\C\setminus\{ 0\}); \,\, 
|\rho x^{\sigma}| < \varepsilon_0, \,\, 
|x (\rho x^{\sigma})^{-1}| < \varepsilon_0 \}, 
\\
& \Omega^{-}_{\sigma,\rho}(\varepsilon_0):=  
\{ x\in \mathcal{R}(\C\setminus\{ 0\}); \,\, 
|x(\rho x^{\sigma})| < \varepsilon_0, \,\, 
| (\rho x^{\sigma})^{-1}| < \varepsilon_0 \}, 
\end{align*} 
$\varepsilon_0=\varepsilon_0(\Sigma_0, \theta_0, \theta_x, \theta_{\infty})$
being a sufficiently small number depending only on $\Sigma_0$ and $(\theta_0,
\theta_x, \theta_{\infty}).$
\par
$(\mathrm{ii})$ $y(\sigma, \rho, x)$ is represented by the convergent series
as follows$:$
\begin{align*}
y_+(\sigma, \rho, x):= &
 1+ \frac{4\sigma^2 (\theta_0  +\theta_x -\sigma)} {(\sigma+\theta_{\infty}
)(\theta_0^2 -(\sigma+ \theta_x)^2 )} \rho x^{\sigma} 
\\
& +\sum_{j\ge 2} c^+_j
(\sigma) (\rho x^{\sigma})^j + \sum_{n=1}^{\infty} x^n \sum_{j\ge 0} c^+_{jn}
(\sigma) (\rho x^{\sigma})^{-n+j}
\end{align*}
in $\Omega^+(\Sigma_0,\varepsilon_0),$ and
\begin{align*}
y_-(\sigma, \rho, x):= &
 1- \frac{4\sigma^2 (\theta_0  +\theta_x +\sigma)} {(\sigma-\theta_{\infty}
)(\theta_0^2 -(\sigma- \theta_x)^2 )}( \rho x^{\sigma})^{-1}
\\
& +\sum_{j\ge 2} c^-_j
(\sigma)(\rho x^{\sigma})^{-j} + \sum_{n=1}^{\infty} x^n \sum_{j\ge 0} c^-_{jn}
(\sigma) (\rho x^{\sigma})^{n-j} 
\end{align*}
in $\Omega^-(\Sigma_0,\varepsilon_0)$,
where $c^{\pm}_j(\sigma), c^{\pm}_{jn}(\sigma)\in \Q_{\theta}(\sigma).$
\end{thm}
\begin{rem}\label{rem2.2}
Note that $|x|<\varepsilon_0^2$ in $\Omega^{\pm}_{\sigma, \rho}(\varepsilon_0)$.
For each $(\sigma, \rho),$ $\Omega^+_{\sigma,\rho}(\varepsilon_0)$ is given by 
\begin{align*}
\re \sigma\cdot \log|x| &  + \log |\rho| + \log(\varepsilon^{-1}_0) <\im \sigma
\cdot \arg x
\\
 & <(\re \sigma -1) \log|x| + \log |\rho| - \log(\varepsilon^{-1}_0), 
\end{align*}
and $\Omega^-_{\sigma,\rho}(\varepsilon_0)$ by
\begin{align*}
(\re \sigma +1) \log|x| & + \log |\rho| + \log(\varepsilon^{-1}_0) <\im \sigma
\cdot \arg x
\\
& <\re \sigma\cdot \log|x| + \log |\rho| - \log(\varepsilon^{-1}_0). 
\end{align*}
\end{rem}
\begin{rem}\label{rem2.3}
The asymptotic solution with $0< \re \sigma < 1$  
in \cite[Theorem 6.1]{Andreev-Kitaev} coincides with
$y_+(\sigma,\rho,x)$.
\end{rem}
\begin{rem}\label{rem2.1.a}
For each $(\sigma, \rho),$ in the domain
\begin{align*}
\Omega_{\sigma,\rho}(\varepsilon_0):
& =\{ x\in \mathcal{R}(\C\setminus\{ 0\}); \,\,
|x(\rho x^{\sigma})| < \varepsilon_0, \,\, 
|x (\rho x^{\sigma})^{-1}| < \varepsilon_0 \}
\\
& \supset \Omega^{+}_{\sigma,\rho}
(\varepsilon_0) \cup \Omega^{-}_{\sigma,\rho}(\varepsilon_0), 
\end{align*}
$y(\sigma,\rho,x)$ is meromorphic and represented by the ratio of convergent 
series (see Sections \ref{ssc5.1}
and \ref{sc7}). In a special case, such expressions of $y(\sigma,\rho, x)$ 
are relatively simple formulas as in the following theorem, which may be 
regarded as counterparts of an elliptic representation for the sixth
Painlev\'e transcendents \cite{G-Critical}, \cite{G-Elliptic}.
\end{rem}
\begin{thm}\label{thm2.1.a}
Suppose that $\theta_0-\theta_x=\theta_{\infty}=0$ and $\sigma \in \Sigma_0.$
\par
$(1)$ If $\sigma^2-4\theta_0^2 \not= 0$ for every $\sigma\in \mathrm{cl}
(\Sigma_0),$ then
\begin{equation}\label{2.4}
y(\sigma,\rho, x)= \tanh^2 \biggl( \frac 12  \log (\tilde{\rho} x^{\sigma})
+  \sum_{n=1}^{\infty} x^n \sum_{j=-n}^n c_{jn}(\sigma) (\tilde
{\rho} x^{\sigma})^j \biggr)
\end{equation}
with $\tilde{\rho} 
=(2\theta_0 -\sigma) (2\theta_0 +\sigma)^{-1}\rho,$ in which the
series with $c_{jn}(\sigma)\in \Q[\theta_0](\sigma)$ converges in $\Omega
_{\sigma,\tilde{\rho}}(\varepsilon_0).$
\par
$(2)$ If $(\sigma +1)^2 -(2\theta_0-1)^2 \not=0$ for every $\sigma
\in \mathrm{cl}(\Sigma_0),$ then
\begin{equation}\label{2.3.a}
y(\sigma,\rho, x) = 1-\frac{2x \sinh (
\log (\breve{\rho} x^{\sigma+1})  +\Sigma(x) )}
{ 2(\sigma+1) + 2x\Sigma'(x) + x \sinh (
\log (\breve{\rho} x^{\sigma+1})  +\Sigma(x) )}
\end{equation}
with
$$
{\breve{\rho}} = \frac{(\sigma+2)(1-\sigma)(2\theta_0 -\sigma)}
{8\sigma (\sigma+1)^2(\sigma+3) } \rho,
$$
where the series
$$
\Sigma(x)=  \sum_{n=1}^{\infty} x^n \sum_{j=-n}^n \tilde{c}_{jn}(\sigma)
 ({\breve{\rho} } x^{\sigma +1})^j, \qquad \tilde{c}_{jn}(\sigma)\in
\Q[\theta_0](\sigma)
$$
and $\Sigma'(x)=(d/dx)\Sigma(x)$ converge in $\Omega_{\sigma+1, {\breve
{\rho}} }(\varepsilon_0).$
\end{thm}
\begin{rem}\label{rem2.3.a}
The expressions above describe oscillatory behaviours. Indeed, 
if $|\tilde{\rho} x^{\sigma}|
^{\pm 1}$ are bounded, then \eqref{2.4} is
$$
- \tan^2 ( - (i/2) \log (\tilde{\rho} x^{\sigma}) +O(x) )
= - \tan^2 (  (1/2) \arg (\tilde{\rho} x^{\sigma}) 
 - (i/2) \log |\tilde{\rho} x^{\sigma}| +O(x) ),
$$
and, if $|\breve{\rho} x^{\sigma +1}|^{\pm 1}$ are bounded, then 
\eqref{2.3.a} is
$$
1 - ((\sigma +1)^{-1} +O(x) ) x \sin ( \arg (\breve{\rho} x^{\sigma +1}) 
 - i \log |\breve{\rho} x^{\sigma + 1}| +O(x) ).
$$
In the case where
$\theta_0-\theta_x \not= 0$ or $\theta_{\infty}\not=0$ as well, $y(\sigma,
\rho, x)$ admits oscillatory expressions as follows 
(cf. Section \ref{sc7}):
\par
(i) under the suppositions of Theorem \ref{thm2.1},
if $|\rho x^{\sigma} |^{\pm 1} $ are bounded,
$$
y(\sigma, \rho, x) = \frac{\Phi_1(x) \Phi_2(x) F(x) G(x) + O(x)} 
{\Psi_1(x) \Psi_2(x) F(x) G(x) + O(x)}; 
$$
\par
(ii) under certain generic conditions added to the suppositions
of Theorem \ref{thm2.1},
if $|\hat{\rho} x^{\sigma+1} |^{\pm 1} $ are bounded,
$$
1 - \frac 1{y(\sigma, \rho, x)} = \frac{x({\Phi}^{\pi}_1(x)
{ \Phi}^{\pi}_2(x) {\Psi}^{\pi}_1(x) { \Psi}^{\pi}_2(x) 
 {F}^{\pi}(x)^2 {G}^{\pi}(x)^2 + O(x) ) } 
{2(\sigma +1)^2 ( (\sigma+1)^2 -(1-\theta_0+\theta_x)^2) 
{ \Phi}^{\pi}_2(x) {\Psi}^{\pi}_1(x) {F}^{\pi}(x)^2 {G}^{\pi}(x)^2 + O(x)} . 
$$
\par\noindent
Here $\Phi_1(x),$ $\Phi_2(x),$ $\Psi_1(x),$ $\Psi_2(x),$ $F(x),$ $G(x)$ are
as in Section \ref{ssc7.1}, and, say ${\Phi}_1^{\pi}(x)$, the result
of the substitution
$$
(\sigma, \rho, \theta_0-\theta_x, \theta_0+\theta_x, \theta_{\infty})
\mapsto 
({\sigma}+1, \hat{\rho}, 1-\theta_{\infty},1- \theta_0+\theta_x, 
\theta_0+\theta_x-1 )
$$
in $\Phi_1(x)$ (cf. \eqref{2.3}), $\hat{\rho}$ being such that
$$
\hat{\rho} =\frac{(\sigma-\theta_{\infty})(2-\theta_0+\theta_x +\sigma)
(\theta_0+\theta_x-\sigma)}{8\sigma^2(\sigma+1)^2} \rho
$$
(cf. the proof of Theorem \ref{thm2.9}). In the expressions above it seems
that the cancellations of the factors $F(x)G(x),$ $\Phi^{\pi}_2(x) 
\Psi^{\pi}_1(x) F^{\pi}(x)^2 G^{\pi}(x)^2$ occur, but
we have not succeeded in proving them.
\end{rem}
For each $(\theta_0, \theta_x, \theta_{\infty})$ we have solutions of special 
complex power type.
\begin{thm}\label{thm2.2}
Suppose that $\theta_0\theta_x \not=0.$ Let $\sigma_0=\theta_0\pm \theta_x$
or $\theta_x-\theta_0$ be such that $\sigma_0 \in \Sigma_+ := \C \setminus
(\{\sigma \le -1 \} \cup \Z)$ and $\sigma_0^2 -\theta_{\infty}^2 \not=0.$
Then $(\mathrm{V})$ admits a one-parameter family 
of solutions
$
\{ y_{\sigma_0}( \rho, x); \,  \rho\in \C \}
$
with the following properties.
\par
$(\mathrm{i})$ $y_{\sigma_0}( \rho, x)$ is holomorphic in $(\rho,x) \in
\Omega^0(\varepsilon_0) \cup \Omega^-(\sigma_0, \varepsilon_0)$
$\subset \C \times \mathcal{R}(\C\setminus \{0\}),$ where
\begin{align*}
& \Omega^{0}( \varepsilon_0):= \bigcup_{\rho \in 
\C} \{\rho \}\times \Omega^{0}_{\rho}(\varepsilon_0),
\quad  \Omega^-(\sigma_0, \varepsilon_0):= \bigcup_{\rho \in 
\C\setminus\{0\}} \{\rho \}\times \Omega^{-}_{\sigma_0,\rho}(\varepsilon_0),
\\
& \Omega^{0}_{\rho}(\varepsilon_0):=  
\{ x\in \mathcal{R}(\C\setminus\{ 0\}); \, |x|<\varepsilon_0, \,\,
|\rho x^{\sigma_0}| < \varepsilon_0 \}, 
\\
& \Omega^{-}_{\sigma_0, \rho}(\varepsilon_0):=  
\{ x\in \mathcal{R}(\C\setminus\{ 0\}); \, 
|x(\rho x^{\sigma_0})| < \varepsilon_0, \,\, 
| (\rho x^{\sigma_0})^{-1}| < \varepsilon_0 \}, 
\end{align*} 
$\varepsilon_0=\varepsilon_0( \theta_0, \theta_x, \theta_{\infty})$
being a sufficiently small number depending only on $(\theta_0,
\theta_x, \theta_{\infty}).$
\par
$(\mathrm{ii})$ $y_{\sigma_0}(\rho, x)$ is represented by the convergent series
as follows$:$
\par
$(\mathrm{ii.a})$ if $\sigma_0=\theta_0+\theta_x,$ then
$$
 1- \frac{\sigma_0^2 } { \theta_0 \theta_x} \rho x^{\sigma_0} 
 +\sum_{j\ge 2} c^0_j(\sigma_0) (\rho x^{\sigma_0})^j
 + \sum_{n=1}^{\infty} x^n \sum_{j\ge 0} c^0_{jn}
(\sigma_0) (\rho x^{\sigma_0})^j
$$
in $\Omega^0(\varepsilon_0),$ and
$$
 1- \frac{4\sigma_0^2 } { \sigma_0^2- \theta_{\infty}^2}(\rho x^{\sigma_0})^{-1} 
 +\sum_{j\ge 2} \tilde{c}^0_j(\sigma_0) (\rho x^{\sigma_0})^{-j}
 + \sum_{n=1}^{\infty} x^n \sum_{j\ge 0} \tilde{c}^0_{jn}
(\sigma_0) (\rho x^{\sigma_0})^{n-j}
$$
in $\Omega^-(\sigma_0,\varepsilon_0);$
\par
$(\mathrm{ii.b})$ if $\sigma_0=\theta_0-\theta_x,$ then
$$
- \frac{\sigma_0-\theta_{\infty}}{\sigma_0+\theta_{\infty}} \biggl(
 1- \frac{\sigma_0} { \theta_0 } \rho x^{\sigma_0} 
 +\sum_{j\ge 2} c^0_j(\sigma_0) (\rho x^{\sigma_0})^j
 + \sum_{n=1}^{\infty} x^n \sum_{j\ge 0} c^0_{jn}
(\sigma_0) (\rho x^{\sigma_0})^j \biggr)
$$
in $\Omega^0(\varepsilon_0),$ and
$$
 1- \frac{4\theta_0\sigma_0}{\sigma_0^2- \theta_{\infty}^2}(\rho x^{\sigma_0})^{-1} 
 +\sum_{j\ge 2} \tilde{c}^0_j(\sigma_0) (\rho x^{\sigma_0})^{-j}
 + \sum_{n=1}^{\infty} x^n \sum_{j\ge 0} \tilde{c}^0_{jn}
(\sigma_0) (\rho x^{\sigma_0})^{n-j} 
$$
in $\Omega^-(\sigma_0,\varepsilon_0);$
\par
$(\mathrm{ii.c})$ if $\sigma_0=\theta_x-\theta_0,$ then
$$
- \frac{\sigma_0+\theta_{\infty}}{\sigma_0-\theta_{\infty}} \biggl(
 1- \frac{\sigma_0} { \theta_x } \rho x^{\sigma_0} 
 +\sum_{j\ge 2} c^0_j(\sigma_0) (\rho x^{\sigma_0})^j
 + \sum_{n=1}^{\infty} x^n \sum_{j\ge 0} c^0_{jn}
(\sigma_0) (\rho x^{\sigma_0})^j \biggr)
$$
in $\Omega^0(\varepsilon_0),$ and
$$
 1- \frac{4\theta_x\sigma_0}{\sigma_0^2- \theta_{\infty}^2}(\rho x^{\sigma_0})^{-1} 
 +\sum_{j\ge 2} \tilde{c}^0_j(\sigma_0) (\rho x^{\sigma_0})^{-j}
 + \sum_{n=1}^{\infty} x^n \sum_{j\ge 0} \tilde{c}^0_{jn}
(\sigma_0) (\rho x^{\sigma_0})^{n-j} 
$$
in $\Omega^-(\sigma_0,\varepsilon_0)$.
\par
Here $c^0_j(\sigma_0),$ $c^0_{jn}(\sigma_0)$ $  \in \Q_{\theta}[\theta_0^{-1},
\theta_x^{-1}] (\sigma_0) $ and
$\tilde{c}^0_j(\sigma_0),$ $\tilde{c}^0_{jn}(\sigma_0)$ 
$ \in \Q_{\theta} (\sigma_0)$. 
\end{thm}
\begin{rem}\label{rem2.4}
For each $\rho\not=0$, $\Omega^0_{\rho}(\varepsilon_0)$ is given by 
$$
|x|<\varepsilon_0, \quad
\re \sigma_0\cdot \log|x|  + \log |\rho| + \log(\varepsilon^{-1}_0)<\im \sigma_0
\cdot \arg x,
$$
and $\Omega^-_{\sigma_0, \rho}(\varepsilon_0)$ by 
\begin{align*}
(1+\re \sigma_0) \log|x| & + \log |\rho| +\log(\varepsilon^{-1}_0)<\im \sigma_0
\cdot \arg x
\\
& <\re \sigma_0\cdot \log|x| + \log |\rho| - \log(\varepsilon^{-1}_0). 
\end{align*}
\end{rem}
\begin{rem}\label{rem2.5}
In $\Omega^0_0(\varepsilon_0),$
$y_{\sigma_0}(0, x)$ in each case is a Taylor series solution. If $\sigma_0
=\pm (\theta_0-\theta_x)$ then $y_{\sigma_0}(0, x)= - (\theta_0 -\theta_x
-\theta_{\infty})/(\theta_0 -\theta_x + \theta_{\infty} )+O(x),$ and if
$\sigma_0 =\theta_0 +\theta_x,$ direct substitution into (V) yields 
$y_{\sigma_0}(0, x)= 1+ (1- \theta_0 -\theta_x)^{-1} x +O(x^2).$ 
Since $\sigma_0\not\in \Z,$ the coefficients $c_{0n}^0(\sigma_0)$ 
of both solutions are uniquely determined,
and they coincide with the solutions (II) and (III) in \cite[Theorem 2]{K-O}, 
respectively.
\end{rem}
The following are solutions of inverse logarithmic type, which correspond to
the Chazy solutions of the sixth Painlev\'e equation \cite{M}.
\begin{thm}\label{thm2.3}
$(1)$ Suppose that $\theta_{\infty}\not=0$ and $\theta_0^2-\theta_x^2\not=
0.$ Then $(\mathrm{V})$ admits a 
one-parameter family of solutions $\{y_{\mathrm{ilog}}(\rho, x);\;
\rho \in \mathcal{R}(\C\setminus \{0\}) \}$ such that
$y_{\mathrm{ilog}}(\rho, x)$ is holomorphic in
$(\rho, x)\in \Omega^*(\varepsilon_0, \Theta_0) \subset \mathcal{R}(\C\setminus
\{0\})^2$ and is represented by the convergent series 
$$
y_{\mathrm{ilog}}(\rho,x)= 1+ \frac{4 } {\theta_{\infty}( \theta_0
- \theta_x)} \log^{-2}( \rho x) 
 +\sum_{j\ge 3} c_j\log^{-j} (\rho x)
 + \sum_{n=1}^{\infty} x^n \sum_{j\ge 0} c_{jn} \log^{2n-j} (\rho x)
$$
with $c_j,$ $c_{jn} \in \Q_{\theta}[\theta^{-1}_{\infty},
(\theta_0^2 - \theta_x^2)^{-1} ],$
where
\begin{align*}
& \Omega^{*}( \varepsilon_0,\Theta_0):=  \bigcup_{\rho \in \mathcal{R}( 
\C\setminus\{0\}) } \{\rho \}\times \Omega^{*}_{\rho}(\varepsilon_0,\Theta_0),
\\
& \Omega^{*}_{\rho}(\varepsilon_0, \Theta_0):=  
\{ x\in \mathcal{R}(\C\setminus\{ 0\}); \, |\rho x|<\varepsilon_0, \,\,
|x(\rho x)^{-1/2}| < \varepsilon_0, \,\, |\arg(\rho x)|<\Theta_0  \}, 
\end{align*} 
$\Theta_0$ being a given positive number and 
$\varepsilon_0=\varepsilon_0(\Theta_0, \theta_0, \theta_x, \theta_{\infty})$
a sufficiently small number depending only on $\Theta_0$ and $(\theta_0,
\theta_x, \theta_{\infty}).$
\par
$(2)$ Suppose that $\theta_{\infty}\not=0$ and $\theta_0^2-\theta_x^2=
0.$ If $\theta_0=\theta_x\not=0$ or if $\theta_0=- \theta_x\not=0$,
then $(\mathrm{V})$ admits 
one-parameter families of solutions $\{y^{\pm}_{\mathrm{ilog}}(\rho, x);\;
\rho \in \mathcal{R}(\C\setminus \{0\}) \}$ or 
$\{y^{+}_{\mathrm{ilog}}(\rho, x);\;
\rho \in \mathcal{R}(\C\setminus \{0\}) \}$, respectively,  
such that each solution is holomorphic in
$(\rho, x)\in \Omega^*(\varepsilon_0, \Theta_0)$ 
and is represented by the convergent series as follows$:$
\par
$\mathrm{(i)}$
if $\theta_0=\theta_x \not=0,$ 
$$
y^{\pm}_{\mathrm{ilog}}(\rho,x)= 1 \mp \frac{2 } {\theta_{\infty}}
 \log^{-1}( \rho x) 
 +\sum_{j\ge 2} c^{\pm}_j\log^{-j} (\rho x)
 + \sum_{n=1}^{\infty} x^n \sum_{j\ge 0} c^{\pm}_{jn} \log^{n-j} (\rho x);
$$
\par
$\mathrm{(ii)}$
if $\theta_0= -\theta_x \not=0,$
$$
y^{+}_{\mathrm{ilog}}(\rho,x)= 1 + \frac{2 } {\theta_0 \theta_{\infty}}
 \log^{-2}( \rho x) 
 +\sum_{j\ge 3} c^{+}_j\log^{-j} (\rho x)
 + \sum_{n=1}^{\infty} x^n \sum_{j\ge 0} c^{+}_{jn} \log^{n-j} (\rho x).
$$
Here
$c_j^{\pm},$ $  c_{jn}^{\pm}$ $ \in \Q[\theta_0, \theta_0^{-1},\theta_{\infty},
\theta^{-1}_{\infty}],$ in particular, 
in case $\theta_0=\theta_x$, $c^+_j=0$ for $j\ge 2.$ 
\par
$(3)$ Suppose that $\theta_{\infty}=0.$ 
If $\theta_0^2-\theta_x^2 \not=0$ or if $\theta_0 = - \theta_x \not= 0$, 
then $(\mathrm{V})$ admits 
one-parameter families of solutions $\{y^{(l)}_{\mathrm{ilog}}(\rho, x);\;
\rho \in \mathcal{R}(\C\setminus \{0\}) \}$ $(l=1,2)$ or
$\{y^{(1)}_{\mathrm{ilog}}(\rho, x);\;
\rho \in \mathcal{R}(\C\setminus \{0\}) \}$, respectively,
such that each solution
is holomorphic in $(\rho, x)\in \Omega^*(\varepsilon_0,\Theta_0)$ and is 
represented by the convergent series as follows$:$
\par
$\mathrm{(i)}$
if $\theta_0^2 -\theta_x^2 \not=0,$ 
$$
y^{(l)}_{\mathrm{ilog}}(\rho,x)= 1+ \frac{(-1)^{l+1}2 } {\theta_0 - \theta_x}
 \log^{-1}( \rho x) 
 +\sum_{j\ge 2} c^{(l)}_j\log^{-j} (\rho x)
 + \sum_{n=1}^{\infty} x^n \sum_{j\ge 0} c^{(l)}_{jn} \log^{2n-j} (\rho x);
$$
\par
$\mathrm{(ii)}$
if $\theta_0= -\theta_x \not=0,$
$$
y^{(1)}_{\mathrm{ilog}}(\rho,x)= 1 + \frac{1 } {\theta_0 }\log^{-1}( \rho x) 
 +\sum_{j\ge 2} c^{(1)}_j\log^{-j} (\rho x)
 + \sum_{n=1}^{\infty} x^n \sum_{j\ge 0} c^{(1)}_{jn} \log^{n-j} (\rho x).
$$
Here $c_j^{(l)},$ $ c_{jn}^{(l)} \in \Q[{\theta_0},\theta_{x},
(\theta_0^2 - \theta_x^2)^{-1} ]$ $(l=1,2)$ and $c_j^{(1)},$ $ c_{jn}^{(1)}
 \in \Q[{\theta_0},\theta_{0}^{-1}]$, respectively.
\end{thm}
\begin{rem}\label{rem2.6}
Applying the B\"{a}cklund transformation with $\pi$
in Lemma \ref{lem6.1} to the inverse logarithmic solutions above 
except for
$y^{+}_{\mathrm{ilog}}(\rho,x)$ with $\theta_0=\theta_x$,
we derive solutions of logarithmic type as follows: 
under the condition $\theta_0 +\theta_x\not=1$, solutions satisfying 
$$
y_{\mathrm{log}}(\rho,x) \sim 1- \frac 12 (1-\theta_0-\theta_x)
 x\log^2(\rho x) 
$$
with $(1-\theta_{\infty})(1-\theta_0+\theta_x)\not=0,$ with
$\theta_{\infty}=1,$ $\theta_0-\theta_x\not=1$ and with
$\theta_{\infty}\not=1,$ $\theta_0-\theta_x=1$ follow from
$y_{\mathrm{ilog}}(\rho,x)$, from
$y^-_{\mathrm{ilog}}(\rho,x)$ with $\theta_0=\theta_x\not=0$ and
from $y^+_{\mathrm{ilog}}(\rho,x)$ with $\theta_0=-\theta_x\not=0,$
respectively; and under the condition $\theta_0+\theta_x=1$, those satisfying
$$
y^{(l)}_{\mathrm{log}}(\rho,x) \sim 1 +(-1)^{l+1} x\log(\rho x) 
\quad \,\, (l=1,2)
$$
with $(1-\theta_{\infty})(1-\theta_0+\theta_x)\not=0$ and with
$\theta_{\infty}\not=1,$ $\theta_0-\theta_x=1$ follow from
$y^{(l)}_{\mathrm{ilog}}(\rho,x)$ $(l=1,2)$ with $\theta_0^2-\theta_x^2
\not=0$ and from
$y^{(1)}_{\mathrm{ilog}}(\rho,x)$ with $\theta_0=-\theta_x\not=0,$
respectively. These logarithmic solutions are studied in \cite{S-Log}. 
For the exceptional case of
$y^{+}_{\mathrm{ilog}}(\rho,x)$ the denominator of the 
B\"{a}cklund transformation is of
the form $x(\cdots)$, and to compute the resultant solution we need to
know some of the coefficients $c_{jn}^+.$
\end{rem}
\begin{thm}\label{thm2.4}
Suppose that $\theta_{\infty}=0.$ 
If $\theta_0=\theta_x $ or if $\theta_0 = - \theta_x ,$ then 
$(\mathrm{V})$ has a one-parameter
family of solutions $\{ y^+_{\mathrm{Taylor}}(a,x); a\in \C \setminus \{0\}
 \}$ or
$\{ y^-_{\mathrm{Taylor}}(a,x); a\in \C \}$, respectively, represented by
the convergent series 
$$
y^{\pm}_{\mathrm{Taylor}}(a,x)= \sum_{n=0}^{\infty} c^{\pm}_n(a) x^n
$$
with $c^+_0(a)=(a+\theta_0)/a,$ $c^+_1(a)=(a+\theta_0)(1-2\theta_0)/a,$
$c_n^+(a) \in \Q[a, a^{-1}, \theta_0 ]$, or
$c^-_0(a)=c^-_1(a)=1,$ $c^-_2(a)=(1-\theta_0 -2a)/2,$
$c_n^-(a) \in \Q[a, \theta_0 ]$. If $\theta_0=\theta_x=0,$ then 
$c_n^+(a)=c_n^-(a)$
for every $n \ge 0.$ 
\end{thm}
\subsection{Analytic continuation}\label{ssc2.2}
Suppose that
$( \sigma^2-(\theta_0\pm \theta_x)^2 )
( \sigma^2-\theta_{\infty}^2 ) \not=0$ for every
$\sigma\in \mathrm{cl} (\Sigma_0).$ Let us discuss the analytic continuation
of $y(\sigma, \rho, x)$ on $\mathcal{R}(\C\setminus \{0 \})$ around $x=0.$  
If $\sigma \in \Sigma_0$ satisfies $0<\sigma <1$ (respectively, $-1<\sigma<0$), 
then the solution $y_+(\sigma,
\rho, x)$ (respectively, $y_-(\sigma,\rho, x)$) in Theorem \ref{thm2.1} 
converges in $\{ x\in \mathcal{R}(\C\setminus\{0\} ); \, |x|<\varepsilon_0' \}$
for some $\varepsilon_0'>0$ (in fact $y_+(\sigma, \rho, x) \equiv
y_-(-\sigma, \rho^{-1}, x) $).
\par
In what follows suppose that $(\sigma,\rho) \in \Sigma_0 \times (\C \setminus
\{ 0 \})$ satisfies $\im \sigma\not=0.$ 
For $\nu\in \Z$ let $D_{\pm}
(\sigma, \rho, \nu) \subset \mathcal{R}(\C\setminus\{0\})$ be domains given by
\begin{align*}
&D_+(\sigma, \rho, \nu):\,\, 
\quad (\re \sigma -2\nu)\log|x| +\log |\rho|+\log(\varepsilon_0^{-1})
< \im \sigma \cdot \arg x
\\
&\phantom{----} <(\re \sigma -2\nu -1) \log|x| +\log |\rho| -\log
(\varepsilon_0^{-1}),
\\
&D_-(\sigma, \rho, \nu):\,\, 
\quad (\re \sigma -2\nu+1)\log|x| +\log |\rho|+\log(\varepsilon_0^{-1})
< \im \sigma \cdot \arg x
\\
&\phantom{----} <(\re \sigma -2\nu ) \log|x| +\log |\rho| -\log
(\varepsilon_0^{-1}).
\end{align*}
Furthermore set
\begin{align*}
&D_{\mathrm{even}}(\sigma, \rho, \nu):\,\,\, |x|<\varepsilon_0,
\\
&\phantom{--} -\log(\varepsilon_0^{-1}) <
 (\re \sigma -2\nu)\log|x| -  \im \sigma \cdot \arg x   +\log |\rho|
< \log(\varepsilon_0^{-1}),
\\
&D_{\mathrm{odd}}(\sigma, \rho, \nu):\,\,\, |x|<\varepsilon_0,
\\
&\phantom{--} -\log(\varepsilon_0^{-1})  <
 (\re \sigma -2\nu+1)\log|x| -  \im \sigma \cdot \arg x +\log |\rho|
< \log(\varepsilon_0^{-1}).
\end{align*}
{\tiny
\begin{center}
\unitlength=0.8mm
\begin{picture}(165,116)(-125,-30)
\put(0,0){\circle*{1}}
\put(-20,0){\circle*{1}}
\put(0,24){\circle*{1}}
\put(0,28){\circle*{1}}
\put(0,-26){\line(0,1){105}}
\put(-20,-26){\line(0,1){105}}
\put(20,0){\line(-1,0){97}}
\put(0,24){\line(-5,3){77}}
\put(0,28){\line(-5,3){77}}
\put(0,24){\line(-5,2){77}}
\put(0,28){\line(-5,2){77}}
\put(0,24){\line(-5,1){77}}
\put(0,28){\line(-5,1){77}}
\put(0,24){\line(-5,0){77}}
\put(0,28){\line(-5,0){77}}
\put(0,24){\line(-5,-1){77}}
\put(0,28){\line(-5,-1){77}}
\put(0,24){\line(-5,-2){77}}
\put(0,28){\line(-5,-2){77}}
\put(0,24){\line(-5,-3){77}}
\put(0,28){\line(-5,-3){77}}
\put(-111,72){\makebox{$D_{\mathrm{odd}}(\sigma,\rho, \nu+2)$}}
\put(-104,64){\makebox{$D_{\mathrm{+}}(\sigma,\rho, \nu+1)$}}
\put(-112,57){\makebox{$D_{\mathrm{even}}(\sigma,\rho, \nu+1)$}}
\put(-104,49){\makebox{$D_{\mathrm{-}}(\sigma,\rho, \nu+1)$}}
\put(-111,41){\makebox{$D_{\mathrm{odd}}(\sigma,\rho, \nu+1)$}}
\put(-97,33){\makebox{$D_{\mathrm{+}}(\sigma,\rho, \nu)$}}
\put(-105,25){\makebox{$D_{\mathrm{even}}(\sigma,\rho, \nu)$}}
\put(-97,17){\makebox{$D_{\mathrm{-}}(\sigma,\rho, \nu)$}}
\put(-104,9){\makebox{$D_{\mathrm{odd}}(\sigma,\rho, \nu)$}}
\put(-104,1){\makebox{$D_{\mathrm{+}}(\sigma,\rho, \nu-1)$}}
\put(-112,-6){\makebox{$D_{\mathrm{even}}(\sigma,\rho, \nu-1)$}}
\put(-104,-13){\makebox{$D_{\mathrm{-}}(\sigma,\rho, \nu-1)$}}
\put(-111,-22){\makebox{$D_{\mathrm{odd}}(\sigma,\rho, \nu-1)$}}
\put(3,1){\makebox{$0$}}
\put(3,77){\makebox{$\im \sigma \cdot \arg x$}}
\put(15,-5){\makebox{$\log|x|$}}
\put(-18,-5){\makebox{$\log(\varepsilon_0^2)$}}
\put(3,28){\makebox{$\log(|\rho|\varepsilon_0^{-1})$}}
\put(3,21){\makebox{$\log(|\rho|\varepsilon_0)$}}
\end{picture}
\end{center}
}
{\small \begin{center} Figure 2.1. \end{center}}  
\vskip0.3cm
In general these are spiral domains, and 
$$
\bigcup_{\nu\in \Z}\Bigl( \mathrm{cl}( D_{\mathrm{odd}}(\sigma,\rho,\nu)) \cup
 D_{-}(\sigma,\rho,\nu) \cup \mathrm{cl}( D_{\mathrm{even}}(\sigma,\rho,\nu) ) 
\cup D_{+}(\sigma,\rho,\nu) \Bigr)  
$$
contains
$ \{ x\in \mathcal{R}(\C\setminus\{0\} ); \, |x|<\varepsilon_0^2 \}.$
For every $\nu\in \Z,$ by Theorem \ref{thm2.1}, 
$y(\sigma -2\nu,\rho,x)$ with $\sigma \in \Sigma_0$ is represented by
$y_+(\sigma -2\nu, \rho, x)$ in $D_+(\sigma, \rho, \nu)$ and by 
$y_-(\sigma -2\nu, \rho, x)$ in $D_-(\sigma, \rho, \nu)$,
as long as $ ( (\sigma-2\nu)^2 -(\theta_0 \pm \theta_x)^2 ) 
 ( (\sigma-2\nu)^2 -\theta_{\infty}^2 ) \not= 0$ for $\sigma\in \mathrm{cl}
(\Sigma_0).$ 
Set
\begin{equation}\label{2.2}
c(\sigma):= \frac{4\sigma^2 (\theta_0+\theta_x -\sigma) }{(\sigma+
\theta_{\infty} )( \theta_0^2 -(\sigma+\theta_x)^2) }.
\end{equation}
To $c(\sigma)$ apply the substitution 
\begin{equation}\label{2.3}
\pi :\,\,\, \, (\theta_0-\theta_x, \theta_0 +\theta_x, \theta_{\infty})
\mapsto (1-\theta_{\infty}, 1-\theta_0+\theta_x, \theta_0+\theta_x -1)
\end{equation}
and denote the result by $c^{\pi}(\sigma)= \tilde{c}(\sigma)$. 
Then we have the following
relation, which gives the analytic continuation of $y(\sigma,\rho,x)$
on $ \{ x\in \mathcal{R}(\C\setminus\{0\} ); \, |x|<\varepsilon_0^2 \}.$
\begin{thm}\label{thm2.5}
For every $\nu\in \Z$,
$$
y(\sigma-2\nu +2, \rho, x) \equiv y(\sigma-2\nu, \gamma(\sigma,\nu)\rho, x)
$$
with
\begin{align*}
&\gamma(\sigma,\nu) = \frac 14 (2\nu -\theta_{\infty} -\sigma)(\sigma -
\theta_{\infty} -2\nu +2)
\\
&\phantom{-----} \times c(2\nu -\sigma) c(\sigma -2\nu+2) \tilde{c}
(\sigma-2\nu +1) \tilde{c}(2\nu-1-\sigma)
\\
&= \frac{64(\sigma-2\nu)^2 (\sigma -2\nu +1)^4(\sigma-2\nu +2)^2}
{(\theta_{\infty}-\sigma+2\nu)(\theta_{\infty}+\sigma -2\nu +2)
((\sigma-2\nu -\theta_x)^2-\theta_0^2)((\sigma-2\nu +2+\theta_x)^2-\theta_0^2)}
\end{align*}
as long as $\gamma(\sigma,\nu)\not=0,\infty$ for $\sigma \in \mathrm{cl}
(\Sigma_0).$
\end{thm}
\subsection{Distribution of poles, zeros and $1$-points}\label{ssc2.21}
If $\theta_0-\theta_x=\theta_{\infty}=0$ and 
$\im \sigma\not=0,$ then, by Remark \ref{rem2.3.a}, \eqref{2.4} 
in $D_{\mathrm{even}} (\sigma,\tilde{\rho}, 0)$ has sequences of zeros
and of poles lying asymptotically along the curve $\log|\tilde{\rho} x^{\sigma}|
=\re \sigma\cdot \log |x| - \im \sigma \cdot \arg x + 
 \log|\tilde{\rho} | =0.$ 
Similarly, \eqref{2.3.a} in $D_{\mathrm{odd}}(\sigma, \breve{\rho}, 0)$ has 
sequences of $1$-points lying asymptotically
along the curve $ \log |\breve{\rho} x^{\sigma+1}| =
(\re \sigma +1) \cdot \log |x| - \im \sigma \cdot \arg x + 
\log|{\breve{\rho}} | =0.$ 
These facts are generalised 
by the following theorems, in which 
\begin{equation}\label{2.5}
L(r_0, \omega)_{\sigma} : \,\,\, (1+\re \sigma -\omega)\log |x| -\im \sigma
\cdot\arg x =r_0,\quad r_0, \,\omega \in \R
\end{equation} 
is a spiral curve if $\re\sigma\not=\omega-1$ or a ray if $\re\sigma=\omega-1$. 
\begin{thm}\label{thm2.6}
In addition to $(\sigma^2-\theta_{\infty}^2 ) ( \sigma^2- (\theta_0
\pm \theta_x)^2 ) \not=0,$ suppose that
$$
\theta_x ( \theta_0+\theta_x-\theta_{\infty})
( \theta_0^2-\theta_x^2+\sigma^2-2\theta_0\theta_{\infty} ) \not=0
$$
for $\sigma\in \mathrm{cl}(\Sigma_0)$ and
\begin{equation}\label{2.6}
 \theta_0 -\theta_x-\theta_{\infty} \not=0.
\end{equation}
Set 
$$
r_0:=\log|\xi_0|, \quad \mu_0:=\arg \xi_0 \quad \text{with \,\, $\xi_0:=
- \displaystyle\frac{\sigma+\theta_0 +\theta_x}{\sigma-\theta_0- \theta_x}
 \rho^{-1}$}.
$$
Then $y(\sigma,\rho, x)$ admits a sequence of simple zeros $\{x^0_{n} \}
_{n\in \N} \subset D_{\mathrm{even}}(\sigma,\rho, 0)$ such that
$$
|\sigma|^2 \log|x^0_{n} | -r_0 \re \sigma -\mu_0\im \sigma \sim -2\pi n
|\im \sigma |
$$
and $\dist(x^0_{n}, L(r_0,1)_{\sigma}) = O(|x^{0}_{n} |^2)$. Furthermore
for
$$
\hat{\xi}_0 := -\frac{(\sigma+\theta_{\infty})( (\sigma+\theta_x)^2 -\theta_0^2)
}{(\sigma-\theta_{\infty})( (\sigma-\theta_x)^2 -\theta_0^2)} \rho^{-1}
$$
there exists another sequence of simple zeros $\{\hat{x}^0_{n} \}_{n\in \N}
 \subset D_{\mathrm{even}}(\sigma, \rho, 0)$ with similar properties, that is,
$$
|\sigma|^2 \log|\hat{x}^0_{n} | -\hat{r}_0 \re \sigma -\hat{\mu}_0\im \sigma
 \sim -2\pi n |\im \sigma |
$$
and $\dist(\hat{x}^0_{n}, L(\hat{r}_0,1)_{\sigma}) = O(|\hat{x}^{0}_{n}
|^2)$, where $\hat{r}_0:=\log|\hat{\xi}_0|,$ $\hat{\mu}_0:= \arg \hat{\xi}_0.$
\end{thm}
\begin{thm}\label{thm2.7}
In addition to $(\sigma^2-\theta_{\infty}^2 ) ( \sigma^2- (\theta_0
\pm \theta_x)^2 ) \not=0,$ suppose that
$$
 \theta_0( \theta_0+\theta_x-\theta_{\infty})
( \theta_0^2-\theta_x^2-\sigma^2+2\theta_x\theta_{\infty} ) \not=0
$$
for $\sigma\in \mathrm{cl}(\Sigma_0)$ and
\begin{equation}
\label{2.7}
 \theta_0 -\theta_x+\theta_{\infty} \not=0.
\end{equation}
Set
$$
r_{\infty}:=\log|\xi_{\infty}|, \quad \mu_{\infty}:=\arg \xi_{\infty} \quad 
\text{with \,\, $\xi_{\infty}:=
 \displaystyle\frac{(\sigma+\theta_x)^2 -\theta_0^2}{(\sigma-\theta_x)^2
- \theta_0^2} \rho^{-1}$}.
$$
Then $y(\sigma,\rho, x)$ admits a sequence of simple poles $\{x^{\infty}_{n}\}
_{n\in \N} \subset D_{\mathrm{even}}(\sigma,\rho, 0)$ such that
$$
|\sigma|^2 \log|x^{\infty}_{n}| -r_{\infty}\re\sigma -\mu_{\infty}
\im \sigma \sim -2\pi n
|\im \sigma |
$$
and $\dist(x^{\infty}_{n}, L(r_{\infty},1)_{\sigma}) 
= O(|x^{\infty}_{n} |^2)$. Another similar sequence of simple poles
$\{\hat{x}^{\infty}_{n} \}_{n\in \N} 
 \subset D_{\mathrm{even}}(\sigma, \rho, 0)$ exists for
$$
\hat{\xi}_{\infty} := \frac{(\sigma+\theta_{\infty})(\sigma+\theta_0+\theta_x)
}{(\sigma-\theta_{\infty})( \sigma-\theta_0 -\theta_x)} \rho^{-1}.
$$
\end{thm}
\begin{thm}\label{thm2.8}
If $\theta_0-\theta_x-\theta_{\infty}=0$ in place of \eqref{2.6}, then
$\{x^{0}_{n}\}_{n\in \N} =\{\hat{x}^{0}_{n}\}_{n\in \N}$ is a sequence
of double zeros.  
If $\theta_0-\theta_x+\theta_{\infty}=0$ in place of \eqref{2.7}, then
$\{x^{\infty}_{n}\}_{n\in \N} =\{\hat{x}^{\infty}_{n}\}_{n\in \N}$ is
a sequence of double poles. 
\end{thm}
Note that the singular values of (V) are $y=0,1,\infty.$ The results above
describe sequences of zeros and of poles in $D_{\mathrm{even}}(\sigma,\rho,0)$.
In $D_{\mathrm{odd}}(\sigma,\rho,0)$ there exist sequences of $1$-points.
\begin{thm}\label{thm2.9}
In addition to $(\sigma^2 -\theta_{\infty}^2) (\sigma^2-(\theta_0 \pm
\theta_x)^2) \not=0,$ suppose that
\begin{align}
\label{2.6.a}
& (\theta_0 -1) (\theta_0+\theta_x -\theta_{\infty}) 
( (2-\theta_0  -\theta_{\infty})^2 - \theta_x^2)
( (\sigma+1)^2-(\theta_0 \pm  \theta_x -1)^2) 
\\
\notag
 &\times ( (\sigma+1)^2 -(1-\theta_{\infty})^2)
 ((\sigma+1)^2+(1-\theta_0)^2 +2( 1-\theta_{\infty})(1-\theta_0) -\theta_x^2) 
\not=0
\end{align}
for $\sigma\in \mathrm{cl}(\Sigma_0).$ Set
$$
r_{1}:=\log|\xi_{1}|, \quad \mu_{1}:=\arg \xi_{1} \quad 
$$
with 
$$
\xi_{1}:=
 \frac{(\sigma+2 -\theta_0+\theta_x) (\sigma+\theta_{\infty})
c(-\sigma)\tilde{c}(\sigma+1)}{2(\sigma+\theta_0-\theta_x)} \rho^{-1}
= \frac{8\sigma^2(\sigma+1)^2}{(\theta_{\infty}-\sigma)
(\theta_0^2-(\theta_x-\sigma)^2)}\rho^{-1}.
$$
Then $y(\sigma,\rho, x)$ admits a sequence of simple $1$-points $\{x^{1}_{n}\}
_{n\in \N} \subset D_{\mathrm{odd}}(\sigma,\rho, 0)$ such that
$$
|\sigma+1|^2 \log|x^{1}_{n}| -r_{1}(\re\sigma+1) -\mu_{1}
\im \sigma \sim -2\pi n|\im \sigma |
$$
and $\dist(x^{1}_{n}, L(r_{1},0)_{\sigma}) 
= O(|x^{1}_{n} |^2)$. Another similar sequence of simple $1$-points
$\{\hat{x}^{1}_{n} \}_{n\in \N} 
 \subset D_{\mathrm{odd}}(\sigma, \rho, 0)$ exists for
$$
\hat{\xi}_{1} :=- \frac{\sigma+\theta_0+\theta_x
}{\sigma +2 -\theta_0 -\theta_x} \xi_{1}.
$$
\end{thm}
\begin{rem}\label{rem2.7}
In the proof of Theorem \ref{thm2.9} in Section \ref{sc7}, if we use
\eqref{6.2} instead of \eqref{6.1}, we obtain a sequence of simple $1$-points
such that
\begin{align*}
&\xi_1:=- \frac{2(\sigma -\theta_0 +\theta_x)}{(\sigma - 2 +\theta_0-\theta_x)
(\sigma-\theta_{\infty}) c(\sigma) \tilde{c}(1-\sigma) } \rho^{-1}
=\frac{(\theta_{\infty}+\sigma)(\theta_0^2-(\theta_x+\sigma)^2)}
{8\sigma^2(\sigma-1)^2} \rho^{-1},
\\
& |\sigma- 1|^2 \log |x^1_{n} | - r_1 (\re \sigma -1) -\mu_1 \im \sigma
\sim -2\pi n |\im \sigma |
\end{align*}
and $\dist (x^1_{n}, L(r_1, 2)_{\sigma} ) =O(|x^1_{n} |^2 ),$ and a
similar sequence for
$$
\hat{\xi}_1 := \frac{\sigma - 2+\theta_0+\theta_x}{\sigma - \theta_0 -\theta_x}
\xi_1.
$$
\end{rem}
\begin{rem}\label{rem2.8}
Using Theorems \ref{thm2.6} through \ref{thm2.9} combined with
the relation of Theorem \ref{thm2.5}, 
we may find sequences of zeros, of poles and of $1$-points beyond 
$\mathrm{cl}(D_{\mathrm{odd}}(\sigma, \rho, 0) ) \cup
D_{-}(\sigma,\rho, 0) \cup \mathrm{cl}(D_{\mathrm{even}}(\sigma, \rho, 0) ) \cup
D_{+}(\sigma,\rho, 0).$ As a result $D_{\mathrm{even}}(\sigma,\rho,\nu)$
contains sequences of zeros and of poles, and 
$D_{\mathrm{odd}}(\sigma,\rho,\nu)$ those of $1$-points. They are lying
asymptotically along the respective spiral curves or rays.
\end{rem}
For the solutions of special complex power type we have
\begin{thm}\label{thm2.10}
Under the same supposition as in Theorem $\ref{thm2.2}$ with $\rho\not=0,$ set
$$
r^*_0:= \log |\xi^*_0|, \quad \mu^*_0:= \arg\xi_0^* \quad
\text{with \,\, ${\xi}^*_{0}:=\displaystyle\frac{2\theta_0}{\sigma_0 +\theta_{\infty}}
\rho^{-1}$},
$$
and
$$
{\hat{\xi}}^*_{0}:= \frac{2\theta_x}{\sigma_0 -\theta_{\infty}}
\rho^{-1}, \quad
{\xi}^*_{\infty}:= - \frac{2\theta_x}{\sigma_0 +\theta_{\infty}}
\rho^{-1}, \quad
{\hat{\xi}}^*_{\infty}:= - \frac{2\theta_0}{\sigma_0 -\theta_{\infty}}
\rho^{-1}.
$$
\par
$(1)$ Let $\sigma_0=\theta_0+\theta_x.$ Then, under \eqref{2.6}, 
$y_{\sigma_0}(\rho, x)$ admits a sequence of simple zeros $\{x^{*0}_{n}\}
_{n\in \N} \subset D_{\mathrm{even}}(\sigma_0,\rho, 0)$ such that
$$
|\sigma_0|^2 \log|x^{*0}_{n}| -r_{0}^* \re\sigma_0  -\mu_{0}^*
\im \sigma_0 \sim -2\pi n|\im \sigma_0 |
$$
and $\dist(x^{*0}_{n}, L(r_{0}^*,1)_{\sigma_0}) 
= O(|x^{*0}_{n} |^2)$, and a similar sequence of simple zeros
$\{\hat{x}^{*0}_{n} \}_{n\in \N} 
 \subset D_{\mathrm{even}}(\sigma_0, \rho, 0)$ for $\hat{\xi}^*_{0}$. 
Under \eqref{2.7}, there exist sequences of simple poles 
$\{{x}^{*\infty}_{n} \}_{n\in \N}$, 
$\{\hat{x}^{*\infty}_{n} \}_{n\in \N} 
 \subset D_{\mathrm{even}}(\sigma_0, \rho, 0)$ for
${\xi}^*_{\infty}$ and $\hat{\xi}^*_{\infty}$,   
respectively.
If $\theta_0-\theta_x -\theta_{\infty}=0,$ then $x^{*0}_n =\hat{x}^{*0}_n$
$(n\in \N)$ are double zeros$;$ and if
$\theta_0-\theta_x +\theta_{\infty}=0,$ then $x^{*\infty}_n 
=\hat{x}^{*\infty}_n$
$(n\in \N)$ are double poles.
\par
$(2)$ If $\sigma_0=\theta_0 -\theta_x$ and $\theta_0+\theta_x\not=\theta
_{\infty},$ then $y_{\sigma_0}(\rho, x)$ admits 
sequences of simple zeros $\{x^{*0}_{n}\}_{n\in \N} $ and of simple poles
$\{\hat{x}^{*\infty}_{n}\}_{n\in \N}$ as in $(1)$.
\par
$(3)$ If $\sigma_0=\theta_x -\theta_0$ and $\theta_0+\theta_x\not=\theta
_{\infty},$ then $y_{\sigma_0}(\rho, x)$ admits 
sequences of simple zeros $\{\hat{x}^{*0}_{n}\}_{n\in \N} $ and of 
simple poles
$\{{x}^{*\infty}_{n}\}_{n\in \N}$ as in $(1)$.
\end{thm}
\begin{thm}\label{thm2.11}
In addition to the supposition of Theorem $\ref{thm2.2}$, suppose 
\eqref{2.6.a} with $\sigma=\sigma_0$. Set
$$
r^*_{1}:=\log|\xi^*_{1}|, \quad \mu^*_{1}:=\arg \xi^*_{1} \quad 
\quad \text{with} \quad 
\xi^*_{1}:=
\frac{ 2 (\sigma_0+1)^2 
c^*(\sigma_0) }{\sigma_0+\theta_0+\theta_x} \rho^{-1},
$$
where
$$
c^*(\sigma_0):=\begin{cases}
-\displaystyle\frac{4\sigma_0^2}{\sigma_0^2-\theta^2
_{\infty}} &\quad \text{if $\sigma_0=\theta_0 +\theta_x,$}
\\[0.4cm]
-\displaystyle\frac{4\theta_0\sigma_0}{\sigma_0^2-\theta_{\infty}^2}
 &\quad \text{if $\sigma_0=\theta_0 -\theta_x,$}
\\[0.4cm]
-\displaystyle\frac{4\theta_x\sigma_0}{\sigma_0^2-\theta_{\infty}^2}
 &\quad \text{if $\sigma_0=\theta_x -\theta_0 .$}
\end{cases}
$$
Then $y_{\sigma_0}(\rho, x)$ admits a sequence of simple $1$-points 
$\{x^{*1}_{n}\}_{n\in \N} \subset D_{\mathrm{odd}}(\sigma_0,\rho, 0)$ 
such that
$$
|\sigma_0+1|^2 \log|x^{*1}_{n}| -r_{1}^*(\re\sigma_0+1) -\mu_{1}^*
\im \sigma_0 \sim -2\pi n|\im \sigma_0 |
$$
and $\dist(x^{*1}_{n}, L(r_{1}^*,0)_{\sigma_0}) 
= O(|x^{*1}_{n} |^2)$. Another similar sequence of simple $1$-points
$\{\hat{x}^{*1}_{n} \}_{n\in \N} 
 \subset D_{\mathrm{odd}}(\sigma_0, \rho, 0)$ exists for
$$
\hat{\xi}^*_{1} :=- \frac{\sigma_0+\theta_0+\theta_x
}{\sigma_0 +2 -\theta_0 -\theta_x} \xi^*_{1}.
$$
\end{thm}
\begin{rem}\label{rem2.9}
The author believes that, in $D_{\mathrm{even}}(\cdots)$ and
$D_{\mathrm{odd}}(\cdots)$, there exists no sequence of zeros, of poles or
of $1$-points other than those given in Theorems \ref{thm2.6} through
\ref{thm2.11}. Indeed, solution \eqref{2.4} with the double zeros and poles
given by Theorem \ref{thm2.8} (respectively, \eqref{2.3.a} with the
$1$-points given by Theorem \ref{thm2.9}) 
has no zeros, poles or $1$-points other than
them in $D_{\mathrm{even}}(\sigma, \rho, 0)$ (respectively, in $D_{\mathrm{odd}}
(\sigma,\rho, 0)).$ 
However we have not succeeded in excluding the possibility of its existence
with the exception of the cases of 
$y_{\sigma_0}(\rho, x)$ with $\sigma_0=\theta_0
+\theta_x$ in $D_{\mathrm{even}}(\sigma_0, \rho, 0)$
and of
$y(\sigma, \rho, x)$ with $\theta_0-\theta_x
=\theta_{\infty}=0$.  
\end{rem}
\subsection{Monodromy data}\label{ssc2.3}
Linear system \eqref{1.1} given by
$$
\frac{dY}{d\lambda} = \Bigl(\frac{A_0(x)}{\lambda} + \frac{A_x(x)}{\lambda-x}
+\frac{J}{2} \Bigr)Y
$$
with the properties (a) and (b) 
admits a fundamental matrix solution of the form
\begin{equation}\label{2.8}
Y(\lambda, x)= (I+O(\lambda^{-1}) )e^{(\lambda/2)J} \lambda^{-(\theta_{\infty}
/2) J}
\end{equation}
as $\lambda\to\infty$
through the sector $-\pi/2 <\arg\lambda < 3\pi/2.$ Other matrix solutions
$Y_1(\lambda, x),$ $Y_2(\lambda,x)$ with the same asymptotic representation
through the sectors $-3\pi/2 <\arg\lambda <\pi/2,$ $\pi/2<\arg \lambda <
5\pi/2,$ respectively, are related to $Y(\lambda, x)$ by
$$
Y(\lambda, x)=Y_1(\lambda, x)S_1, \quad Y_2(\lambda, x)=Y(\lambda, x)S_2,
$$
where $S_1=I+s_1\Delta_-$ and $S_2=I+s_2\Delta$ are Stokes multipliers. Let
$M_0,$ $M_x,$ $M_{\infty}$ be monodromy matrices with respect to $Y(\lambda,x)$
such that $M_0,$ $M_x,$ $M_{\infty}$ are given by loops surrounding $\lambda
=0,$ $x,$ $\infty$, respectively, in the positive sense, and that $M_{\infty}
M_x M_0=I.$ Then the isomonodromy deformation of \eqref{1.1} preserving
the monodromy data $M_0,$ $M_x,$ $M_{\infty},$ $S_1,$ $S_2$ under a small
change of $x$, is controlled by the Schlesinger
equation
\begin{equation}\label{2.9}
x\frac{dA_0}{dx} =[A_x, A_0], \quad 
x\frac{dA_x}{dx} =[A_0, A_x] + \frac x2 [J, A_x]  
\end{equation}
that is equivalent to (V) through
\begin{equation}\label{2.10}
y(x)= \frac{A_x(x)_{12} (A_0(x)_{11} +\theta_0/2 )}
{A_0(x)_{12} (A_x(x)_{11} +\theta_x/2 )}
\end{equation}
(for details see \cite{JM}).
{\tiny
\begin{center}
\unitlength=0.75mm
\begin{picture}(80,58)(-5,-8)
\thicklines
\put(0,40){\circle*{1}}
\put(0,40){\line(0,-1){35}}
\put(0,0){\circle{10}}
\put(0,0){\circle*{1}}
\put(0,40){\line(1,-1){28.3}}
\put(31.8, 8.2){\circle{10}}
\put(31.8, 8.2){\circle*{1}}
\put(0,40){\line(1,0){70}}
\put(75,40){\circle{10}}
\put(75,40){\circle*{1}}
\thinlines
\put(-2,25){\vector(0,-1){8}}
\put(2,17){\vector(0,1){8}}
\qbezier(-2,-7)(0,-8)(2,-7)
\put(2,-7){\vector(3,2){0}}
\put(12,25){\vector(1,-1){5}}
\put(21,22){\vector(-1,1){5}}
\qbezier(34.7, 1.9)(36.8, 2.6)(37.6, 4.7)
\put(37.6, 4.7){\vector(1,4){0}}
\put(35,38){\vector(1,0){8}}
\put(43,42){\vector(-1,0){8}}
\qbezier(82, 38)(83, 40)(82, 42)
\put(82, 42){\vector(-2,3){0}}
\put(70,28){\makebox{$\lambda=\infty$}}
\put(43,4){\makebox{$\lambda=x$}}
\put(9,-6){\makebox{$\lambda=0$}}
\put(3,12){\makebox{$M_0$}}
\put(26,17.5){\makebox{$M_x$}}
\put(52,43){\makebox{$M_{\infty}$}}
\end{picture}
\end{center}
}
{\small \begin{center} Figure 2.2 \end{center} }
\vskip0.3cm\par
The following results give the monodromy data $M_0,$ $M_x$ related
to each solution of (V) by the correspondence described above. 
Note that $S_1$ and $S_2$
follow from the relation $M_{\infty}=M_0^{-1} M_x^{-1} = S_2 e^{\pi i\theta
_{\infty}J} S_1.$
\begin{thm}\label{thm2.12}
Suppose that $\theta_0,$ $\theta_x \not\in \Z.$ Then the monodromy data
for $y(\sigma,\rho, x)$ of Theorem $\ref{thm2.1}$ are given by
$$
M_0=(C_0C_{\infty})^{-1} e^{\pi i \theta_0 J}C_0C_{\infty}, \quad
M_x=(C_xC_{\infty})^{-1} e^{\pi i \theta_x J}C_xC_{\infty}, 
$$
where
\begin{align*}
& C_{\infty} = \begin{pmatrix}
- \dfrac{e^{-\pi i(\sigma +\theta_{\infty})/2} \Gamma(-\sigma)}{\Gamma
(1-(\sigma -\theta_{\infty})/2) }
& -\dfrac{\Gamma(-\sigma)}{\Gamma(1- (\sigma +\theta_{\infty})/2 )}
\\[0.3cm]
- \dfrac{e^{\pi i(\sigma -\theta_{\infty})/2} \Gamma(\sigma)}{\Gamma
((\sigma +\theta_{\infty})/2) }
& \dfrac{\Gamma(\sigma)}{\Gamma((\sigma -\theta_{\infty})/2 )}
\end{pmatrix},
\\[0.2cm]
&C_0= \tilde{C}_0
 \begin{pmatrix}
1 & 0 \\   0 & 2\rho/(\sigma +\theta_{\infty})
\end{pmatrix},
\quad
C_x= \tilde{C}_1
 \begin{pmatrix}
1 & 0 \\   0 & 2\rho/(\sigma +\theta_{\infty})
\end{pmatrix}
\end{align*}
with
\begin{align*}
& \tilde{C}_0 =
\begin{pmatrix}
\dfrac{e^{\pi i (\sigma-\theta_0+\theta_x)/2} \Gamma(1-\sigma) \Gamma(-\theta_0)}
{\Gamma(-\frac{\sigma+\theta_0 +\theta_x}2)\Gamma(1-\frac
{\sigma+\theta_0-\theta_x}2)}
& 
\dfrac{e^{-\pi i(\sigma+\theta_0-\theta_x)/2} \Gamma(1+\sigma) \Gamma(-\theta_0)}
{\Gamma(\frac{\sigma-\theta_0 -\theta_x}2)\Gamma(1+\frac{\sigma-\theta_0+\theta_x}2)}
\\[0.4cm]
\dfrac{e^{\pi i (\sigma+\theta_0+\theta_x)/2} \Gamma(1-\sigma) \Gamma(\theta_0)}
{\Gamma(-\frac{\sigma-\theta_0+\theta_x}2)\Gamma(1-\frac{\sigma-\theta_0-\theta_x}2)}
& 
\dfrac{e^{-\pi i(\sigma-\theta_0-\theta_x)/2} \Gamma(1+\sigma) \Gamma(\theta_0)}
{\Gamma(\frac{\sigma+\theta_0 -\theta_x}2)\Gamma(1+\frac{\sigma+\theta_0+\theta_x}2)}
\end{pmatrix},
\\[0.3cm]
& \tilde{C}_1 =
\begin{pmatrix}
\dfrac{ \Gamma(1-\sigma) \Gamma(-1-\theta_x)}
{\Gamma(-\frac{\sigma-\theta_0 +\theta_x}2)\Gamma(-\frac{\sigma+\theta_0+\theta_x}2)}
& 
\dfrac{ \Gamma(1+\sigma) \Gamma(-1-\theta_x)}
{\Gamma(\frac{\sigma+\theta_0 -\theta_x}2)\Gamma(\frac{\sigma-\theta_0-\theta_x}2)}
\\[0.4cm]
\dfrac{ \Gamma(1-\sigma) \Gamma(1+\theta_x)}
{\Gamma(1-\frac{\sigma-\theta_0-\theta_x}2)\Gamma(1-\frac{\sigma+\theta_0-\theta_x}2)}
& 
\dfrac{ \Gamma(1+\sigma) \Gamma(1+\theta_x)}
{\Gamma(1+\frac{\sigma+\theta_0 +\theta_x}2)\Gamma(1+\frac{\sigma-\theta_0+\theta_x}2)}
\end{pmatrix}.
\end{align*}
\end{thm}
\begin{thm}\label{thm2.13}
Suppose that $\theta_0,$ $\theta_x \not\in \Z.$ Let $\sigma_0$ be as in
Theorem $\ref{thm2.2}$. For $\tilde{C}_0=((\tilde{C}_0)_{ij}),$
$\tilde{C}_1=((\tilde{C}_1)_{ij})$ and $C_{\infty}$ in Theorem $\ref{thm2.12}$,
set 
$$
(\tilde{C}_0)_{ij}(\sigma_0):=(\tilde{C}_0)_{ij}|_{\sigma=\sigma_0}, \quad
(\tilde{C}_1)_{ij}(\sigma_0):=(\tilde{C}_1)_{ij}|_{\sigma=\sigma_0}, 
\quad
C_{\infty}(\sigma_0):=C_{\infty}|_{\sigma=\sigma_0}.
$$
Then the monodromy data for $y_{\sigma_0}(\rho, x)$ of Theorem $\ref{thm2.2}$
are given by
$$
M_0=(C^*_0C_{\infty}(\sigma_0))^{-1} e^{\pi i \theta_0 J}C^*_0C_{\infty}
(\sigma_0), \quad
M_x=(C^*_xC_{\infty}(\sigma_0))^{-1} e^{\pi i \theta_x J}C^*_xC_{\infty}
(\sigma_0) 
$$
with 
$$
C^*_0= \tilde{C}^*_0
 \begin{pmatrix}
1 & 0 \\   0 & \rho
\end{pmatrix}, \quad
C^*_x= \tilde{C}^*_1
 \begin{pmatrix}
1 & 0 \\   0 & \rho
\end{pmatrix} . 
$$
Here $\tilde{C}^*_0$ and $\tilde{C}^*_1$ are as follows$:$
\par
$(\mathrm{i})$ if $\sigma_0=\theta_0 +\theta_x$,
$$
\tilde{C}^*_0 = \begin{pmatrix}
(\tilde{C}_0)_{11}(\sigma_0) & (\tilde{C}^*_0)_{12}  \\
  0 &     (\tilde{C}_0)_{22}(\sigma_0) 
\end{pmatrix}, \quad
\tilde{C}^*_1 = \begin{pmatrix}
(\tilde{C}_1)_{11}(\sigma_0) & (\tilde{C}^*_1)_{12}  \\
  0 &     (\tilde{C}_1)_{22}(\sigma_0) 
\end{pmatrix}
$$
with
$$
(\tilde{C}^*_{0})_{12} = - \frac{ \sigma_0 e^{-\pi i\theta_0} \Gamma(1+\sigma_0)
\Gamma(-\theta_0)}{\theta_0 \Gamma(1+\theta_x) } , \quad
(\tilde{C}^*_{1})_{12} = - \frac{ \sigma_0 \Gamma(1+\sigma_0)
\Gamma(-1-\theta_x)}{ \Gamma(1+\theta_0) } ;
$$
\par
$(\mathrm{ii})$ if $\sigma_0=\theta_0 -\theta_x$,
$$
\tilde{C}^*_0 = \begin{pmatrix}
(\tilde{C}_0)_{11}(\sigma_0) & (\tilde{C}_0)_{12}(\sigma_0)/\theta_x  \\
 0 &  (\tilde{C}_0)_{22}(\sigma_0)    
\end{pmatrix}, \quad
\tilde{C}^*_1 = \begin{pmatrix}
 0 & (\tilde{C}_1)_{12}(\sigma_0)/\theta_x  \\
 (\tilde{C}_1)_{21}(\sigma_0)\theta_x &  (\tilde{C}_1)_{22}(\sigma_0)    
\end{pmatrix};
$$
\par
$(\mathrm{iii})$ if $\sigma_0=\theta_x -\theta_0$,
$$
\tilde{C}^*_0 = \begin{pmatrix}
0 & (\tilde{C}_0)_{12}(\sigma_0) \sigma_0/\theta_0   \\
 (\tilde{C}_0)_{21}(\sigma_0) \theta_0/\sigma_0 & (\tilde{C}^*_0)_{22} 
\end{pmatrix}, \quad
\tilde{C}^*_1 = \begin{pmatrix}
(\tilde{C}_1)_{11}(\sigma_0) & (\tilde{C}^*_1)_{12}  \\
  0 &     (\tilde{C}_1)_{22}(\sigma_0) 
\end{pmatrix}
$$
with
$$
(\tilde{C}^*_{0})_{22} = \frac{ e^{\pi i\theta_0} \Gamma(1+\sigma_0)
\Gamma(\theta_0)}{\Gamma(1+\theta_x) } , \quad
(\tilde{C}^*_{1})_{12} = - \frac{ \sigma_0 \Gamma(1+\sigma_0)
\Gamma(-1-\theta_x)}{ \Gamma(1-\theta_0) } .
$$
\end{thm}
Let $\psi(x)=\Gamma'(x)/\Gamma(x).$
\begin{thm}\label{thm2.14}
Suppose that $\theta_0,$ $\theta_x \not\in \Z.$ Then the monodromy data
for $y_{\mathrm{ilog}} (\rho,x)$ or $y^{(l)}_{\mathrm{ilog}}
(\rho, x)$ $(l=1,2)$ of Theorem $\ref{thm2.3}$ with $\theta_0^2-\theta_x^2
\not=0$ are given by
$$
M_0=(C^{-}_0C_{\infty})^{-1} e^{\pi i \theta_0 J}C^{-}_0C_{\infty}, \quad
M_x=(C^{-}_xC_{\infty})^{-1} e^{\pi i \theta_x J}C^{-}_xC_{\infty},
$$
those for $y^{\pm}_{\mathrm{ilog}}(\rho, x)$ with $\theta_0^2-\theta_x^2=0$ by
$$
M_0=(C^{\pm}_0C_{\infty})^{-1} e^{\pi i \theta_0 J}C^{\pm}_0C_{\infty}, \quad
M_x=(C^{\pm}_xC_{\infty})^{-1} e^{\pi i \theta_x J}C^{\pm}_xC_{\infty},
$$
and those for $y^{(1)}_{\mathrm{ilog}}(\rho, x)$ 
with $\theta_0=-\theta_x\not=0$ by
$$
M_0=(C^{+}_0C_{\infty})^{-1} e^{\pi i \theta_0 J}C^{+}_0C_{\infty}, \quad
M_x=(C^{+}_xC_{\infty})^{-1} e^{\pi i \theta_x J}C^{+}_xC_{\infty},
$$
where
$$
C^{\pm}_0 =  \tilde{C}^{\pm}_0 \rho_0^{-\Delta} , \quad
C^{\pm}_x =  \tilde{C}^{\pm}_1 \rho_0^{-\Delta}  
$$
with $\rho_0=\rho \exp(-2\theta_x(\theta_0^2-\theta_x^2)^{-1})$ if
$\theta_0^2-\theta_x^2\not=0,$ $\rho_0=\rho$ otherwise. The
matrices $C_{\infty},$ $\tilde{C}_0^{\pm},$ $\tilde{C}^{\pm}_1$ are as follows$:$
\par
$(\mathrm{i})$ if $\theta_{\infty}\not=0,$ then
$$
C_{\infty} = \begin{pmatrix}
\dfrac{e^{-\pi i \theta_{\infty}/2}(\psi(1+\theta_{\infty}/2) - 2\psi(1)-\pi i) 
 }{\Gamma(1+ \theta_{\infty}/2) }
&
\dfrac{\psi(-\theta_{\infty}/2) - 2\psi(1) }{\Gamma(1 -\theta_{\infty}/2) }
\\[0.3cm]
\dfrac{e^{-\pi i \theta_{\infty}/2} }{\Gamma(1+ \theta_{\infty}/2) }
 &
\dfrac{ 1   }{\Gamma(1 -\theta_{\infty}/2) }
\end{pmatrix},
$$
and if $\theta_{\infty}=0,$ then 
$C_{\infty}= I- \psi(1) \Delta$ for $l=1,$ and 
$C_{\infty}= (1- \pi i -\psi(1)) (I+J)/2 +\Delta + \Delta_-$ 
for $l=2;$
\par
$(\mathrm{ii})$ 
$$
\tilde{C}^{\pm}_0 =
\begin{pmatrix}
\dfrac{e^{-\pi i (\theta_0 \pm \theta_x)/2} \Gamma(-\theta_0) }
{\Gamma(-(\theta_0\mp \theta_x)/2)\Gamma(1-(\theta_0\pm  \theta_x)/2)}
&
\dfrac{e^{-\pi i (\theta_0 \pm \theta_x)/2} \psi^0_{12}(\theta_0, \theta_x)
 \Gamma(-\theta_0) }
{\Gamma(-(\theta_0\mp \theta_x)/2)\Gamma(1-(\theta_0\pm  \theta_x)/2)}
\\[0.4cm]
\dfrac{e^{\pi i (\theta_0 \mp \theta_x)/2} \Gamma(\theta_0) }
{\Gamma(1+(\theta_0\mp \theta_x)/2)\Gamma((\theta_0\pm  \theta_x)/2)}
&
\dfrac{e^{\pi i (\theta_0 \mp \theta_x)/2} \psi^0_{22}(\theta_0, \theta_x)
 \Gamma(\theta_0) }
{\Gamma(1+(\theta_0\mp \theta_x)/2)\Gamma((\theta_0\pm  \theta_x)/2)}
\end{pmatrix}
$$
with
\begin{align*}
& \psi^0_{12}(\theta_0, \theta_x) =
\psi(-(\theta_0 \mp\theta_x)/2) + \psi(1-(\theta_0 \pm\theta_x)/2) -2\psi(1)
+\pi i,
\\ 
& \psi^0_{22}(\theta_0, \theta_x) =
\psi(1+(\theta_0 \mp\theta_x)/2) + \psi((\theta_0 \pm\theta_x)/2) -2\psi(1)
+\pi i,
\end{align*}
and
\begin{align*}
&\tilde{C}^{\pm}_1 = K^{\pm}
\\
& \times
\begin{pmatrix}
\dfrac{\Gamma( -1\pm \theta_x) }
{\Gamma(-(\theta_0\mp \theta_x)/2)\Gamma((\theta_0\pm  \theta_x)/2)}
&
\dfrac{ \psi^1_{12}(\theta_0, \theta_x)
 \Gamma( -1\pm\theta_x) }
{\Gamma(-(\theta_0\mp \theta_x)/2)\Gamma((\theta_0\pm  \theta_x)/2)}
\\[0.4cm]
\dfrac{ \Gamma(1\mp \theta_x) }
{\Gamma(1+(\theta_0\mp \theta_x)/2)\Gamma(1-(\theta_0\pm  \theta_x)/2)}
&
\dfrac{ \psi^1_{22}(\theta_0, \theta_x)
 \Gamma(1\mp \theta_x) }
{\Gamma(1+(\theta_0\mp \theta_x)/2)\Gamma(1-(\theta_0\pm  \theta_x)/2)}
\end{pmatrix}
\end{align*}
with $ K^- =I,$ $ K^+ = \Delta + \Delta_-,$
\begin{align*}
& \psi^1_{12}(\theta_0, \theta_x) =
\psi(-(\theta_0 \mp\theta_x)/2) + \psi((\theta_0 \pm\theta_x)/2) -2\psi(1),
\\ 
& \psi^1_{22}(\theta_0, \theta_x) =
\psi(1+(\theta_0 \mp\theta_x)/2) + \psi(1-(\theta_0 \pm\theta_x)/2) -2\psi(1).
\end{align*}
\end{thm}
\begin{rem}\label{rem2.10}
In $\tilde{C}_0^{\pm}$ and $\tilde{C}_1^{\pm}$ above, if $(\theta_0 \mp
\theta_x)/2 \in \Z$, read $\psi(-n)/\Gamma(-n) =(-1)^{n+1} n!$ for $n=0,1,
2, \ldots .$ For example, if $(\theta_0 \mp \theta_x)/2 =0,$ then
$$
\tilde{C}^{\pm}_0 = \begin{pmatrix}
0 &  e^{-\pi i \theta_0}/\theta_0  \\
1  &  \pi i +\psi(\theta_0) -\psi(1) 
\end{pmatrix}, \quad
\tilde{C}^{\pm}_1 = \begin{pmatrix}
0 &  1/(1-\theta_0)  \\
1  &  \psi(1- \theta_0) -\psi(1) 
\end{pmatrix}. 
$$
\end{rem}
\begin{rem}\label{rem2.11}
For $\theta_0\in \Z$ or $\theta_x \in \Z$ as well, it is possible to compute
the monodromy data $M_0,$ $M_x$ by using Propositions \ref{prop9.6} through 
\ref{prop9.8} instead of Propositions \ref{prop9.3} and \ref{prop9.5}
in an argument of Section \ref{sc10}. 
\end{rem}
\begin{thm}\label{thm2.15}
If $\theta_0\not= 0,$ the monodromy data for $y^{\pm}_{\mathrm{Taylor}}
(a,x)$ of Theorem $\ref{thm2.4}$ are given by $M_0=T e^{\pi i \theta_0 J}
T^{-1},$ $M_x=T e^{ - \pi i \theta_0 J}T^{-1},$ and if $\theta_0=0,$ then 
$M_0 =T_0 e^{2\pi i \Delta} T_0^{-1},$ $M_x =T_0 e^{-2\pi i \Delta} T_0^{-1},$
where
$$
T= \begin{pmatrix}
1 & 1 \\ a  & \theta_0 + a 
\end{pmatrix}, \quad
T_0= \begin{pmatrix}
1 & 1 \\ a  & a -1 
\end{pmatrix}. 
$$
\end{thm}
\begin{rem}\label{rem2.12}
The transformation $Y=DZ$ with an invertible matrix $D$ preserves the form of
\eqref{1.1} and its properties (a), (b) if and only if $D$ is diagonal. Then
\eqref{1.1} becomes
\begin{equation}\label{2.11}
\frac{dZ}{d\lambda} = \Bigl( \frac{\tilde{A}_0(x) } {\lambda} +\frac{\tilde
{A}_x(x)}{\lambda- x} +\frac J 2 \Bigr) Z
\end{equation}
with $\tilde{A}_0(x) = D_0^{-1} A_0(x) D_0,$ 
$\tilde{A}_x(x) = D_0^{-1} A_x(x) D_0,$ $D_0 =\mathrm{diag} [r, 1/r],$
$r\not=0,$ and we have
$$
y(x) = \frac{ A_x(x)_{12} (A_0(x)_{11} +\theta_0/2 )}
{ A_0(x)_{12} (A_x(x)_{11} +\theta_x/2 )} =
 \frac{ \tilde{A}_x(x)_{12} (\tilde{A}_0(x)_{11} +\theta_0/2 )}
 { \tilde{A}_0(x)_{12} (\tilde{A}_x(x)_{11} +\theta_x/2 )},
$$
that is, $(\tilde{A}_0(x), \tilde{A}_x(x) )= (D_0^{-1} A_0(x)
D_0, D_0^{-1} A_x(x) D_0 )$ yields the same solution as that obtained from 
$(A_0(x), A_x(x))$ given in each of Theorems \ref{thm2.1} through \ref
{thm2.4} (cf. Section \ref{sc5}). Hence, $(M_0,M_x)$ for each
solution above may be replaced by $(D_0 M_0 D_0^{-1}, D_0 M_x D_0^{-1})$, 
which is
the monodromy for the matrix solution $Z(\lambda, x)= (I+O(\lambda^{-1}))
e^{(\lambda/2)J} \lambda^{-(\theta_{\infty}/2)J}$ of \eqref{2.11}
(for example, cf. \cite[Theorem 6.2]{Andreev-Kitaev}).
\end{rem}
\section{Solutions of the Schlesinger equation}\label{sc3}
We give matrix solutions of \eqref{2.9}, from which 
Theorems \ref{thm2.1}, 
\ref{thm2.2} and \ref{thm2.3} are
obtained by using \eqref{2.10}.
\par
Suppose that $\Lambda_0, \Lambda_x \in M_2(\Q_{\theta}[\sigma,\sigma^{-1}]),$
$T \in GL_2(\Q_{\theta}[\sigma,\sigma^{-1}])$ and $\Lambda:= \Lambda_0+
\Lambda_x$ have the properties:
\par
(P.1) the eigenvalues of $\Lambda_{\iota}$ $(\iota=0,x)$ are $\pm \theta_{\iota}
/2;$
\par
(P.2) $(\Lambda)_{11}=(\Lambda_0+\Lambda_x)_{11} =
- (\Lambda)_{22}=-(\Lambda_0+\Lambda_x)_{22} = -\theta_{\infty}/2;$
\par
(P.3) $T^{-1} \Lambda T = (\sigma/2)J.$
\par\noindent
For $\Sigma_0$ and $\Sigma_+$ as in Theorems \ref{thm2.1} and \ref{thm2.2} we
have
\begin{prop}\label{prop3.1}
$(1)$ System \eqref{2.9} possesses a two-parameter family of solutions
$
\{ (A_0(\sigma, \rho, x), A_x(\sigma, \rho, x)) ; \, (\sigma,\rho)\in \Sigma_0
\times (\C\setminus \{0 \}) \}
$ 
given by the convergent series
\begin{align*}
& A_0(\sigma, \rho, x) =(\rho x^{\sigma})^{\Lambda/\sigma}
\left( \Lambda_0 +\sum_{n=1}^{\infty} x^n \Pi^n_0 (\sigma, \rho x^{\sigma})
\right) (\rho x^{\sigma} )^{-\Lambda/\sigma},
\\
& A_x(\sigma, \rho, x) =(\rho x^{\sigma})^{\Lambda/\sigma}
\left( \Lambda_x +\sum_{n=1}^{\infty} x^n \Pi^n_x (\sigma, \rho x^{\sigma})
\right) (\rho x^{\sigma} )^{-\Lambda/\sigma}
\end{align*}
with 
$$
\Pi^n_{\iota}(\sigma, \xi) = \sum_{m=-n}^{n} C_{\iota m}^n (\sigma) \xi^m,
\quad C^n_{\iota m}(\sigma) \in M_2(\Q_{\theta}(\sigma) )
$$
$(\iota= 0,x)$, which are holomorphic in $(\sigma, \rho, x) \in \Omega(\Sigma_0,
\varepsilon_0) \subset \Sigma_0 \times (\C\setminus \{0\}) \times \mathcal{R}
(\C \setminus \{0\})$ with
\begin{align*}
&\Omega(\Sigma_0, \varepsilon_0) := \bigcup_{(\sigma, \rho) \in 
 \Sigma_0 \times (\C\setminus \{0\}) }   \{(\sigma, \rho) \} \times 
\Omega_{\sigma, \rho}(\varepsilon_0),
\\
&\Omega_{\sigma, \rho}(\varepsilon_0)= \{ x\in \mathcal{R}(\C\setminus\{ 0\});
\,\, |x(\rho x^{\sigma})|<\varepsilon_0, \,\,
 |x(\rho x^{\sigma})^{-1}|<\varepsilon_0 \},
\end{align*}
$\varepsilon_0$ being a sufficiently small number depending only on $\Sigma_0$
and $(\theta_0, \theta_x, \theta_{\infty}).$ 
\par
$(2)$ If $(T^{-1}\Lambda_0 T)_{21}$ vanishes at $\sigma=\sigma_0 \in \Sigma_+,$
then \eqref{2.9} admits a one-parameter family of solutions $\{(A_0(\sigma_0,
\rho, x), A_x(\sigma_0,\rho, x) ) ; \, \rho\in \C \}$ given by the 
representations above restricted to $\sigma=\sigma_0$ whose inner sums
satisfy
$$
\xi^{\Lambda/\sigma_0} \Pi^n_{\iota} (\sigma_0, \xi) \xi^{-\Lambda/\sigma_0}
= \sum^{n+1}_{m=0} \tilde{C}^n_{\iota m}(\sigma_0) \xi^m, \quad
\xi^{\Lambda/\sigma_0} \Lambda_{\iota} \xi^{-\Lambda/\sigma_0} =\tilde{C}^0
_{\iota 0} (\sigma_0) +\tilde{C}^0_{\iota 1}(\sigma_0) \xi
$$
with $\tilde{C}^n_{\iota m}(\sigma_0) \in M_2(\Q_{\theta}(\sigma_0))$ $(\iota
=0,x)$ for $n\ge 0.$ Each entry of the solution is holomorphic in $(\rho,
x) \in \Omega(\varepsilon_0) \subset \C \times \mathcal{R}(\C \setminus
\{ 0\}),$ where
\begin{align*}
&\Omega(\varepsilon_0) := \bigcup_{ \rho \in \C } \{ \rho \} \times 
\Omega_{\rho}(\varepsilon_0),
\\
&\Omega_{\rho}(\varepsilon_0):= \{ x\in \mathcal{R}(\C\setminus\{ 0\});
\, |x|<\varepsilon_0, \,\, |x(\rho x^{\sigma_0})|<\varepsilon_0 \},
\end{align*}
$\varepsilon_0$ being a sufficiently small number depending only on $\sigma_0$
and $(\theta_0, \theta_x, \theta_{\infty}).$ 
\end{prop}
Set $\Q_{\tilde{\theta}} := \Q_{\theta} [\theta_{\infty}^{-1} ]= \Q
[\theta_0, \theta_x, \theta_{\infty},\theta_{\infty}^{-1} ]$ if $\theta_{\infty}
\not=0,$ and $= \Q[\theta_0, \theta_x ]$ if $\theta_{\infty}=0.$
Let $\Lambda_0, \Lambda_x \in M_2(\Q_{\tilde{\theta}}),$
$T \in GL_2(\Q_{\tilde{\theta}})$ and $\Lambda:= \Lambda_0+
\Lambda_x$ have the properties $(\mathrm{P}.1), (\mathrm{P}.2)$ and
\par
$(\mathrm{P}.3')$ $T^{-1} \Lambda T = \Delta.$
\par\noindent
Then we have
\begin{prop}\label{prop3.2}
System \eqref{2.9} possesses a one-parameter family of solutions of logarithmic
type
$
\{ (A_0(\rho, x), A_x(\rho, x)) ; \, \rho \in 
\mathcal{R}(\C\setminus \{0 \}) \}
$ 
given by the convergent series
\begin{align*}
& A_0(\rho, x) =(\rho x)^{\Lambda}
\left( \Lambda_0 +\sum_{n=1}^{\infty} x^n \Pi^{*n}_0 (\log( \rho x))
\right) (\rho x )^{-\Lambda},
\\
& A_x(\rho, x) =(\rho x)^{\Lambda}
\left( \Lambda_x +\sum_{n=1}^{\infty} x^n \Pi^{*n}_x (\log( \rho x))
\right) (\rho x )^{-\Lambda}
\end{align*}
with 
$$
\Pi^{*n}_{\iota}(\xi) = \sum_{m=0}^{2n} C_{\iota m}^{*n} \xi^m,
\quad C^{*n}_{\iota m} \in M_2(\Q_{\tilde{\theta}})
$$
$(\iota= 0,x)$, which are holomorphic in $(\rho, x) \in \Omega^*(
\varepsilon_0, \Theta_0) \subset \mathcal{R}(\C \setminus \{0\})^2$,
where $\Omega^*(\varepsilon_0,\Theta_0)$ is the domain as in Theorem $\ref
{thm2.3}$. Furthermore, if 
$(T^{-1}\Lambda_0 T)_{21}=0,$ then   
$$
e^{\Lambda\xi} \Pi^{*n}_{\iota} (\xi) e^{-\Lambda\xi}
= \sum^{n+1}_{m=0} \tilde{C}^{*n}_{\iota m} \xi^m, \quad
e^{\Lambda \xi} \Lambda_{\iota} e^{-\Lambda \xi} =\tilde{C}^{*0}
_{\iota 0} +\tilde{C}^{*0}_{\iota 1} \xi
$$
with $\tilde{C}^{*n}_{\iota m} \in M_2(\Q_{\tilde{\theta}})$ $(\iota
=0,x)$ for $n\ge 0.$ 
\end{prop}
System \eqref{2.9} corresponds to the Schlesinger equation associated with
the sixth Painlev\'e equation studied in \cite{S2}.
These propositions are proved by the same arguments as in the proofs of
\cite[Theorems 2.1 and 2.2]{S2} given in \cite[\S 5]{S2}. We describe the
outline of the proofs of them. By the change of variables
\begin{equation}\label{3.1}
x= \kappa t, \quad A_0 = t^{\Lambda} (\Lambda_0 + U_0) t^{-\Lambda}, \quad
A_0+A_x=\Lambda + U_{\infty} ,
\end{equation}
where $\Lambda=\Lambda_0+\Lambda_x$ and
$\kappa\not= 0$, equation \eqref{2.9} is taken to
\begin{equation}\label{3.2}
\begin{split}
&t\frac{dU_0}{dt} = [t^{-\Lambda} U_{\infty} t^{\Lambda}, \Lambda_0 + U_0],
\\
&t\frac{dU_{\infty}}{dt} = \kappa t [J /2, t^{\Lambda} \Lambda_{x} t^{-\Lambda}
- t^{\Lambda} U_0 t^{-\Lambda} +U_{\infty} ],
\end{split}
\end{equation}
since $A_x= \Lambda +U_{\infty} - t^{\Lambda}(\Lambda_0 + U_0) t^{-\Lambda}
=t^{\Lambda} (\Lambda_x -U_0 +t^{-\Lambda} U_{\infty} t^{\Lambda})t^{-\Lambda}.$ 
The form of system \eqref{3.2} is similar to that of \cite[(5.2)]{S2}.
\par 
To show Proposition \ref{prop3.2} for each fixed $(\theta_0, \theta_x, 
\theta_{\infty}),$ we use the ring $\widehat{\mathfrak{L}}$
of formal series of the form
$$
\Phi=\Phi(\kappa, t)=\sum_{n=1}^{\infty} \sum_{m=0}^{2n} C^n_m (\kappa t)^n
\log^m t, \quad C^n_m \in M_2(\Q_{\tilde{\theta}}),
$$
and the subring
$$
\mathfrak{L}(D) :=\{\Phi \in \widehat{\mathfrak{L}} ; \, \|\Phi \| <\infty \,\,
\text{for} \,\, (\kappa, t) \in D \},
$$
which are defined in \cite[\S 4.1]{S2}. Here, for $\Phi\in \widehat{\mathfrak
{L}}$ as above, $\|\Phi \|$ is the norm defined by
$$
\|\Phi \| :=  \sum_{n=1}^{\infty} \sum_{m=0}^{2n} \|C^n_m\| \, |\kappa t|^n
|t|^{-m/4} 
$$
with the standard norm of the matrix $\| C^n_m \|=\max_{i=1,2}\{|(C^n_m)_{i1}|
+|(C^n_m)_{i2}| \}$, and $D$ is a subdomain of $(\C \setminus \{0\})
\times \mathcal{R}(\C\setminus \{0\}) $ such that $|\log t| \le |t|^{-1/4}$
for every $(\kappa, t) \in D.$
Then the holomorphic nature of $\Phi(\kappa,t)\in \mathfrak{L}(D)$ in $D$ is
guaranteed by \cite[Proposition 4.1]{S2}.
For $(m,n)\in (\N\cup\{0\})\times \N$ and $C\in M_2(\Q_{\tilde{\theta}})$ let
$$
\mathcal{I}[C(\kappa t)^n\log^m t ]
:= C\frac{(\kappa t)^n}{n} \Bigl( \log^m t +\cdots +
\frac{(-1)^jm!}{n^j(m-j)!} \log^{m-j} t + \cdots + \frac{(-1)^m m!}{n^m} \Bigr).
$$
Then $\mathcal{I}$ induces a linear operator $\mathcal{I}:$ $\widehat{\mathfrak
{L}} \to \widehat{\mathfrak{L}}$ assigning the formal primitive function of
$t^{-1} \Phi$ to each $\Phi\in \widehat{\mathfrak{L}},$ which satisfies
$\mathcal{I}[\Phi] \in \mathfrak{L}(D),$ $t(d/dt)\mathcal{I}[\Phi] =\Phi$ and
$\| \mathcal{I}[\Phi] \|\le 2 \|\Phi \|$ for $\Phi \in \mathfrak{L}(D),$
provided that, for $(\kappa, t)\in D$, $|t|$ is sufficiently
small. To construct a solution of \eqref{3.2} we define the sequence
$\{ (U_0^{(\nu)}, U_{\infty}^{(\nu)} ) \in (\widehat{\mathfrak{L}})^2; \,
\nu \ge 0 \}$ by 
\begin{align}
\notag
&U_{\infty}^{(0)} = U^{(0)}_0  \equiv 0,
\\
\label{3.3}
\begin{split}
&U_{\infty}^{(\nu+1)} = \mathcal{I} [ \kappa t [J/2, t^{\Lambda}\Lambda_x
t^{-\Lambda} - t^{\Lambda} U_0^{(\nu)} t^{-\Lambda} +U_{\infty}^{(\nu)} ]\, ],
\\
&U_{0}^{(\nu+1)} = \mathcal{I} [\, [t^{-\Lambda} U_{\infty}^{(\nu+1)} t^{\Lambda},
\Lambda_0+ U_0^{(\nu)} ]\, ]
\end{split}
\end{align} 
with $\Lambda_0,$ $\Lambda_x$ and $\Lambda$ satisfying (P.1), (P.2), (P.$3'$).
This converges to a formal series solution of \eqref{3.2}. Choosing $D$
suitably, we may show that $(U^{\infty}_0, U^{\infty}_{\infty})=
\lim_{\nu\to\infty} (U^{(\nu)}_0, U^{(\nu)}_{\infty} ) 
\in \mathfrak{L}(D)^2$ solves \eqref{3.2} and that $t^{-\Lambda} U^{\infty}
_{\infty} t^{\Lambda} \in \mathfrak{L}(D)$.
Setting $\kappa t=x$ and $t=\rho x$ in 
$$
A_0=t^{\Lambda}(\Lambda_0 +U^{\infty}_0)t^{-\Lambda},  \quad 
A_x=t^{\Lambda}(\Lambda_x -U^{\infty}_0 + t^{-\Lambda} U^{\infty}_{\infty}
t^{\Lambda} )t^{-\Lambda},
$$  
we obtain the
family of solutions $\{(A_0(\sigma, \rho, x), A_x(\sigma, \rho, x) ) \}$ 
as in Proposition \ref{prop3.2}. If $(T^{-1}\Lambda_0 T)
_{21} =0,$ then $\kappa t\cdot t^{\Lambda}U^{\infty}_0 t^{-\Lambda},$
$U^{\infty}_{\infty} \in \widehat{\mathfrak{L}}^*,$ where $\widehat{\mathfrak
{L}}^*$ is the subring of $\widehat{\mathfrak{L}}$ consisting of formal series
of the form
$$
\Phi=\sum_{n=1}^{\infty} \sum_{m=0}^{n} C^n_m (\kappa t)^n
\log^m t, \quad C^n_m \in M_2(\Q_{\tilde{\theta}}).
$$
From this fact the remaining part of Proposition \ref{prop3.2} follows (cf.
\cite[\S 5.2]{S2}).
\par
To show Proposition \ref{prop3.1} for each fixed $(\theta_0,\theta_x,
\theta_{\infty}),$  
consider the ring $\widehat{\mathfrak{S}}$ of formal series of the form
$$
\Phi=\Phi(\sigma, \kappa, t) = \sum_{n=1}^{\infty} \sum_{m=-n}^n C^n_m(\kappa t)
^n t^{\sigma m}, \quad C^n_m=C^n_m(\sigma) \in M_2(\Q_{\theta}(\sigma) ),
$$
and the subring
$$
\mathfrak{S}(D(\Sigma_0) ):=\{\Phi \in\widehat{\mathfrak{S}}; \,
\|\Phi\| <\infty \,\, \text{for} \,\, (\sigma,\kappa, t)\in D(\Sigma_0) \}
$$
as in \cite[\S 4.2]{S2}. Here, for $\Phi \in \widehat{\mathfrak{S}}$ as above,
$$
\|\Phi \| :=  \sum_{n=1}^{\infty} \sum_{m=-n}^n \| C^n_m\| \, |\kappa t|^n
| t^{\sigma }|^m, 
$$
and $D(\Sigma_0)$ is a subdomain of $\Sigma_0 \times (\C\setminus \{0\})
\times \mathcal{R}(\C\setminus \{0\})$. For $(m,n) \in\Z\times \N$ and $C\in
M_2(\Q_{\theta}(\sigma)),$ let $\mathcal{I}$ be such that
$$
\mathcal{I} [C (\kappa t)^n t^{\sigma m} ] :=\frac 1{n+ \sigma m} C(\kappa t)^n
t^{\sigma m}.
$$
This induces a linear operator $\mathcal{I}:$ $\widehat{\mathfrak{S} }\to
\widehat{\mathfrak{S}}$ satisfying
$\mathcal{I}[\Phi] \in \mathfrak{S}(D(\Sigma_0)),$ $t(d/dt)\mathcal{I}[\Phi]
 =\Phi$ and
$\| \mathcal{I}[\Phi] \|\le L_0 \|\Phi \|$ for some $L_0>0$ if 
$\Phi \in \mathfrak{S}(D(\Sigma_0)).$
Then we define $\{ (U^{(\nu)}_0, U^{(\nu)}_{\infty} ) \in (\widehat{\mathfrak
{S}})^2; \, \nu\ge 0\}$ by \eqref{3.3} with $\Lambda_0,$ $\Lambda_x$ and 
$\Lambda$ satisfying (P.1), (P.2), (P.3). Choosing $D(\Sigma_0)$ suitably, we
may show that $(U^{\infty}_0, U^{\infty}_{\infty}) = \lim_{\nu\to\infty}
(U^{(\nu)}_0, U^{(\nu)}_{\infty} ) \in \mathfrak{S}(D(\Sigma_0))^2$ solves 
\eqref{3.2} and that $t^{-\Lambda} U^{\infty}_{\infty} t^{\Lambda} \in
\mathfrak{S}(D(\Sigma_0))$. Setting $\kappa t= x$
and $t^{\sigma}= \rho x^{\sigma}$ we obtain Proposition \ref{prop3.1}, (1).
The second assertion is proved by considering, for $\sigma_0 \in \Sigma_+$, 
the ring $\widehat{\mathfrak{S}}^+(\sigma_0)$ of the formal series
$$
\Phi(\sigma_0, \kappa, t) = \sum_{n=1}^{\infty} \sum_{m=0}^n C^n_m(\kappa t)
^n t^{\sigma_0 m}, \quad C^n_m \in M_2(\Q_{\theta}(\sigma_0) ),
$$
and by using the sequence $\{ (Z^{(\nu)}_0, U_{\infty}^{(\nu)} )|_{\sigma=
\sigma_0}; \, \nu\ge 0 \}$ with $Z_0^{(\nu)}:= t^{\Lambda} U_0^{(\nu)}
t^{-\Lambda}$ such that $\kappa t Z_0^{(\nu)},$ $U_{\infty}^{(\nu)} \in 
\widehat{\mathfrak{S}}^+(\sigma_0)$ if $
(T^{-1}\Lambda_0 T)_{21}|_{\sigma=\sigma_0} =0.$ In a suitable domain, we get
$(U^{\infty}_{\infty}, \kappa t Z^{\infty}_0 )=\lim_{\nu\to\infty} (U^{(\nu)}
_{\infty}, \kappa t Z^{(\nu)}_0 ),$ from which the desired solution of
\eqref{2.9} follows (cf. \cite[\S 5.3]{S2}).
\begin{prop}\label{prop3.3}
The solution $(A_0(x), A_x(x))$ of \eqref{2.9} given by Proposition $\ref
{prop3.1}$ or $\ref{prop3.2}$ satisfies the conditions $(\mathrm{a})$ and
$(\mathrm{b})$ on \eqref{1.1}, and the corresponding system \eqref{1.1} has 
the isomonodromy property. 
\end{prop} 
\begin{proof}
Observing that $(d/dx) (A_0(x) + A_x(x)) =[J, A_x(x)]/2,$ in which the 
diagonal part on the right-hand side vanishes identically, 
we deduce (b) from the fact 
that $A_0(x)+ A_x(x) \to \Lambda$ as $x\to 0$ along a suitable curve. 
The property (a) may be verified by the same argument as in the 
proof of \cite[Proposition 3.1]{S2}.
\end{proof}
\section{Lemmas on matrices}\label{sc4}
For $\sigma\not=0,$ we have the following (cf. \cite[Lemma A.2]{G-Elliptic}):
\begin{lem}\label{lem4.1}
The matrices
\begin{align*}
& \Lambda_0 = T \begin{pmatrix}
(\Lambda'_0)_{11}  &  1  \\
(\Lambda'_0)_{21}    &  - (\Lambda'_0)_{11} 
\end{pmatrix} T^{-1}, 
\quad
 \Lambda_x = T \begin{pmatrix}
(\Lambda'_x)_{11}  &  - 1  \\
(\Lambda'_x)_{21}    &  - (\Lambda'_x)_{11} 
\end{pmatrix} T^{-1}, 
\\
& (\Lambda'_0)_{11}= \frac{ \sigma}2 - (\Lambda'_x)_{11} = \frac 1{4\sigma}
(\sigma^2 +\theta_0^2 -\theta_x^2),
\\
& -(\Lambda'_0)_{21}= (\Lambda'_x)_{21} = \frac 1{16\sigma^2}
((\theta_0 +\theta_x)^2 -\sigma^2) ((\theta_0 -\theta_x)^2 -\sigma^2), 
\\
&  T = \begin{pmatrix}
  (\sigma -\theta_{\infty})/2 & -1  \\
(\sigma+\theta_{\infty}) /2 &   1 
\end{pmatrix}, 
\quad
\Lambda =\Lambda_0+\Lambda_x = \begin{pmatrix}
-\theta_{\infty} /2 &  (\sigma -\theta_{\infty})/2  \\
(\sigma+\theta_{\infty}) /2 &  \theta_{\infty}/2  
\end{pmatrix} 
\end{align*}
have the properties $(\mathrm{P}.1)$, $(\mathrm{P}.2)$ and $(\mathrm{P}.3).$
\end{lem}
Using \cite[Proposition 2.1, Jordan case]{G-Matching}, we have
\begin{lem}\label{lem4.2}
$(1)$ If $\theta_{\infty}\not=0,$ then the matrices
\begin{align*}
&\Lambda_0 = T \begin{pmatrix}
\mp \theta_x /2  & 1  \\
(\theta_0^2 -\theta_x^2)/4  & \pm \theta_x/2
\end{pmatrix} T^{-1},
\quad
\Lambda_x = T \begin{pmatrix}
\pm \theta_x /2  & 0  \\
(\theta_x^2 -\theta_0^2)/4  & \mp \theta_x/2
\end{pmatrix} T^{-1},
\\
& 
T =\begin{pmatrix}
 -\theta_{\infty}  /2   &   1   \\
 \theta_{\infty}  /2   &   0  
\end{pmatrix} ,
\quad
\Lambda=\Lambda_0 +\Lambda_x =\frac{\theta_{\infty}}2 \begin{pmatrix}
  -1   &  -1  \\     1    &   1   \end{pmatrix} 
\end{align*}
have the properties $(\mathrm{P}.1)$, $(\mathrm{P}.2)$ and $(\mathrm{P}.3').$
\par
$(2)$ If $\theta_{\infty} =0,$ then the matrices $\Lambda_0,$ $\Lambda_x$
given as above with
$$
T=I \quad \Biggl(respectively, \,\,\, T=\begin{pmatrix} 0 & 1 \\ 1 & -1
\end{pmatrix} \, \Biggr)
$$
and $\Lambda =\Delta$ $($respectively, $\Lambda=\Delta_-$ $)$
have the properties $(\mathrm{P}.1)$, $(\mathrm{P}.2)$ and $(\mathrm{P}.3').$
\end{lem}
\begin{lem}\label{lem4.3}
Suppose that $\theta_{\infty}=0.$ 
\par
$(1)$ If $\theta_0 =\pm \theta_x \not=0,$ then
$$
\Lambda_0= -\Lambda_x = T(\theta_0/2) J T^{-1} =\begin{pmatrix}
\theta_0/2 +a &  -1 \\  a(\theta_0 +a) &  -\theta_0/2 -a 
\end{pmatrix},
\quad
T=\begin{pmatrix}
1 & 1 \\ a  & \theta_0 +a 
\end{pmatrix}
$$
have the properties $(\mathrm{P}.1)$ and $(\mathrm{P}.2).$
\par
$(2)$ If $\theta_0 =\theta_x =0,$ then
$$
\Lambda_0= -\Lambda_x = T\Delta T^{-1} =\begin{pmatrix}
a &  -1 \\  a^2 &  -a 
\end{pmatrix},
 \quad T=\begin{pmatrix}
1 & 1 \\ a  & a -1 
\end{pmatrix}
$$
have the properties $(\mathrm{P}.1)$ and $(\mathrm{P}.2).$
\end{lem}
\section{Proofs of Theorems \ref{thm2.1}, and \ref{thm2.2}
through \ref{thm2.4}}\label{sc5}
\subsection{Proof of Theorem \ref{thm2.1}}\label{ssc5.1}
By Proposition \ref{prop3.1} with $(\Lambda_0, \Lambda_x, T, \Lambda)$ 
in Lemma \ref{lem4.1} we have, for $\sigma\in \Sigma_0,$ $\rho \in 
\C \setminus \{ 0\},$
\begin{align*}
A_0(\sigma,\rho,x) & = T (\rho x^{\sigma})^{J/2} 
\biggl( T^{-1} \Lambda_0 T + \sum_{n=1}
^{\infty} x^n T^{-1} \Pi^n_0(\sigma, \rho x^{\sigma}) T \biggr) (\rho x^{\sigma}
)^{-J/2} T^{-1}
\\
&= T \biggl( (\rho x^{\sigma})^{J/2} T^{-1} \Lambda_0 T (\rho x^{\sigma})^{-J/2}
 + \sum_{n=1}
^{\infty} x^n \sum_{j=-n-1}^{n+1} \tilde{C}^n_{0j}(\sigma) (\rho x^{\sigma})^j 
 \biggr) T^{-1}
\end{align*}
with $\tilde{C}^n_{0j}(\sigma) \in M_2(\Q_{\theta}(\sigma) ),$ and hence
\begin{align*}
(A_0)_{11} &= -\sigma^{-1} \Bigl( \theta_{\infty} (\Lambda'_0)_{11} +\frac 14
(\sigma^2-\theta_{\infty}^2) \rho x^{\sigma} + (\Lambda'_0)_{21} (\rho 
x^{\sigma})^{-1} \Bigr) + (\cdots),
\\
(A_0)_{12} &= -\sigma^{-1} \Bigl( -(\sigma-\theta_{\infty})(\Lambda'_0)_{11} 
-\frac 14
(\sigma-\theta_{\infty})^2 \rho x^{\sigma} + (\Lambda'_0)_{21} (\rho 
x^{\sigma})^{-1} \Bigr) + (\cdots),
\end{align*}
where $(\cdots)$ denotes a series of the form $\sum_{n=1}^{\infty} x^n
\sum_{j=-n-1}^{n+1} c_{nj}(\sigma) (\rho x^{\sigma})^j,$ $c_{nj}(\sigma) \in
\Q_{\theta}(\sigma).$ Similarly,
\begin{align*}
(A_x)_{11} &= -\sigma^{-1} \Bigl( \theta_{\infty} (\Lambda'_x)_{11} -\frac 14
(\sigma^2-\theta_{\infty}^2) \rho x^{\sigma} + (\Lambda'_x)_{21} (\rho 
x^{\sigma})^{-1} \Bigr) + (\cdots),
\\
(A_x)_{12} &= -\sigma^{-1} \Bigl(-(\sigma-\theta_{\infty})(\Lambda'_x)_{11} 
+\frac 14
(\sigma-\theta_{\infty})^2 \rho x^{\sigma} + (\Lambda'_x)_{21} (\rho 
x^{\sigma})^{-1} \Bigr) + (\cdots).
\end{align*}
By Proposition \ref{prop3.3} and \eqref{2.10}, $y=Y_{11} Y_{12}$ with
\begin{equation}\label{5.1}
Y_{11} =\frac{(A_0)_{11} + \theta_0/2} {(A_x)_{11} + \theta_x/2}, \quad
Y_{12} =\frac{(A_x)_{12}}{(A_0)_{12} }
\end{equation}
solves (V). Then we may write $Y_{ij} = (Y_{ij})_{\mathrm{num}}
/(Y_{ij})_{\mathrm{den}}$ $( (i, j)=(1,1), (1,2) )$ with
\begin{equation}\label{5.2}
\begin{split}
(Y_{11})_{\mathrm{num}} &= 4\theta_0 \sigma^2 - 2\theta_{\infty} (\sigma^2
+\theta_0^2 - \theta_x^2) -2\sigma(\sigma^2 -\theta_{\infty}^2) \rho x^{\sigma}
\\
& +\frac{1}{2\sigma} ((\theta_0+ \theta_x)^2 -\sigma^2)
((\theta_0- \theta_x)^2 -\sigma^2) (\rho x^{\sigma})^{-1} +\sigma^2(\cdots),
\\
(Y_{11})_{\mathrm{den}} &= 4\theta_x \sigma^2 - 2\theta_{\infty} (\sigma^2
+\theta_x^2 - \theta_0^2) +2\sigma(\sigma^2 -\theta_{\infty}^2) \rho x^{\sigma}
\\
& -\frac{1}{2\sigma} ((\theta_0+ \theta_x)^2 -\sigma^2)
((\theta_0- \theta_x)^2 -\sigma^2) (\rho x^{\sigma})^{-1} +\sigma^2(\cdots),
\end{split}
\end{equation}
and
\begin{equation}\label{5.3}
\begin{split}
(Y_{12})_{\mathrm{num}} &= 2(\sigma-\theta_{\infty}) (\sigma^2
+\theta_x^2 - \theta_0^2) -2\sigma(\sigma -\theta_{\infty})^2 \rho x^{\sigma}
\\
& -\frac{1}{2\sigma} ((\theta_0+ \theta_x)^2 -\sigma^2)
((\theta_0- \theta_x)^2 -\sigma^2) (\rho x^{\sigma})^{-1} +\sigma^2(\cdots),
\\
(Y_{12})_{\mathrm{den}} &= 2(\sigma - \theta_{\infty}) (\sigma^2
+\theta_0^2 - \theta_x^2) +2\sigma(\sigma -\theta_{\infty})^2 \rho x^{\sigma}
\\
& +\frac{1}{2\sigma} ((\theta_0+ \theta_x)^2 -\sigma^2)
((\theta_0- \theta_x)^2 -\sigma^2) (\rho x^{\sigma})^{-1} +\sigma^2(\cdots).
\end{split}
\end{equation}
Aiming at a solution of the form as in Theorem \ref{thm2.1} we replace $\rho$
by
\begin{equation}\label{5.4}
\frac {(\theta_0-\theta_x +\sigma)(\theta_0+\theta_x -\sigma)}{2\sigma
(\sigma +\theta_{\infty})} \rho. 
\end{equation}
Then $(Y_{ij})_{\mathrm{num}}$ and $(Y_{ij})_{\mathrm{den}}$ become
\begin{equation}\label{5.5}
\begin{split}
(Y_{11})_{\mathrm{num}} &= 4\theta_0 \sigma^2 - 2\theta_{\infty} (\sigma^2
+\theta_0^2 - \theta_x^2) -(\sigma -\theta_{\infty})(\theta_{0}^2 -(\sigma
-\theta_x)^2) \rho x^{\sigma}
\\
& +(\sigma+\theta_{\infty}) (\theta_0^2-(\sigma + \theta_x)^2 )
 (\rho x^{\sigma})^{-1} +(\cdots),
\\
(Y_{11})_{\mathrm{den}} &= 4\theta_x \sigma^2 - 2\theta_{\infty} (\sigma^2
-\theta_0^2 + \theta_x^2) +(\sigma -\theta_{\infty})(\theta_{0}^2 -(\sigma
-\theta_x)^2) \rho x^{\sigma}
\\
& -(\sigma+\theta_{\infty}) (\theta_0^2-(\sigma + \theta_x)^2 )
 (\rho x^{\sigma})^{-1} +(\cdots),
\end{split}
\end{equation}
and $(Y_{12})_{\mathrm{num}} =(\sigma+\theta_{\infty})^{-1}
(Y_{12})_{\mathrm{num}}^*$, 
$(Y_{12})_{\mathrm{den}} =(\sigma+\theta_{\infty})^{-1}
(Y_{12})_{\mathrm{den}}^*$ with
\begin{equation}\label{5.6}
\begin{split}
(Y_{12})_{\mathrm{num}}^* &= 2(\sigma^2-\theta_{\infty}^2) (\sigma^2
-\theta_0^2 + \theta_x^2) -(\sigma -\theta_{\infty})^2(\theta_0^2 -
(\sigma-\theta_x)^2) \rho x^{\sigma}
\\
& -(\sigma+\theta_{\infty})^2 (\theta_0^2 -(\sigma+\theta_x)^2)
 (\rho x^{\sigma})^{-1} +(\cdots),
\\
(Y_{12})_{\mathrm{den}}^* &= 2(\sigma^2-\theta_{\infty}^2) (\sigma^2
+\theta_0^2 - \theta_x^2) +(\sigma -\theta_{\infty})^2(\theta_0^2 -
(\sigma-\theta_x)^2) \rho x^{\sigma}
\\
& +(\sigma+\theta_{\infty})^2 (\theta_0^2 -(\sigma+\theta_x)^2)
 (\rho x^{\sigma})^{-1} +(\cdots),
\end{split}
\end{equation}
which are holomorphic
in a domain on $\mathcal{R}(\C\setminus\{0\})$ where $|x(\rho x^{\sigma})|$
and $|x(\rho x^{\sigma})^{-1}|$ are sufficiently small. 
A general solution $y(\sigma,\rho, x)$ meromorphic in such a domain is
represented in terms of \eqref{5.5} and \eqref{5.6}.  
If $\rho x^{\sigma}$ is also sufficiently small, then, under the supposition of
Theorem \ref{thm2.1}, we may write, say $(Y_{11})_{\mathrm{num}}$, in the form
\begin{align*}
&(Y_{11})_{\mathrm{num}} = (\sigma +\theta_{\infty})(\theta_0^2 -(\sigma
+\theta_x)^2) (\rho x^{\sigma})^{-1}
\biggl( 1+ \frac{4\theta_0\sigma^2 -2\theta_{\infty}(\sigma^2 +
\theta_0^2 -\theta_x^2)}
{(\sigma +\theta_{\infty})(\theta_0^2 -(\sigma+\theta_x)^2)} \rho x^{\sigma}
\\
&\phantom{-----}
+ c^{(11)}_2 (\sigma)(\rho x^{\sigma})^2 +\sum_{n=1}^{\infty} x^n \sum_{j=-n-1}
^{n+1}c^{(11)}_{nj}(\sigma) (\rho x^{\sigma})^{j+1} \biggr) 
\end{align*}
with $c_2^{(11)}(\sigma),$ $c_{nj}^{(11)}(\sigma) \in \Q_{\theta}(\sigma).$
From this and analogous expressions of $(Y_{11})_{\mathrm{den}}$, 
$(Y_{12})_{\mathrm{num}}^*$ and $(Y_{12})_{\mathrm{den}}^*$, 
we derive $y_+(\sigma, \rho, x)$ convergent in $\Omega^+(\Sigma_0,
\varepsilon_0).$ If $(\rho x^{\sigma})^{-1}$ is sufficiently small, writing
\begin{align*}
&(Y_{11})_{\mathrm{num}} = - (\sigma -\theta_{\infty})(\theta_0^2 -(\sigma
-\theta_x)^2) \rho x^{\sigma}
\biggl( 1- \frac{4\theta_0\sigma^2 -2\theta_{\infty}(\sigma^2 +
\theta_0^2 -\theta_x^2)}
{(\sigma -\theta_{\infty})(\theta_0^2 -(\sigma-\theta_x)^2)}(\rho x^{\sigma})
^{-1}
\\
&\phantom{-----}
+ {c'}^{(11)}_2 (\sigma)(\rho x^{\sigma})^{-2} +\sum_{n=1}^{\infty} x^n \sum_{j=-n-1}
^{n+1}{c'}^{(11)}_{nj}(\sigma) (\rho x^{\sigma})^{-j-1} \biggr) 
\end{align*}
with ${c'}_2^{(11)}(\sigma),$ ${c'}_{nj}^{(11)}(\sigma) \in \Q_{\theta}(\sigma)$
and so on, we have $y_-(\sigma, \rho, x)$. Thus we obtain Theorem \ref{thm2.1}.
\subsection{Proof of Theorem \ref{thm2.2}}\label{ssc5.2}
If $\sigma^2_0 =(\theta_0 \pm \theta_x)^2,$ then $2\theta_0 \sigma_0^2-
\theta_{\infty}(\sigma_0^2 +\theta_0^2 -\theta_x^2) = 2\theta_0(\theta_0 \pm
\theta_x) (\theta_0 \pm \theta_x -\theta_{\infty}).$ Putting $\sigma=\sigma_0$
in $Y_{11}$ and $Y_{12}$ with \eqref{5.2} and \eqref{5.3}, we have
\begin{equation}
\label{5.7}
\begin{split}
&Y_{11}|_{\sigma=\sigma_0} = \frac
{2\theta_0 (\theta_0 \pm \theta_x) (\theta_0 \pm \theta_x -\theta_{\infty})
-\sigma_0( \sigma_0^2 -\theta_{\infty}^2 )\rho x^{\sigma_0} + \sigma_0^2
(\cdots)}
{2\theta_x (\theta_x \pm \theta_0) (\theta_x \pm \theta_0 -\theta_{\infty})
+\sigma_0( \sigma_0^2 -\theta_{\infty}^2 )\rho x^{\sigma_0} + \sigma_0^2
(\cdots)} ,
\\
&Y_{12}|_{\sigma=\sigma_0} = \frac
{2\theta_x (\theta_x \pm \theta_0) 
-\sigma_0( \sigma_0 -\theta_{\infty} )\rho x^{\sigma_0} + \sigma_0^2
(\sigma_0 -\theta_{\infty})^{-1} (\cdots)}
{2\theta_0 (\theta_0 \pm \theta_x) 
+\sigma_0( \sigma_0 -\theta_{\infty} )\rho x^{\sigma_0} + \sigma_0^2
(\sigma_0 -\theta_{\infty})^{-1} (\cdots)}
\end{split}
\end{equation}
with $(\cdots)$ denoting a series of the form $\sum_{n=1}^{\infty} x^n\sum_{j=0}
^{n+1} c_{nj}(\sigma_0) (\rho x^{\sigma_0})^j.$
Suppose that $\sigma_0=\theta_0 +\theta_x \not=\theta_{\infty}$ and $\theta_0
\theta_x \not=0.$ If $\rho x^{\sigma_0}$ is sufficiently small,
\begin{align*}
Y_{11} Y_{12} |_{\sigma=\sigma_0} &=
\frac{1-(\sigma_0 +\theta_{\infty}) \rho x^{\sigma_0}/(2\theta_0) +\cdots}
{1+(\sigma_0 +\theta_{\infty}) \rho x^{\sigma_0}/(2\theta_x) +\cdots}
\,\cdot \,
\frac{1-(\sigma_0 -\theta_{\infty}) \rho x^{\sigma_0}/(2\theta_x) +\cdots}
{1+(\sigma_0 -\theta_{\infty}) \rho x^{\sigma_0}/(2\theta_0) +\cdots}
\\
&= 1-\frac{\sigma_0^2}{\theta_0\theta_x} \rho x^{\sigma_0} +
\sum_{j\ge 2} c_j^0(\sigma_0) (\rho x^{\sigma_0})^j +\sum_{n=1}^{\infty}
x^n \sum_{j\ge 0} c^0_{jn}(\sigma_0) (\rho x^{\sigma_0})^j
\end{align*}
with $c^0_j(\sigma_0),$ $c_{jn}^0(\sigma_0) \in \Q_{\theta}[\theta_0^{-1},
\theta_x^{-1}] (\sigma_0),$ and if 
$(\rho x^{\sigma_0})^{-1}$ and $x (\rho x^{\sigma_0})$ are sufficiently small,
\begin{align*}
&=
\frac{1- 2\theta_{0}(\rho x^{\sigma_0})^{-1}/(\sigma_0+\theta_{\infty}) +\cdots}
{1+ 2\theta_{x}(\rho x^{\sigma_0})^{-1}/(\sigma_0+\theta_{\infty}) +\cdots}
\,\cdot \,
\frac{1- 2\theta_{x}(\rho x^{\sigma_0})^{-1}/(\sigma_0-\theta_{\infty}) +\cdots}
{1+ 2\theta_{0}(\rho x^{\sigma_0})^{-1}/(\sigma_0-\theta_{\infty}) +\cdots}
\\
&= 1-\frac{4\sigma_0^2}{\sigma_0^2 -\theta_{\infty}^2} 
(\rho x^{\sigma_0})^{-1} +
\sum_{j\ge 2} \tilde{c}_j^0(\sigma_0) (\rho x^{\sigma_0})^{-j}
+\sum_{n=1}^{\infty}
x^n \sum_{j\ge 0} \tilde{c}^0_{jn}(\sigma_0) (\rho x^{\sigma_0})^{n-j}
\end{align*}
with $\tilde{c}^0_j(\sigma_0),$ $\tilde{c}_{jn}^0(\sigma_0) \in \Q_{\theta}
(\sigma_0)$. The other cases are similarly treated.
\subsection{Proof of Theorem \ref{thm2.3}}\label{ssc5.3}
Suppose that $\theta_{\infty}\not=0.$ By Proposition \ref{prop3.2} with
$(\Lambda_0, \Lambda_x, T, \Lambda)$ in Lemma \ref{lem4.2}, (1), we have
$$
A_0(\rho,x) = T\biggl( (\rho x)^{\Delta} T^{-1} \Lambda_0 T(\rho x)^{-\Delta}
+\sum_{n=1}^{\infty} x^n (\rho x)^{\Delta} T^{-1} \Pi^{*n}_0 (\log(\rho x))
T (\rho x)^{-\Delta} \biggr) T^{-1}.
$$
Hence
\begin{align*}
&(A_0)_{11} = \pm \frac {\theta_x}2 - \frac{\theta_{\infty}}2 -\frac 14
(\theta_0^2 -\theta_x^2 \pm 2\theta_x\theta_{\infty}) \log(\rho x) +
\frac 18 (\theta_0^2 -\theta_x^2)\theta_{\infty} \log^2(\rho x) +(\cdots),
\\ 
&(A_0)_{12} = \pm {\theta_x} - \frac{\theta_{\infty}}2 +\frac{\theta_0^2-
\theta_x^2}{2\theta_{\infty}} -\frac 12
(\theta_0^2 -\theta_x^2 \pm \theta_x\theta_{\infty}) \log(\rho x)
\\
&\phantom{---------------} 
 + \frac 18 (\theta_0^2 -\theta_x^2)\theta_{\infty} \log^2(\rho x) +(\cdots),
\end{align*}
and
$$
(A_x)_{11}= -(A_0)_{11} -\frac{\theta_{\infty}} 2 +(\cdots), \quad
(A_x)_{12}= -(A_0)_{12} -\frac{\theta_{\infty}} 2 +(\cdots).
$$
Here the sign $\pm$ is chosen according to $(T^{-1}\Lambda_0 T)_{11} =
\mp \theta_x/2,$ and
$(\cdots)$ denotes a series of the form $\sum_{n=1}^{\infty} x^n
\sum_{j=0}^{n(\theta)} c^*_{nj} \log^j(\rho x)$ with 
$c^*_{jn} \in \Q_{\tilde{\theta}}$, $n(\theta)$ being such that 
$n(\theta)=2n +2$ if $\theta
_0^2-\theta_x^2 \not=0,$ and $=n+1$ if $\theta_0^2 -\theta_x^2 =0.$
Note that
$$
-Y_{11} = 1+ \frac{\theta_0 +\theta_x -\theta_{\infty} +(\cdots)}{2(A_0)_{11}
-\theta_x +\theta_{\infty} +(\cdots)}, \quad
-Y_{12} = 1+ \frac{\theta_{\infty} +(\cdots)}{2(A_0)_{12} +(\cdots)}. 
$$
If $\theta_0^2 -\theta_x^2\not=0,$ then
\begin{align*}
&-Y_{11} = 1+ \frac{4(\theta_0 +\theta_x -\theta_{\infty})}{\theta_{\infty}
(\theta_0^2 -\theta_x^2)} \log^{-2}(\rho x) 
+ \cdots , \phantom{--------}
\\
&-Y_{12} = 1+ \frac{4\theta_{\infty}}{\theta_{\infty}
(\theta_0^2 -\theta_x^2)} \log^{-2}(\rho x) 
+ \cdots ;
\end{align*}
and if $\theta_0^2 =\theta_x^2 \not=0,$ then
\begin{align*}
&-Y_{11} = 1\mp  \frac{\theta_0 +\theta_x -\theta_{\infty}}{\theta_x
\theta_{\infty}} \log^{-1}(\rho x) \Bigl( 1 + \frac{\theta_x \mp \theta_x}
{\theta_x\theta_{\infty}} \log^{-1}(\rho x) + \cdots \Bigr),
\\
&-Y_{12} = 1\mp  \frac{\theta_{\infty}}{\theta_x
\theta_{\infty}} \log^{-1}(\rho x) \Bigl( 1 + \frac{2\theta_{x} 
 \mp \theta_{\infty} }
{\theta_x\theta_{\infty}} \log^{-1}(\rho x) + \cdots \Bigr).
\end{align*}
From these formulas, we derive $y_{\mathrm{ilog}} (\rho, x)$ and 
$y^{\pm}_{\mathrm{ilog}} (\rho, x).$ 
In the case where $\theta_0^2-\theta_x^2\not=0,$ apparently, 
there exist two kinds of inverse logarithmic solutions depending on the
sign of $\mp\theta_x$, but by the following proposition
verified later we may replace $\rho$ suitably to derive
$y_{\mathrm{ilog}}(\rho, x)$ as in (1) independent of the sign; indeed 
$A^*_0(\rho, x)$ in the proposition has entries as follows:
\begin{align*}
&(A^*_0)_{11} =\frac{\theta_{\infty}}{8} (\theta_0^2-\theta_x^2) \log^2
(\rho x) -\frac 14 (\theta_0^2-\theta_x^2) \log (\rho x) -\frac{\theta
_{\infty}\theta_x^2}{2}(\theta_0^2-\theta_x^2)^{-1} -\frac{\theta_{\infty}}2
 + \cdots,
\\
&(A^*_0)_{12} = (A^*_0)_{11} 
 -\frac 14 (\theta_0^2-\theta_x^2) \log (\rho x) +\frac {1}{2\theta
_{\infty}}(\theta_0^2-\theta_x^2)+\cdots.
\end{align*}
\begin{prop}\label{prop5.0}
Suppose that $\theta_0^2-\theta_x^2\not=0.$ Let
$A^-_0(\rho,x)$ denote $A_0(\rho,x)$ in the case where
$(T^{-1} \Lambda_0 T)_{11} = \theta_x/2.$ 
Then $A^*_0(\rho,x) := A^-_0(\rho\exp(-2\theta_x(\theta_0^2-\theta_x^2)^{-1}),
x) = A_0(\rho\exp(\pm 2\theta_x(\theta_0^2-\theta_x^2)^{-1}),x) $ is 
represented by
\begin{align*}
& \frac 14 T  \begin{pmatrix}
(\theta_0^2 -\theta_x^2) \log(\rho x)  &
4-(\theta_0^2 -\theta_x^2) \log^2(\rho x) + 4\theta_x^2(\theta_0^2-\theta_x^2)
^{-1} 
\\
\theta_0^2 -\theta_x^2 &
-(\theta_0^2 -\theta_x^2) \log(\rho x)  \end{pmatrix}
T^{-1}
\\
 &  + \sum_{n=1}^{\infty} x^n \sum_{j=0}^{n(\theta)} A^*_{jn}
\log ^j(\rho x)
\end{align*}
with $A^*_{jn} \in M_2(\Q_{\theta}[ (\theta_0^2-\theta_x^2 )^{-1} ] ),$
which is independent of the sign $\pm.$  
\end{prop}
\noindent
If $\theta_0=\theta_x\not=0,$ from
$$
Y_{11}Y_{12}=\frac
{(2(A_0)_{11} +\theta_0 + (\cdots) ) (2(A_0)_{12} +\theta_{\infty} + (\cdots))} 
{(2(A_0)_{11} -\theta_x +\theta_{\infty} + (\cdots) ) (2(A_0)_{12}  + (\cdots))} 
$$
we derive the expression
$$
y^+_{\mathrm{ilog}} (\rho, x) = 1 -\frac{2}{\theta_{\infty}} \log ^{-1}(\rho x)
+\sum_{n=1}^{\infty} x^n \sum_{j \ge 0}c_{jn}^+ \log^{n-j}(\rho x).
$$
\par
Let us suppose that $\theta_{\infty}=0$. Then we use Lemma \ref{lem4.2}, (2).
If $\Lambda=\Delta,$ 
\begin{align*}
&(A_0)_{11} =\mp \frac{\theta_x}2 +\frac 14(\theta_0^2 -\theta_x^2)\log(\rho x)
+(\cdots)_{11},
\\
&(A_0)_{12} =1 \pm \theta_x \log(\rho x) -\frac 14 (\theta_0^2 -\theta_x^2)
\log^2(\rho x) +(\cdots)_{12},
\\
& (A_x)_{11} = -(A_0)_{11} +(\cdots)_{11}, \quad
(A_x)_{12} = -(A_0)_{12} +1 +(\cdots)_{12}, 
\end{align*}
and hence, under the condition $\theta_0^2-\theta_x^2 \not=0$ or $\theta_x
\not=0,$
$$
- Y_{11}= 1+\frac{\theta_0 +\theta_x +(\cdots)_{11}}{2(A_0)_{11}-\theta_x
 +(\cdots)_{11}},
\quad
- Y_{12}= 1-\frac{1 +(\cdots)_{12}}{(A_0)_{12} +(\cdots)_{12}},
$$
where $(\cdots)_{11}$ (respectively, $(\cdots)_{12}$) denotes the sum of
$c^*_{nj} x^n\log^j(\rho x)$ for $n\ge 1$ and for $0\le j \le n(\theta) -1$
(respectively, $0\le j \le n(\theta)$).
If $\Lambda=\Delta_-,$
\begin{align*}
&(A_0)_{11} =\frac 14(\theta_0^2 -\theta_x^2)\pm \frac{\theta_x}2 
-\frac 14(\theta_0^2 -\theta_x^2)\log(\rho x) +(\cdots)_{11},
\\
&(A_0)_{12} =\frac 14 (\theta_0^2 -\theta_x^2) +(\cdots)_{12},
\\
& (A_x)_{11} = -(A_0)_{11}+(\cdots)_{11} , \quad
(A_x)_{12} = -(A_0)_{12} +(\cdots)_{12}, 
\end{align*}
and hence, under the condition $\theta_0^2-\theta_x^2\not=0,$
$$
- Y_{11}= 1+\frac{\theta_0 +\theta_x +(\cdots)_{11}}{2(A_0)_{11}-\theta_x
 +(\cdots)_{11}},
\quad
- Y_{12}= 1+(\cdots)_{12},
$$
where $(\cdots)_{11}$ (respectively, $(\cdots)_{12}$) denotes the sum of
$c^*_{nj} x^n\log^j(\rho x)$ for $n\ge 1$ and for $0\le j \le n(\theta) -1$
(respectively, $0\le j \le n(\theta)-2$).
The solutions $y^{(1)}_{\mathrm{ilog}}(\rho, x)$ and
$y^{(2)}_{\mathrm{ilog}}(\rho, x)$ in (3) are derived from the formulas
above for $\Lambda=\Delta$ and for $\Lambda=\Delta_-$, respectively. 
In case $\theta_0^2-\theta_x^2\not=0,$ we use Proposition \ref{prop5.0}. 
If $\Lambda=\Delta,$ $\theta_0+\theta_x =0$ and 
$(A_0)_{11} = -\theta_x/2 +(\cdots)_{11}$,
then $y^{(1)}_{\mathrm{ilog}}(\rho, x)$ follows. Thus we obtain the 
expressions in Theorem \ref{thm2.3}.
\par
{\it Proof of Proposition $\ref{prop5.0}$.} 
Note that $t^{\Lambda} \Lambda_0 t^{-\Lambda} = T \hat{\Lambda}_0(\tau)
T^{-1},$ $t^{\Lambda} \Lambda_x t^{-\Lambda} = T(\Delta- \hat{\Lambda}_0(\tau))
T^{-1},$ where $\tau = t \exp (\mp 2\theta_x (\theta_0^2-\theta_x^2)^{-1})$
and 
\begin{align*}
\hat{\Lambda}_0(\tau): & =t^{\Delta} T^{-1} \Lambda_0 T t^{-\Delta}
\\
&= \Delta + 
 \frac 14  \begin{pmatrix}
(\theta_0^2 -\theta_x^2) \log\tau  &
-(\theta_0^2 -\theta_x^2) \log^2\tau + 4\theta_x^2(\theta_0^2-\theta_x^2)
^{-1} 
\\
\theta_0^2 -\theta_x^2 &
-(\theta_0^2 -\theta_x^2) \log\tau  \end{pmatrix}
\\
&= \tau^{\Delta}    \begin{pmatrix}
0 & 1+ \theta_x^2(\theta_0^2-\theta_x^2)^{-1} 
\\
(\theta_0^2 -\theta_x^2)/4  & 0  \end{pmatrix}
\tau^{-\Delta}.
\end{align*}
Putting $t=c_{\pm}\tau,$ $\kappa'= c_{\pm} \kappa$ with $c_{\pm}=
 \exp (\pm 2\theta_x (\theta_0^2-\theta_x^2)^{-1})$ in \eqref{3.3}, we have
\begin{align*}
&U^{(0)}_{\infty} = U^{(0)}_{0} \equiv 0,
\\
&U^{(\nu+1)}_{\infty} = \mathcal{I} \left[ \kappa' \tau [J/2, T(\Delta 
- \hat{\Lambda}_0(\tau) ) T^{-1} - (c_{\pm}\tau)^{\Lambda} U^{(\nu)}_0
(c_{\pm}\tau)^{-\Lambda} + U^{(\nu)}_{\infty} ]\,\right],
\\
&(c_{\pm}\tau)^{\Lambda} U^{(\nu+1)}_{0} (c_{\pm}\tau)^{-\Lambda}
 = \tau^{\Lambda} \mathcal{I} \left[ \tau^{-\Lambda} 
[U_{\infty}^{(\nu+1)} , T 
 \hat{\Lambda}_0(\tau) T^{-1} + (c_{\pm}\tau)^{\Lambda} U^{(\nu)}_0
(c_{\pm}\tau)^{-\Lambda}  ] \tau^{\Lambda}   \right] \tau^{-\Lambda} ,
\end{align*}
since
$$
[t^{-\Lambda} U_{\infty}^{(\nu+1)} t^{\Lambda}, \Lambda_0 + U^{(\nu)}_0 ]
 = t^{-\Lambda} [U_{\infty}^{(\nu+1)}, t^{\Lambda} (\Lambda_0 + U^{(\nu)}_0)
t^{-\Lambda} ] t^{\Lambda}.
$$
Using this new recursive relation, we may inductively show that, for every
integer $\nu,$ 
$( (c_{\pm}\tau)^{\Lambda} U^{(\nu)}_{0} (c_{\pm}\tau)^{-\Lambda}, U_{\infty}
^{(\nu)} )$ does not depend on the choice of the sign $\pm.$
Write
$\hat{U}^{\infty}_0(\kappa', \tau):= \lim_{\nu\to\infty}  
 (c_{\pm}\tau)^{\Lambda} U^{(\nu)}_{0} (c_{\pm}\tau)^{-\Lambda}. $ 
Setting $A^*_0(\rho, x)
= T \hat{\Lambda}_0(\rho x) T^{-1} + \hat{U}^{\infty}_0(1/\rho, \rho x),$ 
which is equal to $A_0(c_{\pm}\rho, x) =A^-_0(c_-\rho, x),$ we arrive
at the conclusion of Proposition \ref{prop5.0}.
\hfill$\square$
\subsection{Proof of Theorem \ref{thm2.4}}\label{ssc5.4}
Suppose that $\theta_{\infty}=0.$ Let $\Lambda_0=-\Lambda_x$ 
be as in Lemma \ref{lem4.3}. Then $U_0$ and $U_{\infty}$ such that $A_0=
\Lambda_0 + U_0,$ $A_0 +A_x=U_{\infty}$ satisfy 
\begin{equation}\label{5.8}
\begin{split}
&t\frac{dU_0}{dt} = [U_{\infty}, \Lambda_0+U_0 ],
\\
&t\frac{dU_{\infty}}{dt} = t [J, -\Lambda_0 -U_0 + U_{\infty}]
\end{split}
\end{equation}
with $t=x/2$ (cf. \eqref{3.1}, \eqref{3.2}). 
\begin{prop}\label{prop5.1}
System \eqref{5.8} admits a solution $(U_0, U_{\infty})= (U^*_0(t),U^*
_{\infty}(t) )$ with
$$
U^*_0(t) =\sum_{j=1}^{\infty} U^0_j t^j , \quad
U^*_{\infty}(t) =\sum_{j=1}^{\infty} U^{\infty}_j t^j 
$$
holomorphic around $t=0.$ Here $U^0_j,$ $U^{\infty}_j \in M_2(\Q[\theta_0,a]),$
and $(U^*_{\infty}(t))_{11}= (U^*_{\infty}(t))_{22}\equiv 0.$
\end{prop}
\begin{proof}
For $V= \sum_{j=0}^{\infty} V_j t^j \in M_2(\Q[\theta_0, a])[[t]]$ and
for $n\in \N$ write $V=O(t^n)$ if $V=\sum_{j=n}^{\infty} V_j t^j.$ 
If $V=O(t),$ we may set $\mathcal{I}[V] : =\sum_{j=1}^{\infty}
(V_j/j) t^j  \in M_2(\Q[\theta_0, a])[[t]]$, and then
$t(d/dt) \mathcal{I}[V]=V.$ By induction on $n$ we may define
$U^{(n)}_0,$ $U^{(n)}_{\infty} \in M_2(\Q[\theta_0,a])[[t]]$ $(n\ge 1)$ by
\begin{align*}
& U^{(0)}_{\infty} \equiv 0, \quad  U^{(0)}_{0} \equiv 0, 
\\
& U^{(n+1)}_{\infty} = \mathcal{I} [ t[ J, -\Lambda_0 -U^{(n)}_0 + U^{(n)}
_{\infty} ]\,],
\quad
 U^{(n+1)}_{0} = \mathcal{I} [ \,[ U^{(n+1)}_{\infty}, \Lambda_0 +U^{(n)}_0 
 ]\,].
\end{align*}
Indeed, supposing $U^{(n)}_{\infty} =O(t)$ and $U^{(n)}_0=O(t)$ we easily show
$U^{(n+1)}_{\infty}=O(t)$ and $U_0^{(n+1)}=O(t).$ Furthermore, since
\begin{align*}
& U^{(n+1)}_{\infty} - U^{(n)}_{\infty} 
= \mathcal{I} [ t[ J, -(U^{(n)}_0 -U^{(n-1)}_0)  + (U^{(n)}_{\infty} -
U^{(n-1)}_{\infty}) ]\,],
\\
& U^{(n+1)}_{0}-U^{(n)}_0  = \mathcal{I} [ \,[ U^{(n+1)}_{\infty} -U^{(n)}
_{\infty} , \Lambda_0 +U^{(n)}_0  ] +
[ U^{(n)}_{\infty}, U^{(n)}_0 -U^{(n-1)}_0  ]\,],
\\
& U^{(1)}_{\infty} -U^{(0)}_{\infty} =O(t), \quad
 U^{(1)}_{0} -U^{(0)}_{0} =O(t), 
\end{align*}
we have $U^{(n)}_{\infty} -U^{(n-1)}_{\infty}=O(t^n),$
$U^{(n)}_{0} -U^{(n-1)}_{0}=O(t^n)$ for $n\ge 1.$ Hence $U^*_{\infty}(t)
:=\lim_{n\to\infty}U^{(n)}_{\infty}$,
$U^*_{0}(t):=\lim_{n\to\infty}U^{(n)}_{0}$ are in $M_2(\Q[\theta_0, a])[[t]],$
and $(U^*_0(t), U^*_{\infty}(t))$ formally solves \eqref{5.8}. By \cite
[Theorem A]{GLS}, $U^*_{0}(t)$ and $U^*_{\infty}(t)$ are convergent around $t=0.$
It is easy to see that the diagonal part of $U^*_{\infty}(t)$ vanishes 
identically. 
\end{proof}
From the relations for $n=0,1$ in the proof above we have
\begin{align*}
&(U^{\infty}_1)_{11}=0, \quad (U^{\infty}_1)_{12}=2, \quad
(U^{0}_1)_{11}=4a(\theta_0+a), \quad (U^{0}_1)_{12}=-2(\theta_0 +2a), 
\\
&(U^{\infty}_2)_{11}=0, \quad (U^{\infty}_2)_{12}=2(\theta_0 +2a +1). 
\end{align*}
Note that $(A_0, A_x) =(\Lambda_0 + U^*_0(x/2), -\Lambda_0 -U^*_0(x/2) +
U^*_{\infty}(x/2))$ solves \eqref{2.9}. Here
\begin{align*}
&(A_0)_{11} =-(A_x)_{11} = \theta_0/2 +a + 2a(\theta_0 +a) x +\cdots,
\quad
(A_0)_{12} =-1  - (\theta_0 +2a) x +\cdots,
\\
&(A_x)_{12} =-(A_0)_{12} + x + ( 1/2) (\theta_0 + 2a +1) x^2 +\cdots,
\end{align*}
and hence $-Y_{12} =1+x+ (1/2)(1-\theta_0 -2a)x^2+\cdots.$ 
If $\theta_0=-\theta_x$,
then $-Y_{11}=1;$ and if $\theta_0=\theta_x,$ 
then $-Y_{11} = (a+\theta_0)a^{-1} (1-2\theta_0 x +\cdots).$ 
From these series we obtain $y^{\pm}_{\mathrm{Taylor}}(a,x).$
\section{Proof of Theorem \ref{thm2.5}}\label{sc6}
Recall the B\"{a}cklund transformation for (V) by Gromak \cite{Gromak} (see
also \cite[\S 39]{GLS}, \cite{Okamoto}).
\begin{lem}\label{lem6.1}
Let $y$ be a given solution of $(\mathrm{V})$ and
let $\pi$ be the substitution defined by \eqref{2.3}, that is,
$$
\pi: \,\,\,  (\theta_0 -\theta_x, \theta_0+\theta_x, \theta_{\infty}) 
\mapsto (1-\theta_{\infty}, 1-\theta_0+\theta_x, \theta_0 +\theta_x -1).
$$
Set
\begin{align*}
\hat{y}= \hat{B}(y):=
& 1-\frac{2xy}{xy' -(\theta_0-\theta_x +\theta_{\infty})y^2/2+(\theta_{\infty}
+x)y +(\theta_0-\theta_x -\theta_{\infty})/2}
\\
=& 1-\frac{2xy}{2(y-1)^2 (A_x)_{11} +\theta_x y^2 +2xy -\theta_x }.
\end{align*}
Then $\hat{y}^{\pi},$ which is the result of application of $\pi$ to $\hat{y}$,
solves $(\mathrm{V})$, that is,
\begin{align*}
&\hat{y}^{\pi} =\hat{B}(y)^{\pi} = {B}(y^{\pi}):
\\
&= 1-\frac{2xy^{\pi}}{x(y^{\pi})' -(\theta_0+\theta_x -\theta_{\infty})(y^{\pi})
^2/2+(\theta_{0} +\theta_x -1
+x)y^{\pi} +1-(\theta_0+\theta_x +\theta_{\infty})/2}
\end{align*}
is also a solution of $(\mathrm{V})$.
\end{lem}
\noindent
The second expression of $\hat{B}(y)$ follows from 
\begin{align*}
2(A_x)_{11}(y-1)^2 =& xy' -(\theta_0+\theta_{\infty}) (y-1)^2 - xy
\\
& + (1/2) (y-1) ((\theta_0 -\theta_x +\theta_{\infty} ) y -
 (3 \theta_0 +\theta_x +\theta_{\infty} ) )
\end{align*}
(cf. \cite[(1.2), (1.3), (1.4)]{Andreev-Kitaev}, \cite{JM}). 
\par
Concerning the uniqueness of a solution of (V) near $x=0$ we have
\begin{lem}\label{lem6.2}
Let $\sigma,$ $\rho_0 \in \C\setminus\{0\}$ with $\im \sigma\not=0.$ Let
$L(r_0,\omega)_{\sigma}$ be the curve defined by \eqref{2.5}. 
If $0<\omega <1$ $($respectively, $1<\omega <2)$,
then a solution of $(\mathrm{V})$ such that 
$$
y(x)= 1+\rho_0 x^{-\sigma}(1+o(1))   \quad
(\text{respectively,} \,\,\, = 1+ \rho_0 x^{\sigma}(1+o(1))\,\,)
$$
as $x\to 0$ along $L(r_0,\omega)_{\sigma}$ is uniquely determined.
\end{lem}
\begin{proof}
By $y=\tanh^2(u/2),$ (V) is changed into
$ x(xu')' =  f(x, e^{-u}, xe^u) $ with
\begin{align*}
 f(x, e^{-u}, xe^u)
=& \frac 18 \Bigl((\theta_0-\theta_x +\theta_{\infty})^2\frac{\sinh(u/2)}
{\cosh^3(u/2)} - (\theta_0 -\theta_x -\theta_{\infty})^2 \frac{\cosh(u/2)}
{\sinh^3(u/2)} \Bigr)
\\
&+ \frac 12 (1-\theta_0 -\theta_x) x \sinh u + \frac{x^2}8 \sinh(2u),
\end{align*}
where $f(x,\xi, \eta)$ is holomorphic around $x=\xi=\eta=0$ and $f(0,0,0)=0.$
Note that $|\rho_0 x^{1+\sigma}| =O(|x|^{\omega})$ along $L(r_0,\omega)_{\sigma}
.$ Suppose that $1<\omega <2$ and that $y(x)=1+\rho_0 x^{\sigma}(1+o(1))
=1+O(|x|^{\omega-1})$ along $L(r_0, \omega)_{\sigma}.$ Let $u(x)$ and $v(x)$
be such that
$u(x)= -\sigma\log x-\log(-\rho_0/4) -v(x)$ with
$y(x)=\tanh^2(u(x)/2)=1- 4 e^{-u(x)}(1+O(e^{-u(x)}) ).$ They satisfy
$(y(x)-1)(\rho_0 x^{\sigma})^{-1} =e^{v(x)}(1+O(e^{-u(x)}) ) =e^{v(x)}
(1+O(|x|^{\omega-1})),$ which implies $v(x)=o(1)$ along $L(r_0,\omega)_{\sigma}.
$ Then $v=v(x)$ solves
$$
x(xv')' = g(x, \rho_0 x^{\sigma}, \rho_0^{-1} x^{1-\sigma}, v),
$$
where $g(x, \xi, \eta, v)=f(x, \xi e^v, \eta e^{-v})$. The function $g(x,\xi,
\eta,v)$ is holomorphic around $x=\xi=\eta=v=0$ and satisfies
$g(x,\xi,\eta,v)=O(|x|+|\xi|+|\eta|)$ and
$
g(x,\xi,\eta,\tilde{v}) -g(x,\xi,\eta, v)=O(|x|+|\xi|+|\eta|)|\tilde{v}-v|
$
if $|v|$ and $|\tilde{v}|$ are small. For $x, x_0 \in L(r_0,\omega)_{\sigma}$
$$
x_0 v'(x_0) -xv'(x) =\int_{L(x_0)\setminus L(x)} g(t, \rho_0 t^{\sigma},
\rho_0^{-1} t^{1-\sigma}, v(t)) \frac{dt}t,
$$
where $L(x)\subset L(r_0,\omega)_{\sigma}$ is a curve joining $0$ to $x$ given
by $t=\tau e^{i\theta(\tau)},$ $\tau=|t|,$ $0<\tau \le |x|$ with 
$\theta(\tau) = ((1+ \re \sigma -\omega)\log\tau - r_0)/\im \sigma.$
Observing that
$dt/\d\tau =O(1)$ and $g(t,\rho_0 t^{\sigma} , \rho_0^{-1} t^{1-\sigma},
v(t)) = O(\tau^{\omega-1} +\tau^{2-\omega})$ along $L(r_0,\omega)_{\sigma}
,$ we have $xv'(x)\to c_0$ as $x\to 0$ for some $c_0 \in \C.$ Since $v(x)=o(1),$
we have $c_0=0,$ and hence
$$
v(x)= \int_{L(x)} \int_{L(s)} 
 g(t, \rho_0 t^{\sigma} , \rho_0^{-1} t^{1-\sigma}, v(t) )
 \frac{dt}t \frac{ds}s.
$$
If $v_1(x),$ $v_2(x)=o(1)$ are solutions of this equation, then $\phi(x)=\sup
_{t\in L(x)} |v_2(t)-v_1(t) |$ satisfies $\phi(x)=O(|x|^{\omega-1}+|x|^{2-
\omega})\phi(x),$ which implies the uniqueness of $y(x).$
\end{proof}
By Remark \ref{rem2.2}, $\Omega^-_{\sigma,\rho}(\varepsilon_0)$ and
$\Omega^+_{\sigma,\rho}(\varepsilon_0)$ are spanned by $L(r_0,\omega)_{\sigma}$
with $0<\omega <1$ and $1<\omega <2$, respectively. Hence
$$
y(\sigma,\rho, x) = \begin{cases}
y_-(\sigma,\rho, x) \sim 1+ c(-\sigma) \rho^{-1} x^{-\sigma} 
& \quad \text{on} \,\, L(r_0,\omega)_{\sigma} \,\, \text{with} \,\,
0<\omega <1,
\\[0.3cm]
y_+(\sigma,\rho, x) \sim 1+ c(\sigma) \rho x^{\sigma} 
& \quad \text{on} \,\, L(r_0,\omega)_{\sigma} \,\, \text{with} \,\,
1<\omega <2
\end{cases}
$$
with $c(\sigma)$ given by \eqref{2.2}.
To $y(\sigma,\rho, x)$, $y_{\pm}(\sigma,\rho,x)$ and $c(\sigma)$, apply 
$\pi$ of Lemma \ref{lem6.1}, and denote the results by
$y^{\pi}(\sigma,\rho, x),$ $y^{\pi}_{\pm}(\sigma,\rho, x)$ and $\tilde{c}
(\sigma):=c^{\pi}(\sigma).$ Then the results of the B\"{a}cklund transformation
$y^*(\sigma, \rho, x):=B(y^{\pi}(\sigma,\rho, x)),$
$y^*_{\pm}(\sigma, \rho, x):=B(y^{\pi}_{\pm}(\sigma,\rho, x))$ also solve
(V) and satisfy
$$
y^*(\sigma,\rho, x) = \begin{cases}
y^*_-(\sigma,\rho, x) \sim 1+ \dfrac{2 \rho x^{1+\sigma}} 
{(1-\theta_{\infty}+\sigma)\tilde{c}(-\sigma) }
& \quad \text{on} \,\, L(r_0,\omega)_{\sigma} \,\, \text{with} \,\,
0<\omega <1,
\\[0.4cm]
y^*_+(\sigma,\rho, x) \sim 1+ \dfrac{2 \rho^{-1} x^{1-\sigma}} 
{(1-\theta_{\infty}-\sigma)\tilde{c}(\sigma) }
& \quad \text{on} \,\, L(r_0,\omega)_{\sigma} \,\, \text{with} \,\,
1<\omega <2.
\end{cases}
$$
From the facts above, for every $\nu \in \Z$ we derive 
the following expressions along the
curve $L(r_0,\omega)_{\sigma} = L(r_0,\omega-2\nu)_{\sigma-2\nu} = 
L(r_0,\omega-2\nu +1)_{\sigma -2\nu +1}$:
\par
(1) if $2\nu < \omega < 2\nu +1,$
$$
y(\sigma-2\nu, \rho, x) \sim 1+ c(2\nu -\sigma) \rho^{-1} x^{2\nu -\sigma};
$$
\par
(2) if $2\nu +1 < \omega < 2\nu +2,$
$$
y(\sigma-2\nu, \rho, x) \sim 1+ c(\sigma- 2\nu ) \rho x^{\sigma- 2\nu};
$$
\par
(3) if $2\nu -1 < \omega < 2\nu,$
$$
y^*(\sigma-2\nu+1, \hat{\rho}, x)
 \sim 1+ \frac{2 \hat{\rho} x^{\sigma- 2\nu+2}}
{(2-\theta_{\infty} -2\nu+\sigma)\tilde{c}(2\nu -1-\sigma)} ;
$$
\par
(4) if $2\nu < \omega < 2\nu+1,$
$$
y^*(\sigma-2\nu+1, \hat{\rho}, x)
 \sim 1+ \frac{2 \hat{\rho}^{-1} x^{2\nu-\sigma }}
{(2\nu-\theta_{\infty} -\sigma)\tilde{c}(\sigma- 2\nu +1)} .
$$
By Lemma \ref{lem6.2}, matching (1) with (4), we derive
\begin{equation}\label{6.1}
y(\sigma-2\nu,\rho, x) \equiv y^*(\sigma-2\nu+1, \hat{\rho}, x)
\end{equation}
if 
$$
c(2\nu -\sigma) \rho^{-1} =\frac{2\hat{\rho}^{-1} }
{(2\nu -\theta_{\infty} -\sigma) \tilde{c}(\sigma-2\nu +1)} .
$$
Similarly, from (2) and (3) we obtain
\begin{equation}\label{6.2}
y(\sigma-2\nu,\rho, x) \equiv y^*(\sigma-2\nu-1, \hat{\rho}, x)
\end{equation}
if 
$$
c(\sigma -2\nu) \rho =\frac{2\hat{\rho}}
{(\sigma -\theta_{\infty} -2\nu) \tilde{c}(2\nu +1-\sigma)} .
$$
From these relations Theorem \ref{thm2.5} follows.
\begin{rem}\label{rem6.1}
For $y^*(\sigma,\rho, x)$ we have
$$
y^*(\sigma -2\nu +1, \rho, x)\equiv y^*(\sigma-2\nu -1, \gamma^*(\sigma,
\nu)\rho, x),
$$
where
$$
\gamma^*(\sigma,\nu) =\frac 14 (\sigma-\theta_{\infty} -2\nu)(2\nu-\theta
_{\infty} -\sigma) c(\sigma-2\nu) c(2\nu -\sigma)
\tilde{c}(2\nu+1-\sigma) \tilde{c}(\sigma-2\nu+1).
$$
\end{rem}
\section{Proofs of Theorems \ref{thm2.1.a}, and 
\ref{thm2.6} through \ref{thm2.11}}\label{sc7}
\subsection{Proofs of Theorems \ref{thm2.6} through \ref{thm2.11}}\label{ssc7.1}
By \eqref{5.5} and \eqref{5.6} we write $Y_{11}$ and $Y_{12}$ as follows:
$$
-Y_{11}=\frac{F(x) \Phi_1(x) +O(x)}{G(x)\Psi_1(x) +O(x) }, \quad
-Y_{12}=\frac{G(x) \Phi_2(x) +O(x)}{F(x)\Psi_2(x) +O(x) }
$$
in $D_{\mathrm{even}}(\sigma,\rho, 0) =\{ x\in \mathcal{R} (\C\setminus\{0\});
\, |x|<\varepsilon_0, \,\, \varepsilon_0 < |\rho x^{\sigma} | < 
\varepsilon_0^{-1} \},$ where
\begin{align*}
& F(x)=(\sigma-\theta_{\infty}) (\sigma +\theta_0 -\theta_x) (\rho x^{\sigma})
^{1/2} - (\sigma +\theta_{\infty})(\sigma-\theta_0+\theta_x)(\rho x^{\sigma})
^{-1/2},
\\
& G(x)=(\sigma-\theta_{\infty}) (\rho x^{\sigma})^{1/2} 
+ (\sigma +\theta_{\infty})(\rho x^{\sigma})^{-1/2},
\\
& \Phi_1(x)= (\sigma -\theta_0 -\theta_x) (\rho x^{\sigma})^{1/2} 
+ (\sigma+\theta_0+\theta_x)(\rho x^{\sigma})^{-1/2},
\\
& \Psi_1(x)=((\sigma -\theta_x)^2 -\theta_0^2) (\rho x^{\sigma})^{1/2} 
- ((\sigma+\theta_x)^2 -\theta_0^2)(\rho x^{\sigma})^{-1/2},
\\
& \Phi_2(x)=(\sigma-\theta_{\infty})((\sigma-\theta_x)^2-\theta_0^2)
(\rho x^{\sigma})^{1/2}
+(\sigma +\theta_{\infty})((\sigma+\theta_x)^2-\theta_0^2)(\rho x^{\sigma})
^{-1/2},
\\
& \Psi_2(x)=(\sigma-\theta_{\infty})(\sigma-\theta_0-\theta_x)(\rho x^{\sigma})
^{1/2} - (\sigma +\theta_{\infty})(\sigma+\theta_0+\theta_x)(\rho x^{\sigma})
^{-1/2}.
\end{align*}
Let us seek the zeros of
$$
\Phi_1(x) =(\sigma -\theta_0-\theta_x)\rho(\rho x^{\sigma})^{-1/2}(x^{\sigma}
-\xi_0 ), \quad \xi_0 = -\frac{\sigma +\theta_0 +\theta_x}{\sigma-\theta_0
-\theta_x} \rho^{-1}.
$$
Set $r_0=\log |\xi_0|$ and $\mu_0 =\arg \xi_0.$ Then we have
$$
x^{\sigma}-\xi_0 = -2i \exp( (\sigma\log x+r_0 +i\mu_0)/2)
\sin( (i\sigma \log x +\mu_0 -ir_0)/2),
$$
in which, along $L(r_0,1)_{\sigma},$
$$
\frac 12 (i\sigma \log x +\mu_0 -ir_0)= -\frac 1{2\im \sigma} (|\sigma|^2
\log|x| -r_0 \re\sigma -\mu_0 \im \sigma ).
$$
Hence $\Phi_1(x)$ has zeros $x_n,$ $n\in\N$ such that
$
|\sigma|^2\log|x_n| - r_0\re\sigma -\mu_0\im\sigma = -2\pi |\im \sigma|n
$
on $L(r_0,1)_{\sigma}.$ Under the supposition of Theorem \ref{thm2.6} the
ratio of $\Phi_1(x)$ to any of the other five functions is not a constant, and
then $F(x_n), \ldots, \Psi_2(x_n)$ other than $\Phi_1(x_n)$ are nonzero
numbers independent of $n$, since $x_n^{\sigma}=\xi_0$. 
Using Rouch\'{e}'s theorem, by the same 
argument as that of \cite[Section 2.2.2]{S2} we may prove the existence of the
sequence of zeros $\{x^0_n\}_{n\in\N}$ in Theorem \ref{thm2.6}. The other
sequence $\{\hat{x}_n^0\}_{n\in\N}$ is obtained from zeros of $\Phi_2(x).$
Theorem \ref{thm2.7} is proved by using $\Psi_1(x)$ and $\Psi_2(x)$.
Theorem \ref{thm2.10} is also proved by using \eqref{5.7} 
by the same argument as above.
Theorem \ref{thm2.8} follows from the facts that $\Phi_1(x)/\Phi_2(x)$ is
a constant if $\theta_0-\theta_x-\theta_{\infty}=0$ and that then every
zero of a solution of (V) is double, and from analogous facts on poles.
\par
To prove Theorem \ref{thm2.9} we use the B\"{a}cklund transformation.
Let $y$ and $\hat{y}=\hat{B}(y)$ be as in Lemma \ref{lem6.1}. Then
$$
\frac{1}{\hat{y}} -1 = \frac{2 xy}{2(y-1)^2(A_x)_{11} +\theta_x (y^2-1)}.
$$
For $x\in D_{\mathrm{even}}(\sigma,\rho, 0)$ substitute
\begin{equation}\label{7.0}
y=y(\sigma, \rho, x)= \frac{\Phi_1(x)\Phi_2(x)F(x)G(x) +O(x) }
{\Psi_1(x)\Psi_2(x)F(x)G(x) +O(x) }
\end{equation}
into the right-hand side. Observing that
$$
(A_x)_{11} +\frac{\theta_x} 2 =\frac 1{8\sigma^2} (Y_{11})_{\mathrm{den}}
= - \frac 1{8\sigma^2} (G(x) \Psi_1(x) +O(x) )
$$
with $(Y_{11})_{\mathrm{den}}$ as in \eqref{5.5}, that
$\Phi_1(x)\Phi_2(x) -\Psi_1(x)\Psi_2(x) = 4\sigma^2(\sigma^2
-(\theta_0 +\theta_x)^2) ,$ and that $(\sigma^2-(\theta_0+\theta_x)^2) G(x)
-2\theta_x \Psi_2(x)=\Phi_2(x)$,
we obtain
\begin{equation}\label{7.1}
1-\frac 1{\hat{y}} = \frac{x(\Phi_1(x) \Phi_2(x) \Psi_1(x) \Psi_2(x)
F(x)^2 G(x)^2 +O(x) ) }
{2\sigma^2(\sigma^2 -(\theta_0+\theta_x)^2) \Phi_2(x) \Psi_1(x) 
F(x)^2 G(x)^2 +O(x) } .
\end{equation}
Under the supposition
\begin{align}\label{7.2}
&\theta_0 (\theta_0+\theta_x -\theta_{\infty})
((\theta_0-\theta_x)^2-\theta_{\infty}^2)(\sigma^2-\theta_{\infty}^2) 
\\
\notag
&\phantom{----}
\times
(\sigma^2-(\theta_0 \pm \theta_x)^2) (\theta_0^2 -\theta_x^2 +\sigma^2
- 2\theta_{\infty}\theta_0 )\not=0,
\end{align}
which is a part of those of Theorems \ref{thm2.6} and \ref{thm2.7},
the ratio of $\Phi_1(x)$ or of $\Psi_2(x)$ to any of the other five functions is
not a constant and its zeros have the same property as that of $x_n$ above. 
Setting $(\sigma, \rho)=(\hat{\sigma}, \hat{\rho})$ 
in \eqref{7.1} and \eqref{7.2},
we apply $\pi$ of Lemma \ref{lem6.1} to them. 
Then \eqref{7.2} becomes \eqref{2.6.a} with 
$\sigma=\hat{\sigma}-1$.
From the result of the application of $\pi$ to \eqref{7.1} with $(\hat{\sigma},
\hat{\rho})$ we conclude that 
$
y^*(\hat{\sigma},\hat{\rho}, x) = B({y}^{\pi}(\hat{\sigma}, \hat{\rho}, x)) 
= \hat{y}^{\pi}(\hat{\sigma}, \hat{\rho}, x) 
$
solving (V) has $1$-points near the zeros of 
$\Phi^{\pi}_1(x)$ and $\Psi^{\pi}_2(x).$ 
By \eqref{6.1} with $\nu=0$ we have $y(\sigma,\rho, x) \equiv y^*(\hat{\sigma},
\hat{\rho}, x)$ if 
$$
\hat{\sigma} =\sigma+1, \quad \hat{\rho}^{-1}= -(\theta_{\infty}+
\sigma) \tilde{c}(\sigma +1) c(-\sigma) \rho^{-1}/2.
$$ 
Observing that $\Phi^{\pi}_1(x)$ (respectively,
$\Psi^{\pi}_2(x)$) vanishes if 
$$
x^{\hat{\sigma}}=\xi_1 = -\frac{\hat{\sigma} +1 -\theta_0 +\theta_x}  
{\hat{\sigma} -1 +\theta_0 -\theta_x} \hat{\rho}^{-1} \,\,\,\,
\quad \Bigl(\text{respectively,} \,\,\, 
x^{\hat{\sigma}}=\hat{\xi}_1 = -  \frac{\hat{\sigma}-1 +\theta_0 +\theta_x}  
{\hat{\sigma}+1 -\theta_0 -\theta_x } \xi_1 \, \Bigr),
$$ 
and that $L(r_1,1)_{\hat{\sigma}} =L(r_1, 1)_{\sigma +1} = L(r_1,0)_{\sigma}$
with $r_1=\log|\xi_1|,$ we obtain Theorem \ref{thm2.9}.
By the expressions of $y_{\sigma_0}(\rho, x)$ in $\Omega^-(\sigma_0,\varepsilon
_0)$ of Theorem \ref{thm2.2}, 
$
y_{\sigma_0}(\rho, x) \sim 1+ c^*(\sigma_0) \rho^{-1} x^{-\sigma_0}$
on $L(r_0,\omega)_{\sigma_0}$ with $0<\omega <1$ in each case, $r_0$ being
a fixed number. On the other hand
$$
y^*(\sigma_0+1, \hat{\rho}, x) 
=y^*_+(\sigma_0+1, \hat{\rho}, x) 
\sim 1- \frac{2\hat{\rho}^{-1} x^{-\sigma_0} }
{(\theta_{\infty} +\sigma_0) \tilde{c}(\sigma_0 +1) }
$$
on $L(r_0,\omega+1)_{\sigma_0+1}= L(r_0,\omega)_{\sigma_0}$ with $0<\omega <1$ 
(cf. Section \ref{sc6}). Hence
$y_{\sigma_0}(\rho, x) \equiv y^*(\sigma_0
+1, \hat{\rho}, x)$ with $-(\theta_{\infty}+\sigma_0) c^*(\sigma_0) \tilde{c}
(\sigma_0 +1) \hat{\rho} =2\rho.$ 
Using this fact we may show Theorem \ref{thm2.11}.
\subsection{Proof of Theorem \ref{thm2.1.a}}\label{ssc7.2}
If $\theta_0-\theta_x=\theta_{\infty}=0$, then by \cite[Theorem 5.6]{S}
equation (V) admits a family of solutions given by
$$
 \tanh^2 \Bigl( \Bigl(\frac 12 -\tilde{\sigma} \Bigr)\log x + \frac 12 \log
  \tilde{\rho} + s_{\mathrm{V}}(\tilde{\sigma}, x, \tilde{\rho}^{-1/2}
x^{\tilde{\sigma}}, \tilde{\rho}^{1/2} x^{1-\tilde{\sigma}}) \Bigr) 
$$
with $ \tilde{\sigma} \in \tilde{\Sigma} \subset \C\setminus (\{ \tilde
{\sigma} \le 0\}\cup \{\tilde{\sigma} \ge 1\}),$ $ \tilde{\rho}\in
\C\setminus \{ 0\} $, $\tilde{\Sigma}$ being a bounded domain such 
that $\mathrm{dist} (\tilde{\Sigma}, \{ \tilde
{\sigma} \le 0\}\cup \{\tilde{\sigma} \ge 1\}) >0.$ 
The series
$$
  s_{\mathrm{V}}(\tilde{\sigma}, x, \xi, \eta) = \sum_{i \ge 1} c^0_i
(\tilde{\sigma}) x^i 
  + \sum_{i\ge 0, \, j \ge 1} c^{1}_{ij}(\tilde{\sigma}) x^i \xi^{2j} 
  + \sum_{i\ge 0, \, j \ge 1} c^{2}_{ij}(\tilde{\sigma}) x^i \eta^{2j} 
$$
with $c_i^0(\tilde{\sigma})$, $c_{ij}^1(\tilde{\sigma})$,  
$c_{ij}^2(\tilde{\sigma}) \in \Q[\theta_0](\tilde{\sigma})$ converges if
$|x|$, $|\xi|,$ $|\eta|$ are sufficiently small. Note that we have replaced 
$(\xi, \eta)$ in \cite[Theorem 5.6]{S} by $(\xi^2,\eta^2)$ since, under the
condition $\theta_0-\theta_x=\theta_{\infty}=0$, 
the local equation in \cite[\S 5]{S} equivalent to (V)
is written in terms of
$F(x,\xi, \eta)= (1/32)(\eta^2-\xi^2) (4(1-\theta_0 -\theta_x) +
(\eta^2+\xi^2) ).$
Putting $\sigma=1-2\tilde{\sigma} \in \Sigma_0$, we have
$$
 \tanh^2 \Bigl( \frac 12 \log (\tilde{\rho}x^{\sigma}) 
  + s_{\mathrm{V}}((1-{\sigma})/2, x, (\tilde{\rho}^{-1}
x^{1-{\sigma}})^{1/2}, (\tilde{\rho} x^{1+{\sigma}})^{1/2} ) 
 \Bigr), 
$$
which satisfies
$$
 =\begin{cases}
 1-4\tilde{\rho} x^{\sigma}(1+o(1) )
 & \quad \text{in $D_+(\sigma,\tilde{\rho},  0)$},
 \\[0.2cm]
 1-4(\tilde{\rho} x^{\sigma})^{-1}(1+o(1) ) & \quad 
 \text{in $D_-(\sigma,\tilde{\rho}, 0)$}.
 \end{cases}
$$
Since $( (\tilde{\rho}^{\pm 1} x^{1\pm \sigma})^{1/2} ) ^{2j} = x^j
(\tilde{\rho} x^{\sigma})^{\pm j}$, $s_{\mathrm{V}}(\cdots)$ has the form
$\sum_{n=1}^{\infty} x^n \sum_{j=-n}^{n} c^*_{nj}(\tilde{\rho} x^{\sigma})^j$.
Noting that $c(\sigma)= -4 (2\theta_0 -\sigma) (2\theta_0 +\sigma)^{-1},$
we compare the behaviour above with that of $y(\sigma, \rho, x)$ to obtain
\eqref{2.4}.
By Lemma \ref{lem6.1},
$
\hat{y}= \hat{B}(y) =1-{2y}/{(y' + y)}.
$
Substitution of \eqref{2.4} with $(\sigma,\rho) =(\hat{\sigma}, \hat{\rho} )$
yields
$$
\hat{y}(\hat{\sigma},\hat{\rho} , x) = 1- \frac
{2 x \sinh (\log ({\breve{\rho}_-} x^{\hat{\sigma}} ) +\Sigma_{\mathrm{V}}
(x) ) }
{2 \hat{\sigma} + 2x \Sigma_{\mathrm{V}}'(x)  +
 x \sinh (\log ({\breve{\rho}_-} x^{\hat{\sigma}} ) +\Sigma_{\mathrm{V}}
(x) ) }
$$
with ${\breve{\rho}_-} = (2\theta_0 -\hat{\sigma})
 (2\theta_0 +\hat{\sigma})^{-1} \hat{\rho},$ where
$$
  \Sigma_{\mathrm{V}}(x) = 2s_{\mathrm{V}}
((1-\hat{\sigma})/2, x, ({\breve{\rho}}^{-1}_-
x^{1-\hat{\sigma}})^{1/2}, ({\breve{\rho}}_-
 x^{1+\hat{\sigma}})^{1/2} ) 
$$
and $\Sigma'_{\mathrm{V}}(x)= (d/dx) \Sigma_{\mathrm{V}}(x).$ Apply $\pi$
to both sides. Observing that ${\breve{\rho}}:=  
({\breve{\rho}}_-)^{\pi} = (2 -\hat{\sigma})
 (2 +\hat{\sigma})^{-1} \hat{\rho},$ and that $y^*(\hat{\sigma}, \hat{\rho},
x) = \hat{y}^{\pi} (\hat{\sigma}, \hat{\rho}, x)$ coincides with $y(\sigma, 
\rho, x) $ if $\hat{\sigma}=\sigma+1$ and
$$
\hat{\rho}^{-1} = - \sigma \tilde{c}(\sigma+1) c(-\sigma) \frac{\rho^{-1}} 2
=  \frac{8\sigma (\sigma +1)^2 } {(\sigma +2)(2\theta_0 -\sigma)} \rho^{-1},
$$
we obtain \eqref{2.3.a}.

\section{Monodromy data for the isomonodromy deformation}\label{sc8}
Let $(A_0(x),A_x(x))$ be the solution in Proposition \ref{prop3.1} or
\ref{prop3.2} yielding each solution of (V) in Section \ref{sc2} (cf.
Proposition \ref{prop3.3} and Section \ref{sc5}).
The associated linear system \eqref{1.1} admits the matrix solution
$$
Y(\lambda,x) =(I+O(\lambda^{-1}) ) e^{(\lambda/2)J} \lambda^{-(\theta_{\infty}
/2)J} \quad\quad\quad
 \lambda\to \infty, \,\,\,  -\pi/2 < \arg \lambda < 3\pi/2
$$
having the isomonodromy property. 
Set $\tilde{\rho}=\rho_{*}^{1/\sigma}$ with
$$
\rho_* := \frac{(\theta_0 -\theta_x +\sigma)(\theta_0 +\theta_x-\sigma)}
{2\sigma (\sigma +\theta_{\infty}) } \rho
$$
if $(A_0(x), A_x(x))$ yields $y(\sigma,\rho, x)$ (cf. \eqref{5.4}), 
$\tilde{\rho}=\rho^{1/\sigma_0}$ if it yields
$y_{\sigma_0}(\rho, x)$, and $\tilde{\rho}=\rho$ if $\sigma=0.$  
Then by Proposition \ref{prop3.1} or \ref{prop3.2}, $A_0$ and $A_x$ satisfy 
$$
x^{-\Lambda} A_0 x^{\Lambda} \to \tilde{\Lambda}_0:= \tilde{\rho}^{\Lambda}
\Lambda_0 \tilde{\rho}^{-\Lambda}, \quad
x^{-\Lambda} A_x x^{\Lambda} \to \tilde{\Lambda}_x:= \tilde{\rho}^{\Lambda}
\Lambda_x \tilde{\rho}^{-\Lambda},
\quad
A_0 + A_x \to \Lambda
$$
as $x \to 0,$ the eigenvalues of $\Lambda$ being $\pm \sigma/2$ with $\sigma
\in (\C\setminus\Z)\cup \{0\}.$ 
\par
Let us apply the argument of \cite[\S 2]{Jimbo} to our case. By \cite{SMJ}
$$
\hat{Y}(\lambda) := \lim_{x\to 0} Y(\lambda, x) \quad \text{and} \quad
\tilde{Y}(\lambda) := \lim_{x\to 0} x^{-\Lambda} Y(x\lambda, x) 
$$
solve
\begin{align}\label{8.1}
&\frac{d\hat{Y}}{d\lambda} = \Bigl( \frac{\Lambda}{\lambda} + \frac J2 \Bigr)
\hat{Y}
\intertext{and}
\label{8.2}
&\frac{d\tilde{Y}}{d\lambda} = \Bigl( \frac{\tilde{\Lambda}_0}{\lambda}
+ \frac{\tilde{\Lambda}_x}{\lambda - 1} \Bigr)\tilde{Y},
\phantom{-------}
\end{align}
respectively.
Since $\sigma \in (\C\setminus \Z)\cup \{0\},$ we may choose a matrix solution
of \eqref{8.1} such that 
\begin{align*}
 \hat{Y}_0(\lambda) & = (I+O(\lambda^{-1}) ) e^{(\lambda/2)J} \lambda^{-(\theta
_{\infty}/2) J} & \quad &\lambda \to \infty,&        
\\
  & = (I+O(\lambda) ) \lambda^{\Lambda} C & \quad &\lambda \to 0,&        
\end{align*}
and of \eqref{8.2} such that
\begin{align}\label{8.4}
 \tilde{Y}_0(\lambda) & = (I+O(\lambda^{-1}) ) \lambda^{\Lambda}
 & \quad &\lambda \to \infty,&        
\\
\notag
  & = G_0(I+O(\lambda) ) \lambda^{(\theta_0/2)J} \lambda^{\Delta_0} C^{(0)} 
& \quad &\lambda \to 0,&        
\\
\notag
  & = G_x(I+O(\lambda -1))(\lambda-1)^{(\theta_x/2)J}(\lambda-1)^{\Delta_x} 
C^{(x)} 
& \quad &\lambda \to 1&        
\end{align}
for $0 <\arg\lambda <\pi,$ $0<\arg(\lambda -1) <\pi,$ where $C$, 
$C^{(0)},$ $C^{(x)},$
$G_0,$ $G_x$ are some invertible matrices, and, for each $\iota=0, x$, 
the matrix $\Delta_{\iota}$ equals 
$0$ if $\theta_{\iota}\not\in \Z,$ $\epsilon_{\iota}\Delta$ if
$\theta_{\iota} \in \N\cup \{0\},$ and $\epsilon_{\iota}\Delta_-$ if
$-\theta_{\iota} \in \N$ with $\epsilon_{\iota} \in \C.$ 
Then by the same argument as in the proof of \cite[Proposition 2.1]{Jimbo}
$$
\hat{Y}(\lambda) = \lim_{x\to 0} Y(\lambda, x) =\hat{Y}_0(\lambda), \quad
\tilde{Y}(\lambda)= \lim_{x\to 0} x^{-\Lambda} Y(x\lambda, x) 
  =\tilde{Y}_0(\lambda) C
$$
as long as $|x^{1+\sigma}|, |x^{1-\sigma}| \to 0,$ and 
\begin{align*} 
  Y(\lambda, x) & = G_0(x) (I+O(\lambda)) \lambda^{(\theta_0/2)J}
 \lambda^{\Delta_0} C^{(0)} C 
& \quad &\lambda \to 0,&        
\\
&= G_x(x) (I+O(\lambda -x))(\lambda-x)^{(\theta_x/2)J}(\lambda-x)^{\Delta_x}
 C^{(x)} C & \quad &\lambda \to x,&        
\end{align*}
where $G_0(x)$ and $G_x(x)$ are invertible matrices. Therefore the 
monodromy matrices $M_0,$ $M_x$ defined in Section \ref{ssc2.3} are written as
follows:
\begin{equation}\label{8.7}
M_{\iota} = \begin{cases}
(C^{(\iota)} C)^{-1} e^{\pi i \theta_{\iota} J} C^{(\iota)} C  \quad &
\theta_{\iota} \not\in \Z,
\\
(-1)^{\theta_{\iota}} (C^{(\iota)} C)^{-1} e^{2\pi i \Delta_{\iota}}
 C^{(\iota)} C  \quad &
\theta_{\iota} \in \Z
\end{cases}
\end{equation}
$(\iota=0,x).$ Monodromy data for each solution will be computed by using
this fact.
\begin{rem} \label{rem8.1}
The matrices $C^{(\iota)}$ $(\iota=0, x)$ are not determined uniquely. Indeed,
say, if $\theta_{\iota}
\not\in \Z,$ we may take $G_{\iota}(x) D_0^{-1},$ $D_0C^{(\iota)}$ with any
invertible diagonal matrix $D_0$ instead of $G_{\iota}(x)$, $C^{(\iota)}.$
\end{rem}
\begin{rem}\label{rem8.2}
In our setting of isomonodromy deformation for system \eqref{1.1} with $x\not
=0$, 
the monodromy data and $y$ are the same as those in \cite{Andreev-Kitaev}, and
the unknown variables corresponding to $z=z_{\mathrm{AK}}$ and 
$u=u_{\mathrm{AK}}$
in \cite{Andreev-Kitaev} are $z_{\mathrm{AK}}$ and
$x^{-\theta_{\infty}}u_{\mathrm{AK}}$, respectively. 
\end{rem}
\section{Connection formulas for the Whittaker and the hypergeometric
systems}\label{sc9}
\subsection{Whittaker system}\label{ssc9.1}
Since $(A_0(x), A_x(x))$ mentioned in Section \ref{sc8} is obtained by using
$\Lambda_0,$ $\Lambda_x$, $T$, $\Lambda$ in Lemma \ref{lem4.1} or \ref{lem4.2},
our concern is \eqref{8.1} corresponding to such matrices.
The connection matrix $C$ for \eqref{8.1} is given by the following.
\begin{prop}\label{prop9.1}
$(1)$ Suppose that $\sigma\not\in \Z.$ For $\Lambda,$ $T$ as in Lemma
$\ref{lem4.1}$, $C= T C_{\infty}$ with 
$$
C_{\infty} = \begin{pmatrix}
-\dfrac{e^{-\pi i(\sigma+\theta_{\infty})/2} \Gamma(-\sigma)}
{\Gamma(1-(\sigma-\theta_{\infty})/2)}  &
-\dfrac{\Gamma(-\sigma)}
{\Gamma(1-(\sigma+\theta_{\infty})/2)}
  \\[0.3cm]
-\dfrac{e^{\pi i(\sigma-\theta_{\infty})/2} \Gamma(\sigma)}
{\Gamma((\sigma+\theta_{\infty})/2)}  &
\dfrac{\Gamma(\sigma)}
{\Gamma((\sigma-\theta_{\infty})/2)} 
\end{pmatrix}.
$$
\par
$(2)$ Suppose that $\sigma=0.$ For $\Lambda,$ $T$ as in Lemma
$\ref{lem4.2}$, $C=T C_{\infty}$ with $C_{\infty}$ given as follows$:$
\par
$(\mathrm{i})$ if $\theta_{\infty}\not=0,$
$$
C_{\infty} =  \begin{pmatrix}
\dfrac{e^{-\pi i\theta_{\infty}/2}(\psi(1+\theta_{\infty}/2) -2\psi(1) -\pi i)}
{\Gamma(1 + \theta_{\infty}/2)}
  &
\dfrac{ \psi(-\theta_{\infty}/2) -2\psi(1)}
{\Gamma(1 - \theta_{\infty}/2)}
 \\[0.3cm] 
\dfrac{e^{-\pi i\theta_{\infty}/2} }
{\Gamma(1+ \theta_{\infty}/2)}   &
 \dfrac{1 }
{\Gamma(1 -\theta_{\infty}/2)}
\end{pmatrix} ;
$$
\par
$(\mathrm{ii})$ if $\theta_{\infty}=0$ and $\Lambda=\Delta,$
then $C_{\infty} =I- \psi(1) \Delta;$
\par
$(\mathrm{iii})$ if $\theta_{\infty}=0$ and $\Lambda=\Delta_-,$
then $C_{\infty} =(1- \pi i- \psi(1))(I+J)/2 +\Delta + \Delta_-.$
\end{prop}
\begin{proof}
If $(\sigma, \theta_{\infty})\not=(0,0),$ then
system \eqref{8.1} with $\Lambda$ as in Lemma \ref{lem4.1} or \ref{lem4.2}
has the matrix solution
$\hat{Y}(\lambda) $ given by 
$$ 
\begin{pmatrix}
e^{\pi i (1-\theta_{\infty})/2} W_{(1-\theta_{\infty})/2,\sigma/2}(e^{-\pi i}
\lambda) &
-\frac 12(\sigma-\theta_{\infty}) W_{(-1+\theta_{\infty})/2,\sigma/2}(\lambda) 
\\
-\frac 12(\sigma+\theta_{\infty})  
e^{\pi i (1-\theta_{\infty})/2} W_{-(1+\theta_{\infty})/2,\sigma/2}(e^{-\pi i}
\lambda) &
 W_{(1+\theta_{\infty})/2,\sigma/2}(\lambda) 
\end{pmatrix}
\lambda^{-1/2}
$$
(cf. \cite[(3.10)]{Jimbo}), which behaves as 
$$
\hat{Y}(\lambda) =(I+O(\lambda^{-1}))e^{(\lambda/2) J} 
 \lambda^{-(\theta_{\infty}/2)J}
\qquad\quad \lambda \to \infty, \,\,\,
0<\arg \lambda <\pi.
$$
Here $W_{\kappa,\mu}(z)$ is the Whittaker function such that $W_{\kappa, \mu}
(z) \sim e^{-z/2} z^{\kappa},$ $|\arg z|<\pi.$ If $\sigma\not\in \Z,$ we have,
around $\lambda=0,$
$$
\hat{Y}(\lambda) = (I+O(\lambda)) \lambda^{\Lambda}C
= T(I+O(\lambda) ) \lambda^{(\sigma/2)J}T^{-1}C,
$$
where $T$ is as in Lemma \ref{lem4.1}. Using the connection formula
$$
e^{z/2} z^{-1/2}W_{\kappa,\sigma/2}(z) =\frac{\Gamma(-\sigma) z^{\sigma/2} }
{\Gamma((1-\sigma)/2 -\kappa)} (1+O(z))
 + \frac{\Gamma(\sigma) z^{-\sigma/2} }
{\Gamma((1+\sigma)/2 -\kappa)} (1+O(z))
$$
(cf. \cite[13.1.3, 13.1.33]{AS}), 
we compare the behaviours around $\lambda=0$ to
derive $C$ as in (1). If $\sigma=0$ and $\theta_{\infty}\not=0,$ then, around
$\lambda=0,$ 
$$
\hat{Y}(\lambda)=(I+O(\lambda))\lambda^{\Lambda} C 
=T(I+O(\lambda) )\lambda^{\Delta}T^{-1}C,
$$
where $T$ and $\Lambda$ are as in Lemma \ref{lem4.2}. Note that (1,1)- and
(1,2)-entries of $T(I+O(\lambda))\lambda^{\Delta}$ are $-\theta_{\infty}/2+O
(\lambda)$ and $-(\theta_{\infty}/2)\log\lambda +1+O(\lambda),$ respectively. 
On the other hand, using
$$
e^{z/2} z^{-1/2} W_{\kappa,0}(z) = - \frac{1}{\Gamma(1/2-\kappa)}
\Bigl( (1+O(z))\log z + \psi(1/2-\kappa) -2\psi(1) +O(z)\Bigr )
$$
(cf. \cite[13.1.6]{AS}), we can also see how  
$\hat{Y}(\lambda)$ behaves around $\lambda=0$. Comparison of these
leads us to $C$ as in (i) of (2). Under the supposition $\sigma=
\theta_{\infty}=0,$ if $\Lambda=\Delta$ (respectively, $\Lambda=\Delta_-$),
then \eqref{8.1} admits the matrix solution
$$
\begin{pmatrix}
e^{\lambda/2} & -\lambda^{-1/2} W_{-1/2,0}(\lambda) 
\\
0 & e^{-\lambda/2}
\end{pmatrix}
\quad \Biggl(\text{respectively,} \,\,\,
\begin{pmatrix}
e^{\lambda/2} &  0
\\
e^{-\pi i/2} \lambda^{-1/2} W_{-1/2,0}( e^{-\pi i} \lambda) 
 & e^{-\lambda/2}
\end{pmatrix}
\,\, \Biggr).
$$
Using the connection formula above together with $T$ in each case, we 
find the matrices in (ii) and (iii).
\end{proof}
\subsection{Hypergeometric system}\label{ssc9.2}
Let us begin with 
$$
\frac{d\mathbf{u}}{dz} = \Biggl( \frac 1z \begin{pmatrix} 0 & 1 \\ 0 & 1-\gamma
\end{pmatrix}  + 
 \frac 1{z-1} \begin{pmatrix} 0 & 0 \\ -\alpha\beta & \gamma-\alpha-\beta -1
\end{pmatrix} \Biggr)  \mathbf{u}, \quad  \mathbf{u} =\begin{pmatrix}
 u  \\  zu'  \end{pmatrix},
$$
in which $u$ solves the hypergeometric equation
$z(1-z)u''+(\gamma -(\alpha+\beta +1)z)u' -\alpha\beta u=0$.
The eigenvalues of the residue matrices are $0, 1-\gamma$ at $z=0,\,$
$0, \gamma-\alpha-\beta -1$ at $z=1$ and $\alpha, \beta$ at $z=\infty.$
Under the supposition $\alpha-\beta\not\in \Z$, diagonalising the residue
matrix at $z=\infty$ and shifting the eigenvalues to $\pm(1-\gamma)/2,$
$\pm(\gamma-\alpha-\beta-1)/2$, $\pm(\beta-\alpha)/2$, we obtain the system
\begin{gather}\label{9.1}
\frac{d\Psi}{dz} = \Bigl( \frac{B_0} z +\frac{B_1}{z-1} \Bigr) \Psi,
\phantom{-------}
\\
\notag
\begin{split}
& B_0=\frac 1{\alpha-\beta} \begin{pmatrix}
(1-\gamma)(\alpha+\beta)/2 +\alpha\beta   &   \beta(\beta -\gamma +1)/R 
\\
\alpha(\gamma-\alpha -1)R     & 
- (1-\gamma)(\alpha+\beta)/2 -\alpha\beta    
\end{pmatrix},
\\
& B_1=\frac 1{\alpha-\beta} \begin{pmatrix}
(\gamma-\alpha -\beta-1)(\alpha+\beta)/2 +\alpha\beta & 
\beta(\gamma -\beta  -1)/R 
\\
\alpha(\alpha - \gamma +1)R     & 
- (\gamma-\alpha -\beta -1)(\alpha+\beta)/2 -\alpha\beta    
\end{pmatrix}
\end{split}
\end{gather} 
with given $R \not=0,$ which has the following property.
\begin{prop}\label{prop9.2}
Suppose that $\alpha-\beta\not\in\Z.$ System \eqref{9.1} admits the matrix
solution 
$$
\Psi(z) = z^{-(1-\gamma)/2} (z-1)^{-(\gamma-\alpha-\beta-1)/2} 
\diag [ 1, R] \, [\mathbf{v}_{\alpha} \,\, 
\mathbf{v}_{\beta} ] \, \diag[ 1, 1/ R]  
$$
with
\begin{align*}
& [\mathbf{v}_{\alpha} \,\, \mathbf{v}_{\beta} ]
 = P^{-1} \begin{pmatrix} u_{\alpha} & u_{\beta}
\\
zu'_{\alpha} & zu'_{\beta}  \end{pmatrix}, \quad P = 
 \begin{pmatrix} 1 &  1   \\
- \alpha & - \beta  \end{pmatrix},  
\\
&u_{\alpha} := z^{-\alpha} F(\alpha, \alpha-\gamma +1, \alpha-\beta +1, z^{-1}),
\\
&u_{\beta} := z^{-\beta} F(\beta, \beta-\gamma +1, \beta-\alpha +1, z^{-1}),
\end{align*}
which satisfies
$$
\Psi(z) = (I+ O(z^{-1})) z^{(\beta-\alpha)J/2}   \qquad\quad z\to \infty.
$$
\end{prop}
We substitute the connection formulas representing $u_{\alpha}$ and
$u_{\beta}$ as linear combinations of hypergeometric series around $z=0$ or
$z=1$ 
(cf. \cite[\S 2.9]{HTF}) and choose gauge matrices suitably to obtain 
\begin{prop}\label{prop9.3}
Suppose that $\alpha-\beta,$ $\gamma,$ $\gamma-\alpha -\beta \not\in \Z.$
Then
\begin{align*}
\Psi(z)|_{R=1} & = G_0 (I+O(z) ) z^{(1-\gamma) J/2} \tilde{C}_0   & \quad
&z\to 0,
\\
 & = G_1 (I+O(z-1) )( z-1)^{(\gamma-\alpha -\beta -1) J/2} \tilde{C}_1 
  & \quad &z\to 1
\end{align*}
for $|\arg z- \pi | <\pi,$ $|\arg(z-1) -\pi |<\pi,$ where
\begin{align*}
& \tilde{C}_0 = \begin{pmatrix}
\dfrac{e^{\pi i(\gamma-\alpha-1)} \Gamma(1-\beta +\alpha) \Gamma(\gamma-1)}
{\Gamma(\alpha) \Gamma(\gamma-\beta) }     &
\dfrac{e^{\pi i(\gamma-\beta-1)} \Gamma(1-\alpha +\beta) \Gamma(\gamma-1)}
{\Gamma(\beta) \Gamma(\gamma-\alpha) }     
\\[0.4cm]
\dfrac{e^{- \pi i\alpha} \Gamma(1-\beta +\alpha) \Gamma(1-\gamma)}
{\Gamma(1-\gamma +\alpha) \Gamma(1-\beta) }     &
\dfrac{e^{- \pi i\beta} \Gamma(1-\alpha+\beta) \Gamma(1-\gamma)}
{\Gamma(1-\gamma +\beta) \Gamma(1-\alpha) } 
\end{pmatrix},
\\[0.2cm]
& \tilde{C}_1 = \begin{pmatrix}
\dfrac{ \Gamma(1-\beta +\alpha) \Gamma(\alpha+\beta-\gamma)}
{\Gamma(1-\gamma +\alpha) \Gamma(\alpha) }     &
\dfrac{ \Gamma(1-\alpha +\beta) \Gamma(\alpha+\beta-\gamma)}
{\Gamma(1-\gamma +\beta) \Gamma(\beta) }    
\\[0.4cm]
\dfrac{ \Gamma(1-\beta +\alpha) \Gamma(\gamma- \alpha-\beta)}
{\Gamma(1-\beta) \Gamma(\gamma-\beta) }     &
\dfrac{ \Gamma(1-\alpha +\beta) \Gamma(\gamma- \alpha-\beta)}
{\Gamma(1-\alpha) \Gamma(\gamma-\alpha) }    
\end{pmatrix}.
\end{align*}
\end{prop}
In the case where $\alpha=\beta$, we transform the system in such a way
that the residue matrix at $z=\infty$ becomes $\Delta$. The result of this
procedure is
\begin{align}\label{9.2}
& \frac{d\Psi}{dz} = \Bigl( \frac{B_0} z +\frac{B_1}{z-1} \Bigr) \Psi,
\\
\notag
\begin{split}
 B_0=&\begin{pmatrix}
(\gamma-2\alpha -1)/2  &  1 
\\
\alpha(\gamma-\alpha -1)  &  -(\gamma-2\alpha -1)/2  
\end{pmatrix},
\\
 B_1=& \begin{pmatrix}
(2\alpha -\gamma+1)/2  &   0 
\\
\alpha(\alpha - \gamma +1)   & -(2\alpha -\gamma+1)/2  
\end{pmatrix}.
\end{split}
\end{align} 
Write, for $m \in \N,$
\begin{align*}
 F_{\mathrm{log}}(\alpha, \beta ,m, z)
& := F(\alpha, \beta,m, z)\log z
\\
&\phantom{---} +\sum_{n=1}^{\infty} \frac{(\alpha)_n(\beta)_n}{(m)_n n!} 
\psi_n(\alpha, \beta, m) z^n - \sum_{n=1}^{m-1}
\frac{(n-1)!(1-m)_n}{(1-\alpha)_n(1-\beta)_n} z^{-n}
\end{align*}
with
\begin{align*}
\psi_n(\alpha, \beta, m) &=\psi(\alpha+n) -\psi(\alpha)
+\psi(\beta +n) -\psi(\beta)
\\
&\phantom{----}
 -\psi(m+n)+\psi(m) -\psi(1+n)+\psi(1).
\end{align*}
Then $z^{-\alpha} F_{\mathrm{log}}(\alpha,\alpha-\gamma+1, 1, z^{-1})$ is
a logarithmic hypergeometric function 
near $z=\infty$, which satisfies, for $|\arg z-\pi| <\pi,$
\begin{align}\label{9.3}
& - z^{-\alpha} F_{\mathrm{log} } (\alpha, \alpha -\gamma +1, 1, z^{-1})
 = \frac{\Gamma(\alpha) \Gamma(\gamma-\alpha)} {\Gamma(\gamma)} 
e^{-\pi i \alpha} F(\alpha, \alpha, \gamma, z)
\\
\notag
&\phantom{----} - (2\psi(1) -\psi(\alpha)
-\psi(\gamma-\alpha) -\pi i) z^{-\alpha} F(\alpha, \alpha-\gamma+1, 1, z^{-1})
\end{align}
(cf. \cite[15.3.13]{AS}, \cite[\S 2.10]{HTF}). Then we have 
\begin{prop}\label{prop9.4}
System \eqref{9.2} admits the matrix solution
$$
\Psi(z) = z^{-(1-\gamma)/2} (z-1)^{-(\gamma- 2\alpha -1)/2} 
[\mathbf{v}_{\alpha} \,\, \mathbf{v}_{\alpha\, \mathrm{log}} ]
$$
with
\begin{align*}
& [\mathbf{v}_{\alpha} \,\, \mathbf{v}_{\alpha\,\mathrm{log}} ] = P^{-1}
 \begin{pmatrix} u_{\alpha} & u_{\alpha\,\mathrm{log}}
\\
zu'_{\alpha} & zu'_{\alpha\,\mathrm{log}}  \end{pmatrix}, \quad P = 
 \begin{pmatrix} 1 &  0   \\
- \alpha & 1  \end{pmatrix},  
\\
&u_{\alpha} := z^{-\alpha} F(\alpha, \alpha-\gamma +1, 1, z^{-1}),
\\
&u_{\alpha\,\mathrm{log}} := - z^{-\alpha} F_{\mathrm{log}}(\alpha, 
\alpha-\gamma +1, 1, z^{-1}),
\end{align*}
and $\Psi(z)$ satisfies
$$
\Psi(z) = (I+ O(z^{-1})) z^{\Delta}   \qquad\quad z\to \infty.
$$
\end{prop}
In \cite[\S 2.10, (3)]{HTF} replacing $z$ by $1-z$ and applying limiting
procedure, we get, for $|\arg z|<\pi,$
\begin{align}\label{9.4}
& F(\alpha, \alpha, 2\alpha-\gamma+1, 1-z) =
\frac{\Gamma(2\alpha -\gamma+1)}{\Gamma(\alpha) \Gamma(\alpha-\gamma+1)}
 z^{-\alpha} \Bigl( - F_{\mathrm{log} } (\alpha, \alpha -\gamma +1, 1, z^{-1})
\\
\notag
&\phantom{----} - (\psi(\alpha) +\psi(\alpha-\gamma+1) - 2\psi(1) ) 
F(\alpha, \alpha-\gamma+1, 1, z^{-1}) \Bigr).
\end{align}
Using \eqref{9.3} and this relation, we have the following:
\par
(i) if $\gamma\not\in \Z,$ then, for $|\arg z -\pi |<\pi,$
\begin{align*}
& z^{-\alpha} F(\alpha, \alpha-\gamma +1, 1, z^{-1}) =\frac{e^{-\pi i\alpha}
\Gamma(1-\gamma)} {\Gamma(1-\alpha) \Gamma(\alpha-\gamma +1)} f_0(z)
 -\frac {e^{\pi i (\gamma-\alpha)} \Gamma(\gamma-1)}{\Gamma(\alpha)
\Gamma(\gamma-\alpha) }  g_0(z),
\\
 -& z^{-\alpha} F_{\mathrm{log}}(\alpha, \alpha-\gamma +1, 1, z^{-1}) 
\\
& = - \frac{ e^{-\pi i\alpha} (2\psi(1) -\psi(1-\alpha) -\psi(\alpha- \gamma
+1) -\pi i) \Gamma(1-\gamma)} {\Gamma(1-\alpha) \Gamma(\alpha-\gamma +1)}
f_0(z)
\\
& \phantom{---} + \frac{e^{\pi i(\gamma -\alpha)} (2\psi(1) -\psi(\alpha) -\psi(\gamma-\alpha)
-\pi i)\Gamma(\gamma-1)} {\Gamma(\alpha) \Gamma(\gamma-\alpha)} g_0(z) 
\end{align*}
with
$$
f_0(z)=F(\alpha, \alpha, \gamma, z), \quad g_0(z)=
z^{1-\gamma} F(\alpha-\gamma+1, \alpha-\gamma+1,2-\gamma, z);
$$
\par
(ii) if $\gamma-2\alpha \not\in \Z,$ then, for $|\arg z|<\pi$, 
$|\arg (1- z) |<\pi,$
\begin{align*}
& z^{-\alpha} F(\alpha, \alpha-\gamma +1, 1, z^{-1}) =\frac{
\Gamma(\gamma-2\alpha)} {\Gamma(1-\alpha) \Gamma(\gamma-\alpha)} f_1(z)
 +\frac {e^{\pi i (\gamma-2\alpha)} \Gamma(2\alpha- \gamma)}{\Gamma(\alpha)
\Gamma(\alpha-\gamma+1) }  g_1(z),
\\
 -& z^{-\alpha} F_{\mathrm{log}}(\alpha, \alpha-\gamma +1, 1, z^{-1}) 
\\
& = \frac{ (\psi(1-\alpha) +\psi(\gamma- \alpha) -2\psi(1))
 \Gamma(\gamma -2\alpha)} {\Gamma(1-\alpha) \Gamma(\gamma-\alpha)} f_1(z)
\\
& \phantom{---} + \frac{e^{\pi i(\gamma -2\alpha)} (\psi(\alpha)
 +\psi(\alpha-\gamma+1) -2\psi(1) )
\Gamma(2\alpha- \gamma)} {\Gamma(\alpha) \Gamma(\alpha-\gamma+1)} g_1(z) 
\end{align*}
with
\begin{align*}
&f_1(z)=F(\alpha, \alpha, 2\alpha-\gamma +1, 1-z), 
\\ 
&g_1(z)=
(1-z)^{\gamma-2\alpha} F(\gamma-\alpha, \gamma-\alpha,\gamma-2\alpha +1, 1-z).
\end{align*}
Then we have
\begin{prop}\label{prop9.5}
Suppose that $\alpha=\beta$ and that $\gamma,$ $\gamma-2\alpha \not\in \Z.$ Then
\begin{align*}
\Psi(z)& = G_0 (I+O(z) ) z^{(1-\gamma) J/2} \tilde{C}_0   & \quad
&z\to 0,
\\
 & = G_1 (I+O(z-1) )( z-1)^{(\gamma-2\alpha -1) J/2} \tilde{C}_1 
  & \quad &z\to 1
\end{align*}
for $0<\arg z< \pi,$ $|\arg(z-1) -\pi |<\pi,$ where
\begin{align*}
& \tilde{C}_0 = \begin{pmatrix}
- \dfrac{e^{\pi i(\gamma-\alpha)} \Gamma(\gamma-1) }
{\Gamma(\alpha) \Gamma(\gamma-\alpha) }     &
\dfrac{e^{\pi i(\gamma-\alpha)} \hat{\psi}^{0}_{12}(\alpha,\gamma) 
\Gamma(\gamma-1)}
{\Gamma(\alpha) \Gamma(\gamma-\alpha) }     
\\[0.4cm]
\dfrac{e^{- \pi i\alpha} \Gamma(1-\gamma)}
{\Gamma(1-\alpha) \Gamma(\alpha-\gamma+1) }     &
- \dfrac{e^{- \pi i\alpha} 
\hat{\psi}^{0}_{22}(\alpha, \gamma)
\Gamma(1-\gamma)}
{\Gamma(1-\alpha) \Gamma(\alpha-\gamma+1) } 
\end{pmatrix},
\\[0.2cm]
& \tilde{C}_1 = \begin{pmatrix}
\dfrac{ \Gamma(2\alpha-\gamma)}
{\Gamma(\alpha)\Gamma(\alpha-\gamma +1) }     &
\dfrac{
\hat{\psi}^{1}_{12}(\alpha, \gamma)
\Gamma(2\alpha-\gamma)}
{\Gamma(\alpha) \Gamma(\alpha-\gamma +1) }    
\\[0.4cm]
\dfrac{ \Gamma(\gamma-2 \alpha)}
{\Gamma(1-\alpha) \Gamma(\gamma-\alpha) }     &
\dfrac{
\hat{\psi}^{1}_{22}(\alpha, \gamma)
\Gamma(\gamma-2\alpha)}
{\Gamma(1-\alpha) \Gamma(\gamma-\alpha) }    
\end{pmatrix}
\end{align*}
with
\begin{align*}
&\hat{\psi}^{0}_{12}(\alpha, \gamma)=
2\psi(1) -\psi(\alpha) -\psi(\gamma-\alpha) -\pi i, 
\\
&\hat{\psi}^{0}_{22}(\alpha, \gamma)
=2\psi(1) -\psi(1-\alpha) -\psi(\alpha-\gamma +1)-\pi i, 
\\
&\hat{\psi}^{1}_{12}(\alpha, \gamma)=
\psi(\alpha) +\psi(\alpha-\gamma+1) -2\psi(1), 
\\
& \hat{\psi}^{1}_{22}(\alpha, \gamma)
 =\psi(1-\alpha) +\psi(\gamma-\alpha) -2\psi(1).
\end{align*}
\end{prop}
\subsection{Non-generic cases}\label{ssc9.3}
Suppose that $\alpha-\beta \not\in \Z.$ In the case where $\gamma \in \Z,$
the hypergeometric function behaves logarithmically around $z=0$. 
Under the condition $\alpha, \beta \not\in \Z,$
for $|\arg z-\pi|<\pi,$ application of limiting procedure to connection 
formulas in the generic case yields the following
(cf. \cite[\S 2.10]{HTF}, \cite[15.5.17, 15.5.19, 15.5.21]{AS}):
\par
(i) if $1- \gamma=l=0,1,2, \ldots,$
\begin{align*}
&z^{-\alpha} F(\alpha, \alpha+l, \alpha-\beta +1, z^{-1})
\\
& = a_-(\alpha, \beta, l) \left(z^l F_{\mathrm{log}}(\alpha+l, \beta+l ,1+l, z)
+b_-(\alpha, \beta, l) z^l F(\alpha+l, \beta+l, 1+l, z) \right ),
\end{align*}
where
\begin{align*}
& a_-(\alpha, \beta, l) = \frac{(-1)^{l-1} e^{-\pi i\alpha}\Gamma(\alpha-\beta +1)
\Gamma(\beta +l)} {\Gamma(\alpha)\Gamma(\beta) \Gamma(1-\beta) l!},
\\
& b_-(\alpha, \beta, l) = \psi(\alpha+l) +\psi(\beta+l) -\psi(1+l) -\psi(1)
-\psi(\beta) +\psi(1-\beta) -\pi i;
\end{align*}
\par
(ii) if $1- \gamma=l =-1, -2, -3, \ldots,$
\begin{align*}
&z^{-\alpha} F(\alpha, \alpha+l, \alpha-\beta +1, z^{-1})
\\
&\phantom{---}
 = a_+(\alpha, \beta, l) \left( F_{\mathrm{log}}(\alpha, \beta ,1-l, z)
+b_+(\alpha, \beta, l) F(\alpha, \beta, 1-l, z)\right),
\end{align*}
where
\begin{align*}
& a_+(\alpha, \beta, l) = - \frac{e^{-\pi i\alpha}\Gamma(\alpha-\beta +1)
\Gamma(\beta)} {\Gamma(\alpha+l)\Gamma(\beta+l) \Gamma(1-\beta-l) (-l)!},
\\
& b_+(\alpha, \beta, l) = \psi(\alpha) +\psi(\beta) -\psi(1-l) -\psi(1)
-\psi(\beta+l) +\psi(1-\beta-l) -\pi i.
\end{align*}
In these cases, $\Psi(z)|_{R=1}$ for \eqref{9.1} around $z=0$ has the form
$G_0(I+O(z)) E(z) \tilde{C}_0$, where $E(z)$ is $z^{lJ/2} z^{\Delta}$ if
$1- \gamma= l= 0, 1, 2, \ldots,$ and $z^{lJ/2} z^{\Delta_-}$ if $1- \gamma=l
=-1, -2, -3, \ldots,$ and hence we have
\begin{prop}\label{prop9.6}
Suppose that $\alpha,$ $\beta,$ $\alpha-\beta \not\in \Z$ 
and $1-\gamma=l \in \Z.$ For $|\arg z-\pi| <\pi,$
\begin{align*}
\tilde{C}_0 &= \begin{pmatrix}
b_-(\alpha, \beta, l) & b_-(\beta, \alpha, l) 
\\
1 & 1
\end{pmatrix}
 \begin{pmatrix}
a_-(\alpha, \beta, l) &  0 \\  0  &  a_-(\beta, \alpha, l) 
\end{pmatrix}  & \quad &\text{if} \,\,\, l =0,1,2, \ldots,
\\
&= \begin{pmatrix}
1 & 1 \\
b_+(\alpha, \beta, l) & b_+(\beta, \alpha, l) 
\end{pmatrix}
 \begin{pmatrix}
a_+(\alpha, \beta, l) &  0 \\  0  &  a_+(\beta, \alpha, l) 
\end{pmatrix}  & \quad &\text{if} \,\,\,  l =-1,-2, -3, \ldots.
\end{align*}
\end{prop}
From $F(\alpha, \beta, \gamma, z) =(1-z)^{-\alpha} F(\alpha, \gamma-\beta,
\gamma, z/(z-1) )$ \cite[15.3.4]{AS}, it follows that
$$
z^{-\alpha} F(\alpha, \alpha-\gamma+1, \alpha-\beta +1, z^{-1}) 
= e^{\pi i \alpha} \zeta^{-\alpha} F(\alpha, \gamma-\beta, \alpha-\beta +1,
\zeta^{-1} )
$$
for $|\arg(1-z) -\pi |<\pi,$ $|\arg z-\pi|<\pi,$ where $\zeta= 1-z =e^{\pi i}
(z-1).$ In the case where $\gamma-\alpha-\beta \in \Z,$ using this relation,
from the connection formulas above we immediately obtain the following:
\par
(i) if $\gamma-\alpha-\beta=l=0,1,2, \ldots,$
\begin{align*}
&z^{-\alpha} F(\alpha, 1-l-\beta, \alpha-\beta +1, z^{-1})
\\
& = e^{\pi i\alpha} a_-(\alpha, \beta, l)
 \Bigl((1-z)^l F_{\mathrm{log}}(\alpha+l, \beta+l ,1+l, 1-z)
\\
&\phantom{----------}
+b_-(\alpha, \beta, l) (1-z)^l F(\alpha+l, \beta+l, 1+l, 1-z) \Bigr );
\end{align*}
\par
(ii) if $\gamma- \alpha-\beta =l = -1, -2, -3, \ldots,$
\begin{align*}
&z^{-\alpha} F(\alpha, 1 -l -\beta, \alpha-\beta +1, z^{-1})
\\
& = e^{\pi i\alpha} a_+(\alpha, \beta, l)
 \left( F_{\mathrm{log}}(\alpha, \beta ,1-l, 1-z)
+b_+(\alpha, \beta, l) F(\alpha, \beta, 1-l, 1-z)\right).
\end{align*}
Observing that $(1-z)^{l J/2} (1-z)^{\Delta} =(-1)^l (I+O(z-1) )
(z-1)^{lJ/2}(z-1)^{\Delta},$ we have
\begin{prop}\label{prop9.7}
Suppose that $\alpha,$ $\beta,$ $\alpha-\beta \not\in \Z$ and $\gamma -\alpha
-\beta =l \in \Z.$ 
For $|\arg z-\pi| <\pi,$ $|\arg(z-1)|<\pi,$
\begin{align*}
\tilde{C}_1 
=& e^{\pi i\alpha} \begin{pmatrix}
b_-(\alpha, \beta, l) & b_-(\beta, \alpha, l) 
\\
1 & 1
\end{pmatrix}
 \begin{pmatrix}
a_-(\alpha, \beta, l) &  0 \\  0  & 
 a_-(\beta, \alpha, l) 
\end{pmatrix}  & \,\, &\text{if} \,\, l=0,1,2, \ldots,
\\
=& e^{\pi i\alpha} \begin{pmatrix}
1 & 1 \\
b_+(\alpha, \beta, l) & b_+(\beta, \alpha, l) 
\end{pmatrix}
 \begin{pmatrix}
 a_+(\alpha, \beta, l) &  0 \\  0  & 
 a_+(\beta, \alpha, l) 
\end{pmatrix}  & \,\, &\text{if} \,\, l =-1, -2, -3, \ldots.
\end{align*}
\end{prop}
For $\alpha\not\in \Z,$ putting $\alpha=\beta$, we have the following:
\par
(i) if $1- \gamma=l =0,1,2, \ldots,$
\begin{align*}
&z^{-\alpha} F(\alpha, \alpha+l, 1, z^{-1})
\\
& = a_-(\alpha, \alpha, l) 
\left(z^l F_{\mathrm{log}}(\alpha+l, \alpha+l ,1+l, z)
+b_-(\alpha, \alpha, l) z^l F(\alpha+l, \alpha+l, 1+l, z) \right );
\end{align*}
\par
(ii) if $1- \gamma=l = -1, -2, -3, \ldots,$
\begin{align*}
&z^{-\alpha} F(\alpha, \alpha+l, 1, z^{-1})
\\
& = a_+(\alpha, \alpha, l)
 \left( F_{\mathrm{log}}(\alpha, \alpha ,1-l, z)
+b_+(\alpha, \alpha, l) F(\alpha, \alpha, 1-l, z)\right).\phantom{-------}
\end{align*}
Combining these formulas with \eqref{9.3}, 
we have another relation in the case where
$\alpha=\beta\not\in \Z,$ $\gamma \in \Z.$ For example, if $1- \gamma= l 
=0, 1, 2, \ldots,$ observing that
$$
\frac 1{\Gamma(\gamma)} F(\alpha, \alpha, \gamma, z) \Bigl |_{\gamma=-l+1}
= \frac{\Gamma(\alpha+ l)^2}{\Gamma(\alpha)^2 l!} z^l F(\alpha +l, \alpha +l,
1+l, z),
$$
we get
\begin{align*}
- z^{-\alpha} F_{\mathrm{log}} (\alpha, & \alpha +l, 1, z^{-1}) 
 = \frac{\Gamma(\alpha+ l)^2 \Gamma(1-\alpha-l)}{\Gamma(\alpha) l!} 
e^{-\pi i\alpha} z^l F(\alpha +l, \alpha +l, 1+l, z)
\\
& - (2\psi(1) -\psi(\alpha) -\psi(1- \alpha -l) -\pi i) 
a_-(\alpha, \alpha, l)
\\
& \times \Bigl( z^l F_{\mathrm{log}} (\alpha+l, \alpha+l, 1+l, z)
+ b_-(\alpha, \alpha, l) z^l F(\alpha+l, \alpha+l, 1+l, z)   \Bigr)
\end{align*}
for $|\arg z- \pi | <\pi.$ On the other hand, in the case where 
$\gamma- 2\alpha \in \Z,$ we have the relations:
\par
(i) if $\gamma-2\alpha=l=0,1,2, \ldots,$
\begin{align*}
&z^{-\alpha} F(\alpha, 1-l-\alpha, 1, z^{-1})
\\
& = e^{\pi i\alpha} a_-(\alpha, \alpha, l)
 \Bigl((1-z)^l F_{\mathrm{log}}(\alpha+l, \alpha+l ,1+l, 1-z)
\\
&\phantom{----------}
+b_-(\alpha, \alpha, l) (1-z)^l F(\alpha+l, \alpha+l, 1+l, 1-z) \Bigr );
\end{align*}
\par
(ii) if $\gamma- 2\alpha =l = -1, -2, -3, \ldots,$
\begin{align*}
&z^{-\alpha} F(\alpha, 1 -l -\alpha, 1, z^{-1})
\\
& = e^{\pi i\alpha} a_+(\alpha, \alpha, l)
 \left( F_{\mathrm{log}}(\alpha, \alpha ,1-l, 1-z)
+b_+(\alpha, \alpha, l) F(\alpha, \alpha, 1-l, 1-z)\right).
\end{align*}
The connection formula for
$ -z^{-\alpha} F_{\mathrm{log}}(\alpha, 1 -l -\alpha, 1, z^{-1})$ is
obtained by an analogous argument, in which we use \eqref{9.4}.
Thus we have
\begin{prop}\label{prop9.8}
$(1)$ Suppose that $\alpha=\beta \not\in \Z$ and $1-\gamma=l \in \Z.$ 
For $|\arg z-\pi| <\pi,$
\begin{align*}
\tilde{C}_0 =& \begin{pmatrix}
b_-(\alpha, \alpha, l) & b_-(\alpha, \alpha, l) 
\\
1 & 1
\end{pmatrix}
 \begin{pmatrix}
a_-(\alpha, \alpha, l) &  0 \\  0  &  \psi_0(\alpha, l) a_-(\alpha, \alpha, l) 
\end{pmatrix}  
\\
& + \frac{\Gamma(\alpha+l)^2 \Gamma(1-\alpha-l)}{\Gamma(\alpha) l!}
e^{-\pi i\alpha} \Delta
 \quad \quad \text{if} \,\,\, l =0,1,2, \ldots,
\\
=& \begin{pmatrix}
1 & 1 \\
b_+(\alpha, \alpha, l) & b_+(\alpha, \alpha, l) 
\end{pmatrix}
 \begin{pmatrix}
a_+(\alpha, \alpha, l) &  0 \\  0  & \psi_0(\alpha, l) a_+(\alpha, \alpha, l) 
\end{pmatrix}  
\\
& + \frac{\Gamma(\alpha) \Gamma(1-\alpha-l)}{ 2\Gamma( 1-l)}
e^{-\pi i\alpha}(I- J) 
 \quad \quad \text{if} \,\,\,  l =-1,-2, -3, \ldots,
\end{align*}
where $\psi_0(\alpha, l) = \psi(\alpha) +\psi(1-\alpha-l)
-2\psi(1) +\pi i.$
\par
$(2)$ Suppose that $\alpha =\beta \not\in \Z$ and $\gamma - 2\alpha=l \in \Z.$ 
For $0<\arg z < \pi,$ $|\arg (z-1)| <\pi,$
\begin{align*}
\tilde{C}_1 =& e^{\pi i\alpha} \begin{pmatrix}
b_-(\alpha, \alpha, l) & b_-(\alpha, \alpha, l) 
\\
1 & 1
\end{pmatrix}
 \begin{pmatrix}
a_-(\alpha, \alpha, l) &  0 \\  0  &  \psi_1(\alpha, l) a_-(\alpha, \alpha, l) 
\end{pmatrix}  
\\
& + \frac{\Gamma(\alpha+l)^2 \Gamma(1-\alpha-l)}{\Gamma(\alpha) l!}
 \Delta
 \quad \quad \quad\text{if} \,\,\, l =0,1,2, \ldots,
\\
=& e^{\pi i\alpha} \begin{pmatrix}
1 & 1 \\
b_+(\alpha, \alpha, l) & b_+(\alpha, \alpha, l) 
\end{pmatrix}
 \begin{pmatrix}
a_+(\alpha, \alpha, l) &  0 \\  0  & \psi_1(\alpha, l) a_+(\alpha, \alpha, l) 
\end{pmatrix}  
\\
& + \frac{\Gamma(\alpha) \Gamma(1-\alpha-l)}{ 2\Gamma( 1-l)}
(I- J) 
 \quad \quad \quad\text{if} \,\,\,  l =-1,-2, -3, \ldots,
\end{align*}
where $\psi_1(\alpha, l) = \psi(\alpha) +\psi(1-\alpha-l)
-2\psi(1) .$
\end{prop}
\section{Proofs of Theorems \ref{thm2.12} through \ref{thm2.15}}\label{sc10}
\subsection{Proofs of Theorems \ref{thm2.12} and \ref{thm2.13}}\label{ssc10.1}
In \eqref{9.1} and Proposition \ref{prop9.2} set $(1-\gamma)/2 =\theta_0/2,$
$(\gamma-\alpha-\beta-1)/2=\theta_x/2,$ $(\beta-\alpha)/2=\sigma/2,$ that is,
$\alpha= -(\sigma +\theta_0+\theta_x)/2,$ $\beta=(\sigma-\theta_0-\theta_x)/2,$
$\gamma=1-\theta_0,$ and $R=\beta(\beta -\gamma +1)/(\alpha-\beta) =
(\theta_0^2-(\sigma-\theta_x)^2)/(4\sigma).$ Then we have 
$T B_0 T^{-1}=\Lambda_0,$
$T B_1 T^{-1}=\Lambda_x,$ $T(\sigma J/2) T^{-1}=\Lambda,$ where $\Lambda_0,$
$\Lambda_x,$ $T$, $\Lambda$ are as in Lemma \ref{lem4.1}, and $Z=T \Psi(\lambda)
$ satisfies
$$
\frac{dZ}{d\lambda} =\Bigl( \frac{\Lambda_0}{\lambda}+\frac{\Lambda_x}
{\lambda-1} \Bigr) Z.
$$
Furthermore, by
$$
Z=\tilde{\rho}^{-\Lambda} Y = (\rho_*^{1/\sigma} )^{-T(\sigma J/2) T^{-1} }
Y = T \rho_* ^{-J/2} T^{-1}Y
$$
this is changed into a system of the form \eqref{8.2}
$$
\frac{dY}{d\lambda} =\Bigl( \frac{\tilde{\Lambda}_0}{\lambda}
+\frac{\tilde{\Lambda}_x}{\lambda-1} \Bigr) Y,
$$
which has a matrix solution such that
\begin{align*}
Y &= T\rho_*^{J/2} T^{-1} Z = T\rho_*^{J/2} \Psi(\lambda) = T \rho_*^{J/2}
(I +O(\lambda^{-1}) )\lambda^{\sigma J/2}
\\
&= T(I+O(\lambda^{-1}))\lambda^{\sigma J/2}\rho_*^{J/2}
=(I+O(\lambda^{-1}) ) \lambda^{\Lambda} T \rho_*^{J/2}
\end{align*}
near $\lambda=\infty.$
Hence the connection matrices $C^{(0)},$ $C^{(x)}$ 
are those for
$$
\tilde{Y} = Y(T\rho_*^{J/2} )^{-1} 
 = T\rho_*^{J/2} \Psi (\lambda) (T\rho_*^{J/2} )^{-1} = (I+O(\lambda^{-1})) \lambda
^{\Lambda}
$$
corresponding to \eqref{8.4}.
By Proposition \ref{prop9.3}, the connection 
matrices for $\Psi(\lambda)$ are
$
 \tilde{C}_{\iota/x} \diag [1,  1/ R]
$
$(\iota =0, x)$
under the supposition $\gamma,$ $\gamma-\alpha-\beta \not\in \Z,$ that is,
$\theta_0,$ $\theta_x \not\in \Z,$ and we may choose $C^{(\iota)} =C_{\iota}
T^{-1}$ in such a way that
\begin{equation}\label{10.1}
C_{\iota} = \rho_*^{I/2}
\tilde{C}_{\iota/x}
\diag[1, 1/ R]
\rho_*^{-J/2}  \quad (\iota =0,x)
\end{equation}
with $\alpha,$ $\beta,$ $\gamma,$ $R$ set as above (cf. Remark \ref{rem8.1}). 
Combining $C$ in Proposition \ref{prop9.1}
with $C^{(0)},$ $C^{(x)}$, from \eqref{8.7} we obtain $M_0,$ $M_x$ in Theorem
\ref{thm2.12}. Theorem \ref{thm2.13} is shown by limiting procedure
$\sigma \to \sigma_0$. In \eqref{10.1}, replacing $\rho_*$ by $\rho$
and letting $\sigma\to \theta_0\pm \theta_x,$ $ \theta_x-\theta_0,$
we obtain $C^*_0$ and $C^*_1$ as in Theorem 2.14 other than $C^*_0$ for
$\sigma_0=\theta_x-\theta_0.$ By Remark \ref{rem8.1},
we may choose $\diag[(\sigma+\theta_0-\theta_x)/2, 
((\sigma+\theta_0-\theta_x)/2)^{-1} ] \tilde{C}_0$ instead of $\tilde{C}_0$.
In \eqref{10.1} with such a matrix, letting $\sigma\to \theta_x-\theta_0$
we derive $C^*_0$ with $\tilde{C}^*_0$ in (iii). 
\subsection{Proof of Theorem \ref{thm2.14}}\label{ssc10.2}
Suppose that $\theta_{\infty}\not=0.$ In \eqref{9.2} and Proposition
\ref{prop9.4} set $1-\gamma =\theta_0,$ $\gamma-2\alpha-1 =\mp\theta_x,$ 
that is, $\alpha=-(\theta_0 \mp \theta_x)/2,$ $\gamma=1-\theta_0.$  
Then, for
$\Lambda_0,$ $\Lambda_x,$ $T$, $\Lambda$ as in Lemma \ref{lem4.2}, we have
$T B_0T^{-1} =\Lambda_0,$ $T B_1T^{-1} =\Lambda_x,$ $T\Delta T^{-1}=\Lambda,$
and, near $\lambda=\infty$,
\begin{align*}
Y&= T\rho^{\Delta} \Psi(\lambda) = T\rho^{\Delta} (I+O(\lambda^{-1}))
\lambda^{\Delta}
\\
& = T (I+O(\lambda^{-1}))\lambda^{\Delta} \rho^{\Delta}
= (I+O(\lambda^{-1}))\lambda^{\Lambda} T \rho^{\Delta}
\end{align*}
solves system \eqref{8.2}. Hence the connection
matrices $C^{(0) \pm},$ $C^{(x) \pm}$ for
$$
\tilde{Y}= Y( T\rho^{\Delta})^{-1}=  T\rho^{\Delta}\Psi(\lambda) 
( T\rho^{\Delta})^{-1} = (I+O(\lambda^{-1}))\lambda^{\Lambda}
$$
are desired ones, which are written as $C_{\iota}^{\pm} T^{-1}$ with
$C^{\pm}_{\iota}= \tilde{C}^{\pm}
_{\iota/x} \rho^{-\Delta} $ $(\iota=0,x)$, where 
$\tilde{C}^{\pm}_{\iota/x}$ are given by Proposition \ref{prop9.5}. 
For $y_{\mathrm{ilog}}(\rho,x)$ and $y^{(l)}_{\mathrm{ilog}}
(\rho,x)$ with $\theta_0^2-\theta_x^2\not=0,$ we restrict to $\alpha=
-(\theta_0+\theta_x)/2$ and replace $\rho$ by $\rho \exp( -2\theta_x 
(\theta_0^2-\theta_x^2)^{-1} )$ according to Proposition \ref{prop5.0}.
In the case
where $\alpha= -(\theta_0 -\theta_x)/2,$ $\tilde{C}^+_1$ should be replaced
by $K^+\tilde{C}^+_1,$ because the local solution of \eqref{9.2}
around $\lambda=1$ has the
form $G_1(I+O(\lambda-1)) (\lambda-1)^{-(\theta_x/2)J}.$
If $\theta_{\infty}=0,$ in the argument above the matrix $T$ is to be replaced 
by those in (2) of Lemma \ref{lem4.2}.
Thus we obtain the monodromy matrices in Theorem \ref{thm2.14}. 
\subsection{Proof of Theorem \ref{thm2.15}}\label{ssc10.3}
The systems corresponding to \eqref{8.1} and \eqref{8.2} are
$$
\frac{d\hat{Y}}{d\lambda} = \frac J 2 \hat{Y}, \qquad
\frac{d\tilde{Y}}{d\lambda} = \Bigl( \frac {\Lambda_0}{\lambda}
-\frac{\Lambda_0}{\lambda-1}\Bigr) \tilde{Y} 
$$
having the matrix solutions $\hat{Y} =\exp ((J/2) \lambda)$ and
\begin{align*}
\tilde{Y} &= T \begin{pmatrix}
\lambda^{\theta_0/2}(\lambda-1)^{-\theta_0/2}  &  0
\\
0  &  \lambda^{- \theta_0/2}(\lambda-1)^{\theta_0/2} 
\end{pmatrix}
T^{-1}    & \quad &\text{if} \,\,\, \theta_0\not= 0,
\\
 &= T \begin{pmatrix}
  1  &  \log (\lambda/(\lambda-1) ) 
\\
  0  &    1 
\end{pmatrix}
T^{-1}    & \quad &\text{if} \,\,\, \theta_0= 0,
\end{align*}
respectively, where $\Lambda_0$ and $T$ are as in Lemma \ref{lem4.3}. 
From these facts Theorem \ref{thm2.15} immediately follows.


\end{document}